%% file: steiner-rigidity-submitted.tex
\documentclass[11pt,a4paper,reqno]{amsart}
\usepackage{vmargin,color}
\usepackage[latin1]{inputenc}
\usepackage{amsmath,amsfonts,amsthm,epsfig,graphicx}
\usepackage{pstricks,pst-plot,multido}
\usepackage{caption}

\def\H{\mathcal{H}}

\def\M{\mathcal{M}}

\def\N{\mathbb N}

\def\Q{\mathbb Q}
\def\R{\mathbb R}

\def\C{\mathbf{C}}

\def\Om{\Omega}
\def\S{\Sigma}

\def\a{\alpha}

\def\g{\gamma}
\def\de{\delta}
\def\e{\varepsilon}

\def\l{\lambda}
\def\s{\sigma}

\def\om{\omega}
\def\vphi{\varphi}

\def\Div{{\rm div}\,}

\def\dist{{\rm dist}}

\def\spt{{\rm spt}}

\def\weak{\rightharpoonup}
\def\toloc{\stackrel{{\rm loc}}{\to}}

\def\ov{\overline}
\def\pa{\partial}

\def\pae{\partial^{{\rm e}}}

\def\p{\mathbf{p}}
\def\q{\mathbf{q}}
\def\C{\mathbf{C}}
\def\D{\mathbf{D}}

\newcommand{\essinf}[1]{\mathop{#1\textrm{-}\mathrm{ess\,inf}}}

\DeclareMathOperator*{\aplim}{ap\,lim}

\newtheorem{theorem}{Theorem}[section]
\newtheorem{remark}{Remark}[section]
\newtheorem{definition}{Definition}[section]
\newtheorem{theoremletter}{Theorem}

\newtheorem{proposition}[theorem]{Proposition}
\newtheorem{lemma}[theorem]{Lemma}
\newtheorem{corollary}[theorem]{Corollary}
\newtheorem{example}[theorem]{Example}

\numberwithin{equation}{section}
\numberwithin{figure}{section}

\begin{document}

\title{Rigidity of equality cases \\ in Steiner's perimeter inequality}

\author{F. Cagnetti}

\address{University of Sussex, Pevensey 2, Department of Mathematics, BN1 9QH, Brighton, United Kingdom}
\email{f.cagnetti@sussex.ac.uk}

\author{M. Colombo}
\address{Scuola Normale Superiore di Pisa, p.za dei Cavalieri 7, I-56126 Pisa, Italy}
\email{maria.colombo@sns.it}

\author{G. De Philippis}
\address{Institute for Applied Mathematics, University of Bonn, Endenicher Allee 60, D-53115 Bonn, Germany}
\email{guido.de.philippis@hcm.uni-bonn.de}

\author{F. Maggi}
\address{Department of Mathematics, The University of Texas at Austin,  2515 Speedway Stop C1200, Austin, Texas 78712-1202, USA}
\email{maggi@math.utexas.edu}

\maketitle

\begin{center}
  {\it Dedicated to Nicola Fusco, for his mentorship}
\end{center}

\begin{abstract}
  {\rm Characterizations results for equality cases and for rigidity of equality cases in Steiner's perimeter inequality are presented. (By rigidity, we mean the situation when all equality cases are vertical translations of the Steiner's symmetral under consideration.) We achieve this through the introduction of a suitable measure-theoretic notion of connectedness and a fine analysis of barycenter functions for sets of finite perimeter having segments as orthogonal sections with respect to an hyperplane.}
\end{abstract}

\tableofcontents

\newpage

\section{Introduction}

\subsection{Overview} Steiner's symmetrization is a classical and powerful tool in the analysis of geometric variational problems. Indeed, while volume is preserved under Steiner's symmetrization, other relevant geometric quantities, like diameter or perimeter, behave monotonically. In particular, Steiner's perimeter inequality asserts the crucial fact that perimeter is decreased by Steiner's symmetrization, a property that, in turn, lies at the heart of a well-known proof of the Euclidean isoperimetric theorem; see \cite{DeGiorgi58ISOP}. In the seminal paper \cite{ChlebikCianchiFuscoAnnals05}, that we briefly review in section \ref{section fuscoannals}, Chleb{\'{\i}}k, Cianchi and Fusco discuss Steiner's inequality in the natural framework of sets of finite perimeter, and provide a sufficient condition for the rigidity of equality cases. By {\it rigidity of equality cases} we mean that situation when equality cases in Steiner's inequality are solely obtained in correspondence of translations of the Steiner's symmetral. Roughly speaking, the sufficient condition for rigidity found in \cite{ChlebikCianchiFuscoAnnals05} amounts in asking that the Steiner's symmetral has ``no vertical boundary'' and ``no vanishing sections''. While simple examples show that rigidity may indeed fail if one of these two assumptions is dropped, it is likewise easy to construct polyhedral Steiner's symmetrals such that rigidity holds true and {\it both} these conditions are violated. In particular, the problem of a geometric characterization of rigidity of equality cases in Steiner's inequality was left open in \cite{ChlebikCianchiFuscoAnnals05}, even in the fundamental case of polyhedra.

In the recent paper \cite{ccdpmGAUSS}, we have fully addressed the rigidity problem in the case of Ehrhard's inequality for Gaussian perimeter. Indeed, we obtain a {\it characterization} of rigidity, rather than a mere sufficient condition for it. A crucial step in proving (and, actually, formulating) this sharp result consists in the introduction of a measure-theoretic notion of connectedness, and, more precisely, of what it means for a Borel set to ``disconnect'' another Borel set; see section \ref{section connectedness} for more details.

In this paper, we aim to exploit these ideas in the study of Steiner's perimeter inequality. In order to achieve this goal we shall need a sharp description of the properties of the barycenter function of a set of finite perimeter having segments as orthogonal sections with respect to an hyperplane (Theorem \ref{thm tauM b delta}). With these ideas and tools at hand, we completely characterize equality cases in Steiner's inequality in terms of properties of their barycenter functions (Theorem \ref{thm characterization with barycenter}). Starting from this result, we obtain a general sufficient condition for rigidity (Theorem \ref{thm sufficient bv}), and we show that, if the slice length function is of special bounded variation with locally finite jump set, then equality cases are necessarily obtained by at most countably many vertical translations of ``chunks'' of the Steiner's symmetral (Theorem \ref{thm mv per v sbv}); see section \ref{section intro baricentro}.

In section \ref{section characterization}, we introduce several {\it characterizations} of rigidity. In Theorem \ref{thm characterization no vertical SUFF} we provide {\it two} geometric characterizations of rigidity under the ``no vertical boundary'' assumption considered in \cite{ChlebikCianchiFuscoAnnals05}. In Theorem \ref{thm characterization poliedri SUPERIOR} we characterize rigidity in the case when the Steiner's symmetral is a generalized polyhedron. (Here, the generalization of the usual notion of polyhedron consists in replacing affine functions over bounded polygons with $W^{1,1}$-functions over sets of finite perimeter and volume.) We then characterize rigidity when the slice length function is of special bounded variation with locally finite jump set, by introducing a condition we call {\it mismatched stairway property} (Theorem \ref{thm characterization sbv}). Finally, in Theorem \ref{thm characterization R2}, we prove two characterizations of rigidity in the planar setting.

By building on the results and methods introduced in this paper, it is of course possible to analyze the rigidity problem for Steiner's perimeter inequalities in higher codimension. Although it would have been natural to discuss these issues in here, the already considerable length and technical complexity of the present paper suggested us to do this in a separate forthcoming paper.

\subsection{Steiner's inequality and the rigidity problem}\label{section fuscoannals} We begin by recalling the definition of Steiner's symmetrization and the main result from \cite{ChlebikCianchiFuscoAnnals05}. In doing so, we shall refer to some concepts from the theory of sets of finite perimeter and functions of bounded variation (that are summarized in section \ref{section sofp}), and we shall fix a minimal set of notation used through the rest of the paper. We decompose $\R^n$, $n\ge 2$, as the Cartesian product $\R^{n-1}\times\R$, denoting by $\p:\R^n\to\R^{n-1}$ and $\q:\R^n\to\R$ the horizontal and vertical projections, so that $x=(\p x,\q x)$, $\p x=(x_1,\dots,x_{n-1})$, $\q x=x_n$ for every $x\in\R^n$. Given a function $v:\R^{n-1}\to[0,\infty)$ we say that a set $E\subset\R^n$ is {\it $v$-distributed} if, denoting by $E_z$ its vertical section with respect to $z\in\R^{n-1}$, that is
\[
E_z=\Big\{t\in\R:(z,t)\in E\Big\}\,,\qquad z\in\R^{n-1}\,,
\]
we have that
\[
v(z)=\H^1(E_z)\,,\qquad \mbox{for $\H^{n-1}$-a.e. $z\in\R^{n-1}$}\,.
\]
(Here, $\H^k(S)$ stands for the $k$-dimensional Hausdorff measure on the Euclidean space containing the set $S$ under consideration.) Among all $v$-distributed sets, we denote by $F[v]$ the (only) one that is symmetric by reflection with respect to $\{\q x=0\}$, and whose vertical sections are segments, that is, we set
\[
F[v]=\Big\{x\in\R^n:|\q x|<\frac{v(\p x)}2\Big\}\,.
\]
If $E$ is a $v$-distributed set, then the set $F[v]$ is the {\it Steiner's symmetral} of $E$, and is usually denoted as $E^s$. (Our notation reflects the fact that, in addressing the structure of equality cases, we are more concerned with properties of $v$, rather than with the properties of a particular $v$-distributed set.) The set $F[v]$ has finite volume if and only if $v\in L^1(\R^{n-1})$, and it is of finite perimeter if and only if $v\in BV(\R^{n-1})$ with $\H^{n-1}(\{v>0\})<\infty$; see Proposition \ref{proposition insieme u1u2}. Denoting by $P(E;A)$ the relative perimeter of $E$ with respect to the Borel set $A\subset\R^n$ (so that, for example, $P(E;A)=\H^{n-1}(A\cap\pa E)$ if $E$ is an open set with Lipschitz boundary in $\R^n$), {\it Steiner's perimeter inequality} implies that, if $E$ is a $v$-distributed set of finite perimeter, then
\begin{equation}\label{steiner inequality}
P(E;G\times\R)\ge P(F[v];G\times\R)\qquad\mbox{for every Borel set $G\subset\R^{n-1}$}\,.
\end{equation}
Inequality \eqref{steiner inequality} was first proved in this generality by De Giorgi \cite{DeGiorgi58ISOP}, in the course of his proof of the Euclidean isoperimetric theorem for sets of finite perimeter. Indeed, an important step in his argument consists in showing that if a set $E$ satisfies \eqref{steiner inequality} with equality, then, for $\H^{n-1}$-a.e. $z\in G$, the vertical section $E_z$ is $\H^1$-equivalent to a segment; see \cite[Chapter 14]{maggiBOOK}. The study of equality cases in Steiner's inequality was then resumed by Chleb{\'{\i}}k, Cianchi and Fusco in \cite{ChlebikCianchiFuscoAnnals05}. We now recall two important results from their paper. The first theorem, that is easily deduced by means of \cite[Theorem 1.1, Proposition 4.2]{ChlebikCianchiFuscoAnnals05}, completes De Giorgi's analysis of necessary conditions for equality, and, in turn, provides a characterization of equality cases whenever $\pa^*E$ has no vertical parts. Given a Borel set $G\subset\R^{n-1}$, we set
\begin{equation}
  \label{MGv}
  \M_G(v)=\Big\{E\subset\R^n:\mbox{$E$ $v$-distributed and $P(E;G\times\R)=P(F[v];G\times\R)$}\Big\}\,,
\end{equation}
to denote the family of sets achieving equality in \eqref{steiner inequality}, and simply set $\M(v)=\M_{\R^{n-1}}(v)$.

\begin{theoremletter}[\cite{ChlebikCianchiFuscoAnnals05}]\label{thm ccf1} Let $v\in BV(\R^{n-1})$ and let $E$ be a $v$-distributed set of finite perimeter. If $E\in\M_G(v)$, then, for $\H^{n-1}$-a.e. $z\in G$, $E_z$ is $\H^1$-equivalent to a segment $(t^-,t^+)$,  with $(z,t^+)$, $(z,t^-)\in\pa^*E$, $\p\nu_E(z,t^+)=\p\nu_E(z,t^-)$, and $\q\nu_E(z,t^+)=-\q\nu_E(z,t^-)$. The converse implication is true provided $\pa^*E$ has no vertical parts above $G$, that is,
  \begin{equation}
    \label{no vertical parts on G}
      \H^{n-1}\Big(\Big\{x\in\pa^*E\cap (G\times \R):\q\nu_{E}(x)=0\Big\}\Big)=0\,,
  \end{equation}
  where, $\pa^*E$ denotes the reduced boundary to $E$, while $\nu_E$ is the measure-theoretic outer unit normal of $E$; see section \ref{section sofp}.
\end{theoremletter}

The second main result from \cite[Theorem 1.3]{ChlebikCianchiFuscoAnnals05} provides a sufficient condition for the rigidity of equality cases in Steiner's inequality over an open connected set. Notice indeed that some assumptions are needed in order to expect rigidity;
\begin{figure}
  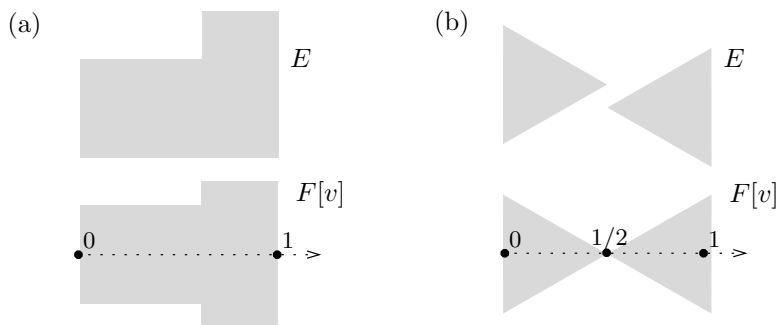\caption{{\small In case (a), $\pa^*F[v]$ has vertical parts over $\Om=(0,1)$ and \eqref{rigidity steiner omega} does not hold true; in case (b), $\pa^*F[v]$ has no vertical parts over $\Om=(0,1)$, but \eqref{v>0 Hn-2} fails (indeed, $0=v^\vee(1/2)=v^\wedge(1/2)$), and, again, \eqref{rigidity steiner omega} does not hold true.}}\label{fig basic}
\end{figure}
see Figure \ref{fig basic}.

\begin{theoremletter}[\cite{ChlebikCianchiFuscoAnnals05}]\label{thm ccf2}
  If $v\in BV(\R^{n-1})$, $\Om\subset\R^{n-1}$ is an open connected set with $\H^{n-1}(\Om)<\infty$, and
  \begin{eqnarray}
    \label{no vertical parts CCF}
    D^sv\llcorner\Om=0\,,&&
    \\\label{v>0 Hn-2}
    v^\wedge>0\,,&&\qquad\mbox{$\H^{n-2}$-a.e. on $\Om$}\,,
  \end{eqnarray}
  then for every $E\in\M_\Om(v)$ we have
  \begin{equation}\label{rigidity steiner omega}
  \H^n\Big(\Big(E\Delta (t\,e_n+F[v])\Big)\cap(\Om\times\R)\Big)=0\,,\qquad\mbox{for some $t\in\R$}\,.
  \end{equation}
\end{theoremletter}

\begin{remark}\label{remark approx limits}
  {\rm Here, $D^sv$ stands for the singular part of the distributional derivative $Dv$ of $v$, while $v^\wedge$ and $v^\vee$ denote the approximate lower and upper limits of $v$ (so that if $v_1=v_2$ a.e. on $\R^{n-1}$, then $v_1^\vee=v_2^\vee$ and $v_1^\wedge=v_2^\wedge$ {\it everywhere} on $\R^{n-1}$). We call $[v]=v^\vee-v^\wedge$ the jump of $v$, and define the approximate discontinuity set of $v$ as $S_v=\{v^\vee>v^\wedge\}=\{[v]>0\}$, so that $S_v$ is countably $\H^{n-2}$-rectifiable, and there exists a Borel vector field $\nu_v:S_v\to S^{n-1}$ such that $D^sv=\nu_v\,[v]\,\H^{n-2}\llcorner S_v+D^cv$, where $D^cv$ stands for the Cantorian part of $Dv$. These concepts are reviewed in sections \ref{section approximate limits} and \ref{section sofp}.}
\end{remark}

\begin{remark}
  {\rm Assumption \eqref{no vertical parts CCF} is clearly equivalent to asking that $v\in W^{1,1}(\Om)$ (so that $v^\wedge=v^\vee$ $\H^{n-2}$-a.e. on $\Om$), and, in turn, it is also equivalent to asking that $\pa^*F[v]$ has no vertical parts above $\Om$, that is, compare with \eqref{no vertical parts on G},
  \begin{equation}
    \label{no vertical parts on Omega}
      \H^{n-1}\Big(\Big\{x\in\pa^*F[v]\cap (\Om\times \R):\q\nu_{F[v]}(x)=0\Big\}\Big)=0\,;
  \end{equation}
  see \cite[Proposition 1.2]{ChlebikCianchiFuscoAnnals05} for a proof.}
\end{remark}

\begin{remark}\label{remark exampless}
  {\rm Although assuming the ``no vertical parts'' \eqref{no vertical parts CCF} and ``no vanishing sections'' \eqref{v>0 Hn-2} conditions appears natural in light of the examples sketched in Figure \ref{fig basic}, it should be noted that these assumptions are far from being necessary for having rigidity. For example, Figure \ref{fig casetta} shows the case of a polyhedron in $\R^3$ such that \eqref{rigidity steiner omega} holds true, but the ``no vertical parts'' condition fails.
  \begin{figure}
  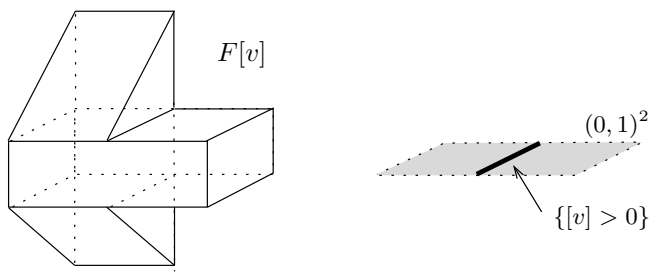\caption{{\small A polyhedron in $\R^3$ such that the rigidity condition \eqref{rigidity steiner omega} holds true (with $\Om=(0,1)^2$) but the ``no vertical parts'' condition fails.}}\label{fig casetta}
  \end{figure}
  Similarly, in Figure \ref{fig salsicciotto}, we have a polyhedron in $\R^3$ such that \eqref{rigidity steiner omega} and \eqref{no vertical parts CCF} hold true, but such that \eqref{v>0 Hn-2} fails.
  \begin{figure}
  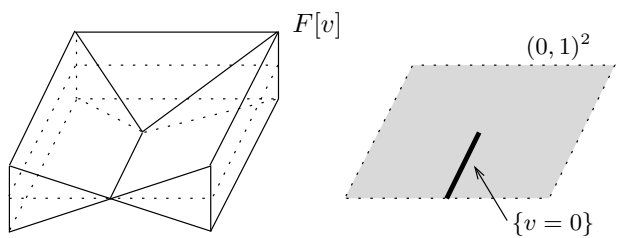\caption{{\small A polyhedron in $\R^3$ such that the rigidity condition \eqref{rigidity steiner omega} and the ``no vertical parts'' condition hold true (with $\Om=(0,1)^2$), but such that the ``no vanishing sections'' condition fails.}}\label{fig salsicciotto}
  \end{figure}}
\end{remark}

\subsection{Essential connectedness}\label{section connectedness} The examples discussed in Figure \ref{fig basic} and Remark \ref{remark exampless} suggest that, in order to characterize rigidity of equality cases in Steiner's inequality, one should first make precise, for example, in which sense a $(n-2)$-dimensional set like $S_v=\{v^\wedge<v^\vee\}$ (contained into the projection of vertical boundaries) may disconnect the $(n-1)$-dimensional set $\{v>0\}$ (that is, the projection of $F[v]$). In the study of rigidity of equality cases for Ehrhard's perimeter inequality, see \cite{ccdpmGAUSS}, we have satisfactorily addressed this kind of question by introducing the following definition.

\begin{definition}\label{definition ess con}
  {\rm Let $K$ and $G$ be Borel sets in $\R^m$. One says that $K$ {\it essentially disconnects} $G$ if there exists a non-trivial Borel partition $\{G_+,G_-\}$ of $G$ modulo $\H^m$ such that
\begin{equation}
  \label{K disconnects G}
  \H^{m-1}\Big(\Big(G^{(1)}\cap\pae  G_+\cap\pae G_-\Big)\setminus K\Big)=0\,;
\end{equation}
conversely, one says that $K$ {\it does not essentially disconnect} $G$ if, for every non-trivial Borel partition $\{G_+,G_-\}$ of $G$ modulo $\H^m$,
\begin{equation}
  \label{K does not disconnect G}
  \H^{m-1}\Big(\Big(G^{(1)}\cap\pae  G_+\cap\pae G_-\Big)\setminus K \Big)>0\,.
\end{equation}
Finally, $G$ is {\it essentially connected} if $\emptyset$ does not essentially disconnect $G$.}
\end{definition}

\begin{remark}
  {\rm By a non-trivial Borel partition $\{G_+,G_-\}$ of $G$ modulo $\H^m$ we mean that
  \[
  \H^m(G_+\cap G_-)=0\,,\qquad  \H^m(G\Delta(G_+\cup G_-))=0\,,\qquad \H^m(G_+)\,\H^m(G_-)>0\,.
  \]
  Moreover, $\pae G$ denotes the essential boundary of $G$, that is defined as
  \[
  \pae G=\R^m\setminus(G^{(0)}\cup G^{(1)})\,,
  \]
  where $G^{(0)}$ and $G^{(1)}$ denote the sets of points of density $0$ and $1$ of $G$; see section \ref{section approximate limits}.}
\end{remark}

\begin{remark}\label{remark essential connected}
  {\rm If $\H^m(G\Delta G')=0$ and $\H^{m-1}(K\Delta K')=0$, then $K$ essentially disconnects $G$ if and only if $K'$ essentially disconnects $G'$;
  \begin{figure}
  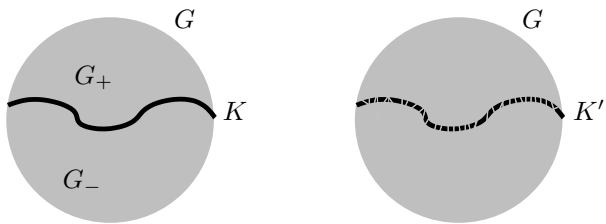\caption{{\small In the picture on the left, $G$ is a disk and $K$ is a smooth curve that divides $G$ in two open regions $G_+$ and $G_-$, in such a way that \eqref{K disconnects G} holds true: thus, $K$ essentially disconnects $G$. Let now $K'$ be obtained by removing some points from $K$. If we remove a set of length zero, that is, if $\H^1(K\setminus K')=0$, then $K'$ still essentially disconnects $G$ (although $G\setminus K'$ may be easily topologically connected!); if, instead, $\H^1(K\setminus K')>0$, then $K'$ does not essentially disconnect $G$, since \eqref{K does not disconnect G} holds true (with $K'$ in place of $K$).}}\label{fig disco}
  \end{figure}
  see Figure \ref{fig disco}.}
\end{remark}

\begin{remark}\label{remark iiii}
  {\rm We refer to \cite[Section 1.5]{ccdpmGAUSS} for more comments on the relation between this definition and the notions of indecomposable currents \cite[4.2.25]{FedererBOOK} and indecomposable sets of finite perimeter \cite[Definition 2.11]{dolzmannmu} or \cite[Section 4]{ambrosiocaselles} used in Geometric Measure Theory. We just recall here that a set of finite perimeter $E$ is said {\it indecomposable} if $P(E)<P(E_+)+P(E_-)$ whenever $\{E_+,E_-\}$ is a non-trivial partition modulo $\H^n$ of $E$ by sets of finite perimeter, and that, in turn, inequality $P(E)<P(E_+)+P(E_-)$ is equivalent to $\H^{n-1}(E^{(1)}\cap\pae E_+\cap\pae E_-)>0$. This measure-theoretic notion of connectedness is compatible with essential connectedness: indeed, as proved in \cite[Remark 2.3]{ccdpmGAUSS}, a set of finite perimeter is indecomposable if and only if it is essentially connected.}
\end{remark}

\subsection{Equality cases and barycenter functions}\label{section intro baricentro} With the notion of essential connectedness at hand we can easily conjecture several possible improvements of Theorem \ref{thm ccf2}. As it turns out, a fine analysis of the barycenter function for sets of finite perimeter with segments as sections is crucial in order to actually prove these results. Given a $v$-distributed set $E$, we define the barycenter function of $E$, $b_E:\R^{n-1}\to\R$, by setting, for every $z\in\R^{n-1}$,
\begin{equation}
  \label{barycenter definition}
b_E(z)=\frac1{v(z)}\int_{E_z}\,t\,d\H^1(t)\,,\qquad\mbox{if $v(z)>0$ and}\hspace{0.1cm} \int_{E_z}\,t\,d\H^1(t)\in\R\,,
\end{equation}
and $b_E(z)=0$ else. In general, $b_E$ may only be a Lebesgue measurable function. When $E$ has segments as sections and finite perimeter, the following theorem provides a degree of regularity for $b_E$ that turns out to be sharp; see Remark \ref{remark fifi}. Notice that the set where $v$ vanishes is critical for the regularity of the barycenter, as implicitly expressed by \eqref{barycenter sharp regularity}.

\begin{theorem}\label{thm tauM b delta}
  If $v\in BV(\R^{n-1})$ and $E$ is a $v$-distributed set of finite perimeter such that $E_z$ is $\H^1$-equivalent to a segment for $\H^{n-1}$-a.e. $z\in\R^{n-1}$, then
  \begin{equation}
    \label{barycenter sharp regularity}
      b_\de=1_{\{v>\de\}}\,b_E \in GBV(\R^{n-1})\,,
  \end{equation}
  for every $\de>0$ such that $\{v>\de\}$ is a set of finite perimeter. Moreover, $b_E$ is approximately differentiable $\H^{n-1}$-a.e. on $\R^{n-1}$, and for every Borel set $G\subset\{v^\wedge>0\}$ the following coarea formula holds,
  \begin{equation}
    \label{coarea-a-ah}
      \int_{\R}\H^{n-2}(G\cap\pae \{b_E>t\})\,dt=\int_G|\nabla b_E|\,d\H^{n-1}+\int_{G\cap S_{b_E}}[b_E]\,d\H^{n-2}+|D^cb_E|^+(G)\,,
  \end{equation}
  where $|D^cb_E|^+$ is the Borel measure on $\R^{n-1}$ defined by
  \begin{equation}
    \label{DcbE}
      |D^cb_E|^+(G)=\lim_{\de\to 0^+}|D^cb_\de|(G)=\sup_{\de>0}|D^cb_\de|(G)\,,\qquad\forall G\subset\R^{n-1}\,.
  \end{equation}
\end{theorem}

\begin{remark}
  {\rm Let us recall that $u\in GBV(\R^{n-1})$ if and only if $\tau_M(u)\in BV_{loc}(\R^{n-1})$ for every $M>0$ (where $\tau_M(s)=\max\{-M,\min\{M,s\}\}$ for $s\in\R$), and that for every $u\in GBV(\R^{n-1})$, we can define a Borel measure $|D^cu|$ on $\R^{n-1}$ by setting
  \begin{equation}
    \label{cantor in gbv}
      |D^cu|(G)=\lim_{M\to\infty}|D^c(\tau_Mu)|(G)=\sup_{M>0}|D^c(\tau_Mu)|(G)\,,
  \end{equation}
  for every Borel set $G\subset\R^{n-1}$. (If $u\in BV(\R^{n-1})$, then the total variation of the Cantorian part of $Du$ agrees with the measure defined in \eqref{cantor in gbv} on every Borel set.) The measures $|D^cb_\de|$ appearing in \eqref{DcbE} are thus defined by means of \eqref{cantor in gbv}, and this makes sense by \eqref{barycenter sharp regularity}. Concerning $|D^cb_E|^+$, we just notice that in case that $b_E\in GBV(\R^{n-1})$, and thus $|D^cb_E|$ is well-defined, then we have
  \[
  |D^cb_E|^+=|D^cb_E|\llcorner\{v^\wedge>0\}\,,\qquad\mbox{on Borel sets of $\R^{n-1}$.}
  \]}
\end{remark}

Starting from Theorem \ref{thm tauM b delta}, we can prove a formula for the perimeter of $E$ in terms of $v$ and $b_E$ (see Corollary \ref{lemma W}), that in turn leads to the following characterization of equality cases in Steiner's inequality in terms of barycenter functions. (We recall that, here and in the following results, the assumption $v\in BV(\R^{n-1};[0,\infty))$ with $\H^{n-1}(\{v>0\})$ is equivalent to asking that $F[v]$ is of finite perimeter, and is thus necessary to make sense of the rigidity problem.)

\begin{theorem}\label{thm characterization with barycenter}
  Let $v\in BV(\R^{n-1};[0,\infty))$ with $\H^{n-1}(\{v>0\})<\infty$, and let $E$ be a $v$-distributed set of finite perimeter. Then, $E\in\M(v)$ if and only if
  \begin{eqnarray}
  \label{finale0}
  \mbox{$E_z$ is $\H^1$-equivalent to a segment}\,,&&\quad\mbox{for $\H^{n-1}$-a.e. $z\in\R^{n-1}$}\,,
  \\
  \label{finale1 E}
  \nabla b_E(z)=0\,,&&\quad\mbox{for $\H^{n-1}$-a.e. $z\in\R^{n-1}$}\,,
  \\\label{finale2 E}
  2[b_E]\le[v]\,,&&\quad\mbox{$\H^{n-2}$-a.e. on $\{v^\wedge>0\}$}\,,
  \\\label{finale3 e}
  D^c(\tau_Mb_\de)(G)=\int_{G\cap \{v>\de\}^{(1)}\cap\{|b_E|<M\}^{(1)}}\hspace{-1.8cm}f\,d\,(D^cv)\,,&&\hspace{0.2cm}
  \begin{array}
    {l l}
    \mbox{for every bounded Borel set $G\subset \R^{n-1}$}
    \\
    \mbox{and for $\H^1$-a.e. $\de>0$ and $M>0$\,,}
  \end{array}\hspace{0.8cm}
  \end{eqnarray}
  where $f:\R^{n-1}\to[-1/2,1/2]$ is a Borel function;
  \begin{figure}
    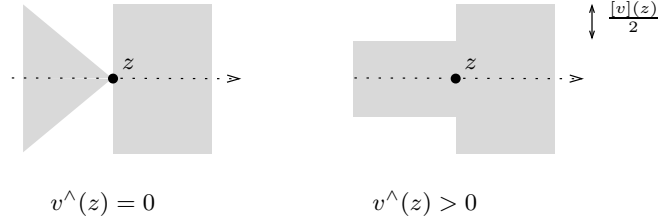\caption{{\small If $E\in\M(v)$, then the jump $[b_E]$ of the barycenter of $E$ can be arbitrarily large on $\{v^\wedge=0\}$, but is necessarily bounded by half the jump of $v$ on $\{v^\wedge>0\}$; see \eqref{finale2 E}. Moreover, the same rule applies to the Cantorian ``jumps'', see \eqref{finale3 e} and \eqref{2DcbE le Dcv}.
    }}\label{fig jumps}
  \end{figure}
  see Figure \ref{fig jumps}. In both cases,
  \begin{eqnarray}\label{2DcbE le Dcv}
    2\,|D^cb_E|^+(G)\le |D^cv|(G)\,,\qquad\mbox{for every Borel set $G\subset\R^{n-1}$}\,,
  \end{eqnarray}
  and, if $K$ is a concentration set for $D^cv$ and $G$ is a Borel subset of $\{v^\wedge>0\}$, then
  \begin{eqnarray}
      \label{coarea-a-ah ottimali}
    \int_{\R}\H^{n-2}(G\cap\pae \{b_E>t\})\,dt=\int_{G\cap S_{b_E}\cap S_v}[b_E]\,d\H^{n-2}+|D^cb_E|^+(G\cap K)\,.
  \end{eqnarray}
\end{theorem}

\begin{remark}
  {\rm By Theorem \ref{thm tauM b delta}, \eqref{finale0} allows us to make sense of $\nabla b_E$, $|D^cb_E|^+$, and $D^c(\tau_Mb_\de)$ (for a.e. $\de>0$), and thus to formulate \eqref{finale1 E}, \eqref{finale3 e}, \eqref{2DcbE le Dcv}, and \eqref{coarea-a-ah ottimali}. In particular, \eqref{coarea-a-ah ottimali} is an immediate consequence of \eqref{coarea-a-ah}, \eqref{finale1 E}, \eqref{finale2 E}, and \eqref{2DcbE le Dcv}.}
\end{remark}

Theorem \ref{thm characterization with barycenter} is a powerful tool in the study of rigidity of equality cases. Indeed, rigidity amounts in asking that $b_E$ is constant on $\{v>0\}$, a condition that, in turn, is equivalent to saying that there exists no subset $I\subset\R$ with $\H^1(I)>0$ such that $\{\{b_E>t\},\{b_E\le t\}\}$ is a non-trivial Borel partition of $\{v>0\}$ modulo $\H^{n-1}$ for every $t\in I$. By combining this information with the coarea formula \eqref{coarea-a-ah ottimali} and with the definition of essential connectedness, we quite easily deduce the following sufficient condition for rigidity.

\begin{theorem}\label{thm sufficient bv}
  If $v\in BV(\R^{n-1};[0,\infty))$, $\H^{n-1}(\{v>0\})<\infty$, and the Cantor part $D^cv$ of $Dv$ is concentrated on a Borel set $K$ such that
  \begin{equation}
    \label{general sufficient condition}
      \mbox{$\{v^\wedge=0\}\cup S_v\cup K$ does not essentially disconnect $\{v>0\}$}\,,
  \end{equation}
  then for every $E\in\M(v)$ there exists $t\in\R$ such that $\H^n(E\Delta(t\,e_n+F[v]))=0$.
\end{theorem}

\begin{remark}
  {\rm The strength of Theorem \ref{thm sufficient bv} is that it provides a sufficient condition for rigidity without a priori structural assumption on $F[v]$. In particular, the theorem admits for non-trivial vertical boundaries and vanishing sections, that are excluded in Theorem \ref{thm ccf2} by \eqref{no vertical parts CCF} and \eqref{v>0 Hn-2}. (In fact, as shown in Appendix \ref{section fusco}, Theorem \ref{thm ccf2} can be deduced from Theorem \ref{thm sufficient bv}.) We also notice that condition \eqref{general sufficient condition} is clearly not necessary for rigidity as soon as vertical boundaries are present; see Figure \ref{fig casetta}.}
\end{remark}

A natural question about equality cases of Steiner's inequality that is left open by Theorem \ref{thm characterization with barycenter} is describing the situation when every $E\in\M(v)$ is obtained by at most countably many vertical translations of parts of $F[v]$. In other words, we want to understand when to expect every $E\in\M(v)$ to satisfy
  \begin{eqnarray}
    \label{E traslato a pezzi}
    &&E=_{\H^n}\bigcup_{h\in I}\Big(c_h\,e_n+(F[v]\cap(G_h\times\R))\Big)
    \\\nonumber
    &&\hspace{1cm}\mbox{where $I$ is at most countable, $\{c_h\}_{h\in I}\subset\R$, and}
    \\\nonumber
    &&\hspace{1cm}\mbox{$\{G_h\}_{h\in I}$ is a Borel partition modulo $\H^{n-1}$ of $\{v>0\}$}\,.
  \end{eqnarray}
The following theorem shows that this happens when $v$ is of special bounded variation with locally finite jump set. The notion of $v$-admissible partition of $\{v>0\}$ used in the theorem is introduced in Definition \ref{v-admissible}, see section \ref{section characterization}.

\begin{theorem}
  \label{thm mv per v sbv}
  Let $v\in SBV(\R^{n-1};[0,\infty))$ with $\H^{n-1}(\{v>0\})<\infty$, and
  \begin{equation}
    \label{Sv localmente finito}
    \mbox{$S_v\cap\{v^\wedge>0\}$ is locally $\H^{n-2}$-finite}\,,
  \end{equation}
  and let $E$ be a $v$-distributed set of finite perimeter. Then, $E\in\M(v)$ if and only if $E$ satisfies \eqref{E traslato a pezzi} for a $v$-admissible partition $\{G_h\}_{h\in I}$ of $\{ v > 0\}$, and $2[b_E]\le [v]$ $\H^{n-2}$-a.e. on $\{v^\wedge>0\}$. Moreover, in both cases, $|D^cb_E|^+=0$.
\end{theorem}

\begin{remark}\label{remark pino}
  {\rm Let us recall that, by definition, $v\in SBV(\R^{n-1})$ if $v\in BV(\R^{n-1})$ and $D^cv=0$. The approximate discontinuity set $S_v$ of a generic $v\in SBV(\R^{n-1})$ is always countably $\H^{n-2}$-rectifiable, but it may fail to be locally $\H^{n-2}$-finite. if $v\in SBV(\R^{n-1})$ but \eqref{Sv localmente finito} fails, then it may happen that \eqref{E traslato a pezzi} does not hold true for some $E\in \M(v)$; see Remark \ref{remark razionali} below.}
\end{remark}

\begin{remark}\label{remark checca}
  {\rm Condition \eqref{E traslato a pezzi} can be reformulated in terms of a property of the barycenter function. Indeed, \eqref{E traslato a pezzi} is equivalent to asking that
  \begin{equation}
  \label{bE traslato a pezzi}
  b_E=\sum_{h\in I}c_h\,1_{G_h}\,,\qquad\mbox{$\H^{n-1}$-a.e. on $\R^{n-1}$}\,,
  \end{equation}
  for $I$, $\{c_h\}_{h\in I}$ and $\{G_h\}_{h\in I}$ as in \eqref{E traslato a pezzi}. It should be noted that,
  {\it if no additional conditions are assumed on the partition $\{ G_h \}_{h \in I}$}, then
   \eqref{bE traslato a pezzi} is {\it not} equivalent to saying that $b_E$ has ``countable range''. An example is obtained as follows. Let $K$ be the middle-third Cantor set in $[0,1]$, let $\{G_h\}_{h\in\N}$ be the disjoint family of open intervals such that $K=[0,1]\setminus\bigcup_{h\in\N}G_h$, and let $\{c_h\}_{h\in\N}\subset\R$ be such that the Cantor function $u_K$ satisfies $u_K=c_h$ on $G_h$. In this way, $u_K=\sum_{h\in\N}c_h\,1_{G_h}$ on $[0,1]\setminus K$, thus, $\H^1$-a.e. on $[0,1]$. Of course, since $u_K$ is a non-constant, continuous, and increasing function, it does not have  ``countable range'' in any reasonable sense. At the same time, if we set $v(z)=1_{[0,1]}(z)\,\dist(z,K)$ for $z\in\R$, then $v$ is a Lipschitz function on $\R$ (thus it satisfies all the assumptions in Theorem \ref{thm mv per v sbv}) and the set
  \[
  E=\Big\{x\in \R^2:u_K(\p x)-\frac{v(\p x)}2<\q x<u_K(\p x)+\frac{v(\p x)}2\Big\}\,,
  \]
  is such that $E\in\M(v)$, as one can check by Corollary \ref{lemma W} and Corollary \ref{lemma W2} in section \ref{section barycenter}. We also notice that, in this example, $|D^cb_E|\llcorner\{v^\wedge=0\}\ne 0$, while $|D^cb_E|^+=0$.}
\end{remark}

\begin{remark}\label{remark razionali}
{\rm We now provide the example introduced in Remark \ref{remark pino}. Given $\{q_h\}_{h\in\N}=\Q\cap[0,1]$ and $\{\a_h\}_{h\in\N}\in(0,\infty)$ such that $\sum_{h\in\N}\a_h<\infty$, we can define $v\in SBV(\R)$ such that $\H^1(\{v>0\})=1$ and $Dv=D^sv=D^jv$, by setting
  \[
  v(t)=\sum_{\{h\in\N:q_h<t\le 1\}}\a_h=\sum_{h\in\N}\a_h\,1_{(q_h,1]}(t)\,,\qquad t\in\R\,.
  \]
  If we plug $v_1=0$, $v_2=v$, and, say, $\l=0$, in Proposition \ref{proposition dsv} below, then we obtain a set $E\in\M(v)$. At the same time, \eqref{bE traslato a pezzi}, thus \eqref{E traslato a pezzi}, cannot hold true, as $b_E=v/2$ $\H^1$-a.e. on $\R$ and $v$ is {\it strictly} increasing  on $[0,1]$. (The requirement that the sets $G_h$ in \eqref{bE traslato a pezzi} are mutually disjoint modulo $\H^{n-1}$ plays of course a crucial role in here.) Notice that, as expected, $S_v\cap\{v^\wedge>0\}=\Q\cap[0,1]$ is not locally $\H^0$-finite.}
\end{remark}

We close our analysis of equality cases with the following proposition, that shows a general way of producing equality cases in Steiner's inequality that (potentially) do not satisfy the basic structure condition \eqref{E traslato a pezzi}.

\begin{proposition}\label{proposition dsv}
  If $v=v_1+v_2$ where $v_1,v_2\in BV(\R^{n-1};[0,\infty))$, $Dv_1=D^av_1$, $v_2$ is not constant (modulo $\H^{n-1}$) on $\{v>0\}$, $Dv_2=D^sv_2$, and $0<\H^{n-1}(\{v>0\})<\infty$, then rigidity fails for $v$. Indeed, if we set
  \begin{equation}
    \label{malefico}
      E=\Big\{x\in\R^n: -\l\,v_2(\p x)-\frac{v_1(\p x)}2\le \q x\le \frac{v_1(\p x)}2+(1-\l)\,v_2(\p x)\Big\}\,,
  \end{equation}
  for $\l\in[0,1]\setminus\{1/2\}$, then $E\in\M(v)$ but $\H^n(E\Delta(t\,e_n+F[v]))>0$ for every $t\in\R$.
\end{proposition}

\subsection{Characterizations of rigidity}\label{section characterization} We now start to discuss the problem of characterizing rigidity of equality cases. We shall analyze this question under different geometric assumptions on the considered Steiner's symmetral, and see how different structural assumptions lead to formulate different characterizations.

We begin our analysis by working under the assumption that no vertical boundaries are present where the slice length function $v$ is essentially positive, that is, on $\{v^\wedge>0\}$. It turns out that, in this case, the sufficient condition \eqref{general sufficient condition} can be weakened to
\begin{equation}
  \label{sufficient condition v=0}
  \mbox{$\{v^\wedge=0\}$ does not essentially disconnect $\{v>0\}$}\,,
\end{equation}
and that, in turn, this same condition is also necessary to rigidity. Moreover, an alternative characterization can obtained by merely requiring that $F[v]$ is indecomposable.

\begin{theorem}\label{thm characterization no vertical SUFF}
  If $v\in BV(\R^{n-1};[0,\infty))$ with $\H^{n-1}(\{v>0\})<\infty$, and
  \begin{eqnarray}\label{no vertical boundary cool}
    D^sv\llcorner \{v^\wedge>0\}=0\,,
  \end{eqnarray}
  then the following three statements are equivalent:
  \begin{enumerate}
    \item[(i)] if $E\in\M(v)$ then $\H^n(E\Delta(t\,e_n+F[v]))=0$ for some $t\in\R$;
    \item[(ii)] $\{v^\wedge=0\}$ does not essentially disconnect $\{v>0\}$;
    \item[(iii)] $F[v]$ is indecomposable.
  \end{enumerate}
\end{theorem}

\begin{remark}
  {\rm Notice that condition \eqref{no vertical boundary cool} does not prevent $\pa^*F[v]$ to contain vertical parts, provided they are concentrated where the lower approximate limit of $v$ vanishes. Indeed, it implies that $D^cv=0$ (see step one in the proof of Theorem \ref{thm characterization no vertical SUFF} in section \ref{section proof of charact W11}), and that $S_v$ is contained into $\{v^\wedge=0\}$ modulo $\H^{n-2}$.}
\end{remark}

\begin{remark}
  {\rm We notice that the equivalence between conditions (ii) and (iii) is actually true whenever $v\in BV(\R^{n-1};[0,\infty))$ with $\H^{n-1}(\{v>0\})<\infty$; in other words, \eqref{no vertical boundary cool} plays no role in proving this equivalence. This is proved in Theorem \ref{thm indecomponibili}, section \ref{section F indecomposable}.}
\end{remark}

The situation becomes much more complex when we consider the possibility for $\pa^*F[v]$ to have vertical parts above $\{v^\wedge>0\}$. As already noticed, simple polyhedral examples, like the one depicted in Figure \ref{fig casetta}, show that condition \eqref{general sufficient condition} is not even a viable candidate as a characterization of rigidity in this case. We shall begin our discussion of this problem by completely solving it in the case of polyhedra and, in fact, in the much broader class of sets introduced in the next definition.

\begin{definition}
  {\rm Let $v:\R^{n-1}\to[0,\infty)$. We say that $F[v]$ is a {\it generalized polyhedron} if there exists a {\it finite} disjoint family of indecomposable sets of finite perimeter and volume $\{A_j\}_{j\in J}$ in $\R^{n-1}$, and a family of functions $\{v_j\}_{j\in J}\subset W^{1,1}(\R^{n-1})$, such that
  \begin{eqnarray}\label{gp1}
  &&\hspace{1.6cm}v=\sum_{j\in J}v_j\,1_{A_j}\,,
  \\\label{gp2}
  &&\Big(\{v^\wedge=0\}\setminus\{v=0\}^{(1)}\Big)\cup S_v\subset_{\H^{n-2}}\bigcup_{j\in J}\pae A_j\,.
  \end{eqnarray}
  (Here and in the following, $A\subset_{\H^k}B$ stands for $\H^k(A\setminus B)=0$.)}
\end{definition}

\begin{remark}
  {\rm Condition \eqref{gp2} amounts in asking that $v$ can jump or essentially vanish on $\{v>0\}$ only inside the essential boundaries of the sets $A_j$. For example, if $\{A_j\}_{j\in J}$ is a finite disjoint family of bounded open sets with Lipschitz boundary in $\R^{n-1}$, $\{v_j\}_{j\in J}\subset C^1(\R^{n-1})$, and $v_j>0$ on $A_j$ for every $j\in J$, then $v=\sum_{j\in J}v_j\,1_{A_j}$ defines a generalized polyhedron $F[v]$. Notice that \eqref{gp2} holds true since $v$ can jump only over the boundaries of the $A_j$, so that $S_v\subset\bigcup_{j\in J}\pa A_j$, while $\{v_j=0\}\cap\ov{A_j}\subset \pa A_j$ for every $j\in J$.}
\end{remark}

\begin{remark}\label{remark poliedri}
  {\rm Clearly, if $F[v]$ is a generalized polyhedron, then $v\in SBV(\R^{n-1})$ with $\H^{n-1}(\{v>0\})<\infty$, so that $F[v]$ has necessarily finite perimeter and volume.}
\end{remark}

\begin{theorem}\label{thm characterization poliedri SUPERIOR}
  If $v:\R^{n-1}\to[0,\infty)$ is such that $F[v]$ is a generalized polyhedron, then the following two statements are equivalent:
  \begin{enumerate}
    \item[(i)] if $E\in\M(v)$ then $\H^n(E\Delta(t\,e_n+F[v]))=0$ for some $t\in\R$;
    \item[(ii)] for every $\e>0$ the set $\{v^\wedge=0\}\cup\{[v]>\e\}$ does not essentially disconnect $\{v>0\}$.
  \end{enumerate}
\end{theorem}

\begin{remark}
  {\rm In the example depicted in Figure \ref{fig casetta} the set $\{v^\wedge=0\}\cap\{v>0\}^{(1)}$ is empty, the set $\{[v]>0\}$ essentially disconnects $\{v>0\}$, but there is no $\e>0$ such that $\{[v]>\e\}$ essentially disconnects $\{v>0\}$. Indeed, in this case, rigidity holds true.}
\end{remark}

As noticed in Remark \ref{remark poliedri}, if $F[v]$ is a generalized polyhedron, then $v\in SBV(\R^{n-1})$ with $S_v$ locally $\H^{n-2}$-rectifiable, so that  $v$ satisfies the assumptions of Theorem \ref{thm mv per v sbv}. We now discuss the rigidity problem in this more general situation.

We start by noticing that, as shown by Example \ref{remark maria} below, condition (ii) in Theorem \ref{thm characterization poliedri SUPERIOR} is not even a sufficient condition to rigidity under the assumptions on $v$ considered in Theorem \ref{thm mv per v sbv}. A key remark here is that, in the situations considered in Theorem \ref{thm characterization no vertical SUFF} and Theorem \ref{thm characterization poliedri SUPERIOR}, when rigidity fails we can make this happen by performing a vertical translation of $F[v]$ above a single part of $\{v>0\}$. For example, when condition (ii) in Theorem \ref{thm characterization poliedri SUPERIOR} fails, there exist $\e>0$ and a non-trivial Borel partition $\{G_+,G_-\}$ of $\{v>0\}$ modulo $\H^{n-1}$ such that
\[
\{v>0\}^{(1)}\cap\pae G_+\cap\pae G_-\subset_{\H^{n-2}}\{v^\wedge=0\}\cup\{[v]>\e\}\,.
\]
Correspondingly, as we shall prove later on, the $v$-distributed set $E(t)$ defined as
\[
E(t)=\Big((t\,e_n+F[v])\cap(G_+\times\R)\Big)\cup\Big(F[v]\cap(G_-\times\R)\Big)\,,\qquad t\in\R\,,
\]
and obtained by a single vertical translation of $F[v]$ above $G_+$, satisfies $P(E(t))=P(F[v])$ whenever $t\in(0,\e/2)$. (Moreover, when condition \eqref{sufficient condition v=0} fails, we have $E(t)\in\M(v)$  for every $t\in\R$.) However, there may be situations in which violating rigidity by a single vertical translation of $F[v]$ is impossible, but where this task can be accomplished by simultaneously performing countably many independent vertical translations of $F[v]$. An example is obtained as follows.

\begin{example}\label{remark maria}
  {\rm We construct a function $v:\R^2\to[0,\infty)$ in such a way that $v\in SBV(\R^2)$, $S_v$ is locally $\H^1$-rectifiable, the set $\{v^\wedge=0\}\cup\{[v]>\e\}$ does not essentially disconnect $\{v>0\}$ for any $\e>0$, but, nevertheless, rigidity fails. Given $t\in\R$ and $\ell>0$, denote by $Q(t,\ell)$ the open square in $\R^2$ with center at $(t,0)$, sides parallel to the direction $(1,1)$ and $(1,-1)$, and diagonal of length  $2\,\ell$.
  \begin{figure}
    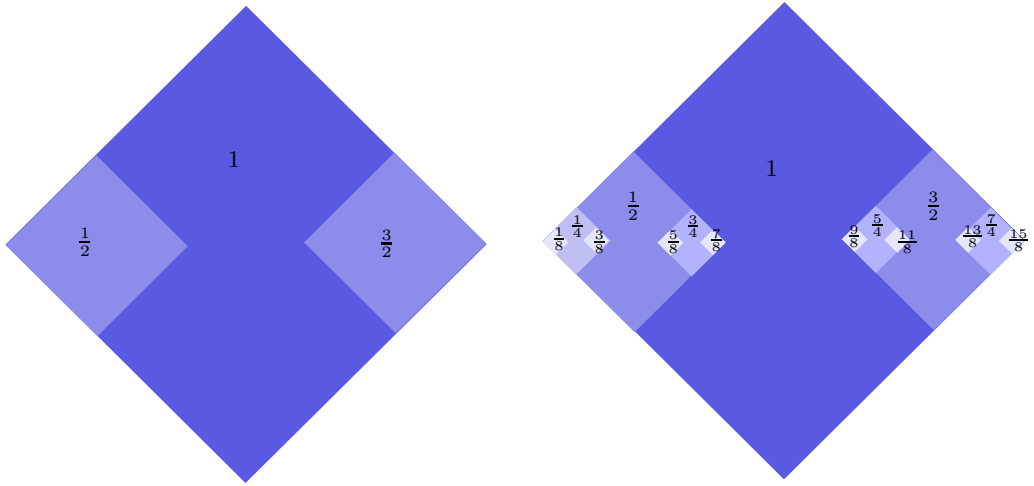\caption{{\small The functions $u_2$ and $u_4$ in the construction of Example \ref{remark maria}.}}\label{fig mariasuper}
  \end{figure}
  Then we set $u_1=1_{Q(0,1)}$, and define a sequence $\{u_j\}_{j\in\N}$ of piecewise constant functions,
  \begin{eqnarray*}
    u_2&=&u_1-\frac{1_{Q(-3/4,1/4)}}2+\frac{1_{Q(3/4,1/4)}}2\,,
    \\
    u_3&=&u_2-\frac{1_{Q(-15/16,1/16)}}4+\frac{1_{Q(-9/16,1/16)}}4-\frac{1_{Q(9/16,1/16)}}4+\frac{1_{Q(15/16,1/16)}}4\,,
  \end{eqnarray*}
  etc.; see Figure~\ref{fig mariasuper}. This sequence has pointwise limit $v\in SBV(\R^2;[0,\infty))$ such that $\{v>0\}=Q(0,1)$ and $Dv=D^sv$. In particular, if we define $E$ as in \eqref{malefico} with $\l=0$, $v_1=0$, and $v_2=v$, then, by Proposition \ref{proposition dsv}, $E\in\M(v)$. Since $b_E=v/2$, we easily see that \eqref{bE traslato a pezzi}, thus \eqref{E traslato a pezzi}, holds true: in other words, $E$ is obtained by countably many vertical translations of $F[v]$ over suitable disjoint Borel sets $G_h$, $h\in\N$. At the same time, any set $E_0$ obtained by a vertical translation of $F[v]$ over one (or over finitely many) of the $G_h$'s is bound to violate the necessary condition for equality $2[b_{E_0}]\le [v]$ $\H^{n-2}$-a.e. on $S_v\cap\{v^\wedge>0\}$, as the infimum of $[v]$ on $\pae G_h\cap S_v\cap\{v^\wedge>0\}$ is zero for every $h\in\N$. We also notice that,  as a simple computation shows, $S_v\cap\{v^\wedge>0\}$ is not only countably $\H^1$-rectifiable in $\R^2$, but actually $\H^1$-finite (thus, it is locally $\H^1$-rectifiable).}
\end{example}

All the above considerations finally suggest to introduce the following condition, that, in turn, characterizes rigidity under the assumptions on $v$ considered in Theorem \ref{thm mv per v sbv}. We begin by recalling the definition of Caccioppoli partition.

\begin{definition}\label{definition caccioppoli}
{\rm Let $G\subset \R^{n-1}$ be a set of finite perimeter, and let $\{ G_h \}_{h \in I}$ be an at most countable Borel partition of $G$ modulo $\H^{n-1}$. (That is, $I$ is a finite or countable set with $\#\,I\ge 2$, $G=_{\H^{n-1}}\bigcup_{h\in I}G_h$, $\H^{n-1}(G_h)>0$ for every $h\in I$ and $\H^{n-1}(G_h\cap G_k)=0$ for every $h,k\in I$, $h\ne k$.) We say that $\{ G_h \}_{h \in I}$ is a {\it Caccioppoli partition of $G$}, if $\sum_{h \in I} P (G_h) < \infty$.}
\end{definition}

\begin{remark}
{\rm When $G$ is an open set and $\{ G_h \}_{h \in I}$ is an at most countable Borel partition of $G$ modulo $\H^{n-1}$, then, according to \cite[Definition 4.16]{AFP}, $\{G_h\}_{h\in I}$ is a Caccioppoli partition of $G$ if $\sum_{h\in I}P(G_h;G)<\infty$. Of course, if we assume in addition that $G$ is of finite perimeter, then $\sum_{h\in I}P(G_h;G)<\infty$ is equivalent to $\sum_{h\in I}P(G_h)<\infty$. Thus Definition \ref{definition caccioppoli} and \cite[Definition 4.16]{AFP} agree in their common domain of applicability (that is, on open sets of finite perimeter).}
\end{remark}

\begin{definition} \label{v-admissible}
{\rm Let $v\in BV(\R^{n-1};[0,\infty))$, and let $\{ G_h \}_{h \in I}$ be an at most countable Borel partition of $\{ v > 0 \}$.
We say that $\{ G_h \}_{h \in I}$ is a {\it $v$-admissible partition} of $\{ v > 0 \}$, if $\{ G_h \cap B_R\cap \{ v > \delta \} \}_{h \in I}$ is a Caccioppoli partition of $\{ v > \delta \}\cap B_R$, for every $\delta > 0$ such that $\{ v > \delta \}$ is of finite perimeter and for every $R>0$.}
\end{definition}

\begin{definition}\label{definition ctrp}
  {\rm One says that $v\in BV(\R^{n-1};[0,\infty))$ satisfies the {\it mismatched stairway property} if the following holds:
  If $\{G_h\}_{h\in I}$ is a $v$-admissible partition of $\{ v > 0 \}$ and if $\{c_h\}_{h\in I}\subset\R$ is a sequence with $c_h\ne c_k$ whenever $h\ne k$,  then there exist $h_0,k_0\in I$ with $h_0\ne k_0$, and a Borel set $\S$ with
  \begin{equation}\label{ctrp sigma}
  \S\subset\pae G_{h_0}\cap\pae G_{k_0}\cap\{v^\wedge>0\}\,,\qquad \H^{n-2}(\S)>0\,,
  \end{equation}
  such that
  \begin{equation}
  \label{ctrp v}
  [v](z)<2|c_{h_0}-c_{k_0}|\,,\qquad\forall z\in\S\,.
  \end{equation}}
\end{definition}

\begin{remark}
  {\rm The terminology adopted here would like to suggest the following idea. One considers a $v$-admissible partition $\{G_h\}_{h\in I}$ of $\{v>0\}$ such that $\{v>0\}^{(1)}\cap\bigcup_{h\in I}\pae G_h$ is contained into $\{v^\wedge=0\}\cup S_v$. Next, one modifies $F[v]$ by performing vertical translations $c_h$ above each $G_h$, thus constructing a new  set $E$ having a ``stairway-like'' barycenter. This new set will have the same perimeter of $F[v]$, and thus will violate rigidity if $\#\,I\ge 2$, provided {\it all} the steps of the stairway match the jumps of $v$, in the sense that $2[b_E]=2|c_h-c_k|\le[v]$ on each $\pae G_h\cap\pae G_k\cap\{v^\wedge>0\}$. Thus, when all equality cases are stairway-like, we expect rigidity to be equivalent to asking that every such stairway has {\it at least  one} step that is {\it  mismatched} with respect to $[v]$.}
\end{remark}

\begin{remark}\label{remark epsilon}
  {\rm If $v\in BV(\R^{n-1};[0,\infty))$ has the mismatched stairway property, then, for every $\e>0$, $\{v^\wedge=0\}\cup \{[v]>\e\}$ does not essentially disconnect $\{v>0\}$. In particular, $\{v^\wedge=0\}$ does not essentially disconnect $\{v>0\}$, $\{v>0\}$ is essentially connected, and although it may still happen that $\{v^\wedge=0\}\cup S_v$ essentially disconnects $\{v>0\}$, in this case one has
  \[
  \essinf{\H^{n-2}}_{S_v\cap\{v^\wedge>0\}}\,[v]=0\,.
  \]
  We prove the claim arguing by contradiction. If $\{v^\wedge=0\}\cup \{[v]>\e\}$ essentially disconnects $\{v>0\}$, then there exist $\e>0$ and a non-trivial Borel partition $\{G_+,G_-\}$ of $\{v>0\}$ modulo $\H^{n-1}$ such that $\{v>0\}^{(1)}\cap\pae G_+\cap\pae G_-\subset_{\H^{n-2}}\{v^\wedge=0\}\cup\{[v]>\e\}$. Since \eqref{dens3} below implies $\{v^\wedge>0\}\subset\{v>0\}^{(1)}$, then we have
  \begin{equation}\label{neil young1}
      \{v^\wedge>0\}\cap\pae G_+\cap\pae G_-\subset_{\H^{n-2}}\{[v]>\e\}\,,
  \end{equation}
  so that, for every $\de>0$,
  \begin{eqnarray}
  \{ v > \delta \}^{(1)} \cap \pae G_{+}= \{ v > \delta \}^{(1)} \cap \pae G_{+} \cap \pae G_{-}
  \label{agosto}
  \subset_{\H^{n-2}}\{[v]>\e\}\,.
  \end{eqnarray}
  If we set $G_{\pm\, \delta} = G_\pm \cap \{ v > \delta\}$, then $\pae G_{\pm\, \delta} \subset \pae \{ v > \delta \} \cup (\{ v > \delta \}^{(1)} \cap \pae G_{\pm})$, and, by \eqref{agosto}, $\pae G_{\pm\,\de}\subset_{\H^{n-2}}\pae\{v>\de\}\cup\{[v]>\e\}$. Since $[v]\in L^1(\H^{n-2}\llcorner S_v)$, we find $\H^{n-2}(\{[v]>t\})<\infty$ for every $t>0$, and, in particular
  \[
  P(G_{+\,\de})+P(G_{-\,\de})\le 2\,P(\{v>\de\})+2\,\H^{n-2}(\{[v]>\e\})<\infty\,,
  \]
  whenever $\{v>\de\}$ is of finite perimeter. This shows that $\{G_+,G_-\}$ is a $v$-admissible partition. If we now set $I=\{+,-\}$, $c_+=\e/2$, and $c_-=0$, then $I$, $\{G_h\}_{h\in I}$, and $\{c_h\}_{h\in I}$ are admissible in the mismatched stairway property. By the mismatched stairway property, there exists a Borel set $\S\subset\{v^\wedge>0\}\cap\pae G_+\cap\pae G_-$ such that $[v]< 2|c_+-c_-|=\e$ on $\S$ and $\H^{n-2}(\S)>0$, a contradiction to \eqref{neil young1}.}
\end{remark}

It turns out that if $v$ is a $SBV$-function with locally finite jump set, then rigidity is characterized by the mismatched stairway property.

\begin{theorem}
\label{thm characterization sbv}
  If $v\in SBV(\R^{n-1};[0,\infty))$, $\H^{n-1}(\{v>0\})<\infty$, and $S_v\cap\{v^\wedge>0\}$ is locally $\H^{n-2}$-finite,
  then the following two statements are equivalent:
    \begin{enumerate}
    \item[(i)] if $E\in\M(v)$, then $\H^n(E\Delta(t\,e_n+F[v]))=0$ for some $t\in\R$;
    \item[(ii)] $v$ has the mismatched stairway property.
      \end{enumerate}
  \end{theorem}

  \begin{remark}
  {\rm Is it important to observe that, in order to characterize rigidity, only $v$-admissible partitions of $\{ v > 0\}$ have to be considered in the definition of mismatched stairway property. Indeed, let $n=2$ and set $v=1_{(0,1)}\in SBV(\R;[0,\infty))$, so that rigidity holds true for $v$. Let now $\{ G_h \}_{h \in\N}$ be the family of open intervals used to define the middle-third Cantor set $K$, so that $K = [0,1] \setminus\bigcup_{h \in\N} G_h$. Notice that $\{ G_h \}_{h \in\N}$ is a non-trivial countable Borel partition of $\{v>0\}=(0,1)$ modulo $\H^1$. However, since $\pae G_h \cap \pae G_k = \emptyset$ whenever $h \neq k$, it is not possible to find a set $\Sigma$ satisfying \eqref{ctrp sigma} whatever the choice of $\{c_h\}_{h\in\N}$ we make. In particular, if we would not restrict the partitions in Definition \ref{definition ctrp} to $v$-admissible partitions, then this particular $v$ (satisfying rigidity) would not have the mismatched stairway property. Notice of course that, in this example, $\sum_{h\in\N}P(G_h\cap\{v>\de\}\cap B_R)=\infty$ for every $\de,R>0$.}
\end{remark}

The question for a geometric characterization of rigidity when $v\in BV$ is thus left open. The considerable complexity of the mismatched stairway property may be seen as a negative indication about the tractability of this problem. In the planar case, due to the trivial topology of the real line, these difficulties can be overcome, and we obtain the following complete result.

\begin{theorem}\label{thm characterization R2}
  If $v\in BV(\R;[0,\infty))$ and $\H^1(\{v>0\})<\infty$, then, equivalently,
  \begin{enumerate}
    \item[(i)] if $E\in\M(v)$, then $\H^2(E\Delta(t\,e_2+F[v]))=0$ for some $t\in\R$;
    \item[(ii)] $\{v>0\}$ is $\H^1$-equivalent to a bounded open interval $(a,b)$, $v\in W^{1,1}(a,b)$, and $v^\wedge>0$ on $(a,b)$;
    \item[(iii)] $F[v]$ is an indecomposable set that has no vertical boundary above $\{v^\wedge>0\}$, i.e.
    \begin{equation}
      \label{paura}
          \H^1\Big(\Big\{x\in\pa^*F[v]:\q\nu_{F[v]}(x)=0\,,v^\wedge(\p x)>0\Big\}\Big)=0\,.
    \end{equation}
\end{enumerate}
\end{theorem}

We close this introduction by mentioning that the extension of our results to the case of the localized Steiner's inequality is discussed in appendix \ref{section fusco}. In particular, we shall explain how to derive Theorem \ref{thm ccf2} from Theorem \ref{thm sufficient bv} via an approximation argument.

\medskip

\noindent {\bf Acknowledgement}\,: This work was carried out while FC, MC, and GDP were visiting the University of Texas at Austin. The work of FC was partially supported by the UT Austin-Portugal partnership through the FCT post-doctoral fellowship SFRH/BPD/51349/2011. The work of GDP was partially supported by ERC under FP7, Advanced Grant n. 246923. The work of FM was partially supported by ERC under FP7, Starting Grant n. 258685 and Advanced Grant n. 226234, by the Institute for Computational Engineering and Sciences and by the Mathematics Department of the University of Texas at Austin during his visit to these institutions, and by NSF Grant DMS-1265910.

\section{Notions from Geometric Measure Theory}\label{section preliminaries} We gather here some notions from Geometric Measure Theory needed in the sequel, referring to \cite{AFP,maggiBOOK} for further details. We start by reviewing our general notation in $\R^n$. We denote by $B(x,r)$ the open Euclidean ball of radius $r>0$ and center $x\in\R^n$. Given $x\in\R^n$ and $\nu\in S^{n-1}$ we denote by $H_{x,\nu}^+$ and $H_{x,\nu}^-$ the complementary half-spaces
\begin{eqnarray}\label{Hxnu+}
  H_{x,\nu}^+&=&\Big\{y\in\R^n:(y-x)\cdot\nu\ge 0\Big\}\,,
  \\\nonumber
  H_{x,\nu}^-&=&\Big\{y\in\R^n:(y-x)\cdot\nu\le 0\Big\}\,.
\end{eqnarray}
Finally, we decompose $\R^n$ as the product $\R^{n-1}\times\R$, and denote by $\p:\R^n\to\R^{n-1}$ and $\q:\R^n\to\R$ the corresponding horizontal and vertical projections, so that
\[
x=(\p x,\q x)=(x',x_n)\,,\qquad x'=(x_1,\dots,x_{n-1})\,,\qquad\forall x\in\R^n\,,
\]
and define the vertical cylinder of center $x\in\R^n$ and radius $r>0$, and the $(n-1)$-dimensional ball in $\R^{n-1}$ of center $z\in\R^{n-1}$ and radius $r>0$, by setting, respectively,
\begin{eqnarray*}
\C_{x,r}&=&\Big\{y\in\R^n:|\p x- \p y|<r\,,|\q x- \q y|<r\Big\}\,,
\\
\D_{z,r}&=&\Big\{w\in\R^{n-1}:|w-z|<r\Big\}\,.
\end{eqnarray*}
In this way, $\C_{x,r}=\D_{\p x,r}\times(\q x-r,\q x+r)$. We shall use the following two notions of convergence for Lebesgue measurable subsets of $\R^n$. Given Lebesgue measurable sets $\{E_h\}_{h\in\N}$ and $E$ in $\R^n$, we shall say that $E_h$ locally converge to $E$, and write
\[
E_h\toloc E\,,\qquad\mbox{as $h\to\infty$}\,,
\]
provided $\H^n((E_h\Delta E)\cap K)\to 0$ as $h\to\infty$ for every compact set $K\subset\R^n$; we say that $E_h$ converge to $E$ as $h\to\infty$, and write $E_h\to E$, provided $\H^n(E_h\Delta E)\to 0$ as $h\to\infty$.

\subsection{Density points and approximate limits}\label{section approximate limits} If $E$ is a Lebesgue measurable set in $\R^n$ and $x\in\R^n$, then we define the {\it upper and lower $n$-dimensional densities} of $E$ at $x$ as
\begin{eqnarray*}
  \theta^*(E,x)=\limsup_{r\to 0^+}\frac{\H^n(E\cap B(x,r))}{\om_n\,r^n}\,,
  \qquad
  \theta_*(E,x)=\liminf_{r\to 0^+}\frac{\H^n(E\cap B(x,r))}{\om_n\,r^n}\,,
\end{eqnarray*}
respectively. In this way we define two Borel functions on $\R^n$, that agree a.e. on $\R^n$. In particular, the {\it $n$-dimensional density} of $E$ at $x$
\[
\theta(E,x)=\lim_{r\to 0^+}\frac{\H^n(E\cap B(x,r))}{\om_n\,r^n}\,,
\]
is defined for a.e. $x\in\R^n$, and $\theta(E,\cdot)$ is a Borel function on $\R^n$ (up to extending it by a constant value on the $\H^n$-negligible set $\{\theta^*(E,\cdot)>\theta_*(E,\cdot)\}$). Correspondingly, for $t\in[0,1]$, we define
\begin{eqnarray}\label{E(t)}
  E^{(t)}=\Big\{x\in\R^n:\theta(E,x)=t\Big\}\,.
\end{eqnarray}
By the Lebesgue differentiation theorem, $\{E^{(0)},E^{(1)}\}$ is a partition of $\R^n$ up to a $\H^n$-negligible set. It is useful to keep in mind that
\begin{eqnarray*}
  &&x\in E^{(1)}\qquad\mbox{if and only if}\qquad E_{x,r}\toloc\R^n\quad\mbox{as $r\to 0^+$}\,,
  \\
  &&x\in E^{(0)}\qquad\mbox{if and only if}\qquad E_{x,r}\toloc\emptyset\quad\mbox{as $r\to 0^+$}\,,
\end{eqnarray*}
where $E_{x,r}$ denotes the blow-up of $E$ at $x$ at scale $r$, defined as
\begin{eqnarray*}
  E_{x,r}=\frac{E-x}{r}=\Big\{\frac{y-x}r:y\in E\Big\}\,,\qquad x\in\R^n\,,r>0\,.
\end{eqnarray*}
The set $\pae E=\R^n\setminus(E^{(0)}\cup E^{(1)})$ is called the {\it essential boundary} of $E$. Thus, in general, we only have $\H^n(\pae E)=0$, but we do not know $\pae E$ to be ``$(n-1)$-dimensional'' in any sense. Strictly related to the notion of density is that of approximate upper and lower limits of a measurable function. Given a Lebesgue measurable function $f:\R^n\to\R$ we define the (weak) {\it approximate upper and lower limits} of $f$ at $x\in\R^n$ as
  \begin{eqnarray*}
  f^\vee(x)&=&\inf\Big\{t\in\R:\theta(\{f>t\},x)=0\Big\}=\inf\Big\{t\in\R:\theta(\{f<t\},x)=1\Big\}\,,
  \\
  f^\wedge(x)&=&\sup\Big\{t\in\R:\theta(\{f<t\},x)=0\Big\}=\sup\Big\{t\in\R:\theta(\{f>t\},x)=1\Big\}\,.
  \end{eqnarray*}
  As it turns out, $f^\vee$ and $f^\wedge$ are Borel functions with values on $\R\cup\{\pm\infty\}$ defined \textit{at every point $x$} of $\R^n$, and they do not depend on the Lebesgue representative chosen for the function $f$. Moreover, for $\H^n$-a.e. $x\in \R^n$, we have that $f^\vee(x)=f^\wedge(x)\in\R\cup\{\pm\infty\}$, so that the {\it approximate discontinuity} set of $f$, $S_f=\{f^\wedge<f^\vee\}$,
  satisfies $\H^n(S_f)=0$. On noticing that, even if $f^\wedge$ and $f^\vee$ may take infinite values on $S_f$, the difference $f^\vee(x)-f^\wedge(x)$ is always well defined in $\R\cup\{\pm\infty\}$ for $x\in S_f$, we define the {\it approximate jump} of $f$ as the Borel function $[f]:\R^n\to[0,\infty]$ defined by
  \begin{eqnarray*}
    [f](x)=\left\{\begin{array}{l l}
      f^\vee(x)-f^\wedge(x)\,,&\mbox{if $x\in S_f$}\,,
      \\
      0\,,&\mbox{if $x\in \R^n\setminus S_f$}\,.
    \end{array}
    \right .
  \end{eqnarray*}
  so that $S_f=\{[f]>0\}$. Finally, the {\it approximate average} of $f$ is the Borel function $\widetilde{f}:\R^n\to\R\cup\{\pm\infty\}$ defined as
  \begin{eqnarray}\label{f average}
    \widetilde{f}(x)=\left\{\begin{array}{l l}
      \frac{f^\vee(x)+f^\wedge(x)}2\,,&\mbox{if $x\in \R^n\setminus \{f^\wedge=-\infty\,, f^\vee=+\infty\}$}\,,
      \\
      0\,,&\mbox{if $x\in\{f^\wedge=-\infty\,, f^\vee=+\infty\}$}\,.
    \end{array}
    \right .
  \end{eqnarray}
  The motivation behind definition \eqref{f average} is that (in step two of the proof of Theorem \ref{lemma u1u2}) we want the limit relation
  \begin{equation}
    \label{limite mistico}
      \widetilde{f}(x)=\lim_{M\to\infty}\widetilde{\tau_M(f)}(x)=\lim_{M\to\infty}\frac{\tau_M(f^\vee)+\tau_M(f^\wedge)}2\,,\qquad\forall x\in\R^n\,,
  \end{equation}
  to hold true for every Lebesgue measurable function $f:\R^n\to\R$, where here and in the rest of the paper we set
  \begin{equation}\label{def tauM}
    \tau_M(s)=\max\{-M,\min\{M,s\}\}\,,\qquad s\in\R\cup\{\pm\infty\}\,.
  \end{equation}
  The validity of \eqref{limite mistico} is easily checked by noticing that
  \begin{equation}
    \label{tauM wedge}
    \tau_M(f)^\wedge=\tau_M(f^\wedge)\,,\qquad
        \tau_M(f)^\vee=\tau_M(f^\vee)\,,\qquad
        \widetilde{\tau_M(f)}(x)=\frac{\tau_M(f^\vee)+\tau_M(f^\wedge)}2\,.
  \end{equation}
 With these definitions at hand, we notice the validity of the following properties, which follow easily from the above definitions, and hold true for every Lebesgue measurable $f:\R^n\to\R$ and for every $t \in \mathbb{R}$:
\begin{align}
\{ |f|^{\vee}  < t \} &=\{ - t < f^{\wedge} \} \cap \{ f^{\vee} < t \}\,, \label{dens1}
\\
\{ f^{\vee} < t \} &\subset \{ f < t \}^{(1)}\subset\{f^\vee\le t\}\,, \label{dens2}
\\
\{ f^{\wedge} > t \} &\subset \{ f > t \}^{(1)}\subset\{f^\wedge \ge t\}\,. \label{dens3}
\end{align}
(Note that all the inclusions may be strict, that we also have $\{ f < t \}^{(1)}=\{ f^\vee < t \}^{(1)}$, and that all the other analogous relations hold true.) Moreover, if $f,g:\R^n\to\R$ are Lebesgue measurable functions and $f=g$ $\H^n$-a.e. on a Borel set $E$, then
\begin{equation}\label{gino1}
f^{\vee} (x) = g^{\vee} (x) \,,\qquad f^{\wedge} (x) = g^{\wedge} (x)\,,\qquad [f](x)=[g](x)\,,\qquad\forall x\in E^{(1)}\,.
\end{equation}
If $f:\R^n\to\R$ and $A\subset\R^n$ are Lebesgue measurable, and $x\in\R^n$ is such that $\theta^*(A,x)>0$, then we say that $t\in\R\cup\{\pm\infty\}$ is the {\it approximate limit of $f$ at $x$ with respect to $A$}, and write $t=\aplim (f,A,x)$, if
\begin{eqnarray*}
  &&\theta\Big(\{|f-t|>\e\}\cap A;x\Big)=0\,,\qquad\forall\e>0\,,\hspace{0.3cm}\qquad (t\in\R)\,,
  \\
  &&\theta\Big(\{f<M\}\cap A;x\Big)=0\,,\qquad\hspace{0.6cm}\forall M>0\,,\qquad (t=+\infty)\,,
  \\
  &&\theta\Big(\{f>-M\}\cap A;x\Big)=0\,,\qquad\hspace{0.3cm}\forall M>0\,,\qquad (t=-\infty)\,.
\end{eqnarray*}
We say that $x\in S_f$ is a jump point of $f$ if there exists $\nu\in S^{n-1}$ such that
\[
f^\vee(x)=\aplim(f,H_{x,\nu}^+,x)\,,\qquad f^\wedge(x)=\aplim(f,H_{x,\nu}^-,x)\,.
\]
If this is the case we set $\nu=\nu_f(x)$, the approximate jump direction of $f$ at $x$. We denote by $J_f$ the set of approximate jump points of $f$, so that $J_f\subset S_f$; moreover, $\nu_f:J_f\to S^{n-1}$ is a Borel function. It will be particularly useful to keep in mind the following proposition; see \cite[Proposition 2.2]{ccdpmGAUSS} for a proof.

\begin{proposition}\label{lemma nonciserve}
  We have that $x\in J_f$ if and only if,  for every $\tau\in(f^\wedge(x),f^\vee(x))$,
  \begin{eqnarray}\label{jump point level sets}
    \{f>\tau\}_{x,r}\toloc H_{0,\nu}^+\,,\qquad\{f<\tau\}_{x,r}\toloc H_{0,\nu}^-\,,\qquad\mbox{as $r\to 0^+$}\,.
  \end{eqnarray}
\end{proposition}

%\begin{proof} We prove the ``only if'' part of the first equivalence only, leaving the other implications to the reader. Let us set $t=f^\vee(x)$ and $s=f^\wedge(x)$. By assumption
%  \begin{eqnarray*}
%   \Big(\Big\{|f-t|>\e\Big\}\cap H_{x,\nu}^+\Big)_{x,r}\toloc\emptyset\,,
%   \qquad
%   \Big(\Big\{|f-s|>\e\Big\}\cap H_{x,\nu}^-\Big)_{x,r}\toloc\emptyset\,,
%  \end{eqnarray*}
%  as $r\to 0^+$. As a consequence, as $r\to 0^+$,
%  \begin{eqnarray*}
%   \Big(\Big\{|f-t|\le\e\Big\}\cup H_{x,\nu}^-\Big)_{x,r}\toloc\R^n\,,
%   \qquad
%   \Big(\Big\{|f-s|\le\e\Big\}\cup H_{x,\nu}^+\Big)_{x,r}\toloc\R^n\,.
%  \end{eqnarray*}
%  As $E^{(1)}\cap F^{(1)}=(E\cap F)^{(1)}$, we find
%   \begin{eqnarray*}
%   &&\Big(\Big(\Big\{|f-t|\le\e\Big\}\cup H_{x,\nu}^-\Big)
%   \cap\Big(\Big\{|f-s|\le\e\Big\}\cup H_{x,\nu}^+\Big)\Big)_{x,r}\toloc\R^n\,,
%   \\\
%   \mbox{that is}\hspace{0.5cm}&&\Big(\Big\{|f-t|\le\e\Big\}\cap H_{x,\nu}^+\Big)_{x,r}
%   \cup\Big(\Big\{|f-s|\le\e\Big\}\cap H_{x,\nu}^-\Big)_{x,r}\toloc\R^n\,,
%  \end{eqnarray*}
%  Since the two sets are disjoint, the first one contained in $H_{\nu,0}^+$, the second one in $H_{0,\nu}^-$, we complete the proof.
%\end{proof}

Finally, if $f:\R^n\to\R$ is Lebesgue measurable, then we say $f$ is {\it approximately differentiable} at $x\in S_f^c$ provided $f^\wedge(x)=f^\vee(x)\in\R$ and there exists $\xi\in\R^n$ such that
\[
\aplim(g,\R^n,x)=0\,,
\]
where $g(y)=(f(y)-\widetilde{f}(x)-\xi\cdot(y-x))/|y-x|$ for $y\in\R^n\setminus\{x\}$. If this is the case, then $\xi$ is uniquely determined, we set $\xi=\nabla f(x)$, and call $\nabla f(x)$ the {\it approximate differential} of $f$ at $x$. The localization property \eqref{gino1} holds true also for approximate differentials: precisely, if $f,g:\R^n\to\R$ are Lebesgue measurable functions, $f=g$ $\H^n$-a.e. on a Borel set $E$, and $f$ is approximately differentiable $\H^n$-a.e. on $E$, then $g$ is approximately differentiable $\H^n$-a.e. on $E$ too, with
\begin{equation}
  \label{gino2}
  \nabla f(x)=\nabla g(x)\,,\qquad\mbox{for $\H^n$-a.e. $x\in E$}\,.
\end{equation}

\subsection{Rectifiable sets and functions of bounded variation}\label{section sofp} Let $1\le k\le n$, $k\in\N$. A Borel set $M\subset\R^n$ is {\it countably $\H^k$-rectifiable} if there exist  Lipschitz functions $f_h:\R^k\to\R^n$ ($h\in\N$) such that $M\subset_{\H^k}\bigcup_{h\in\N}f_h(\R^k)$.
We further say that $M$ is {\it locally $\H^k$-rectifiable} if $\H^k(M\cap K)<\infty$ for every compact set $K\subset\R^n$, or, equivalently, if $\H^k\llcorner M$ is a Radon measure on $\R^n$. Hence, for a locally $\H^k$-rectifiable set $M$ in $\R^n$ the following definition is well-posed: we say that $M$ has a $k$-dimensional subspace $L$ of $\R^n$ as its {\it approximate tangent plane} at $x\in\R^n$, $L=T_xM$, if $\H^k\llcorner (M-x)/r\weak \H^k\llcorner L$ as $r\to 0^+$ weakly-star in the sense of Radon measures. It turns out that $T_xM$ exists and is uniquely defined at $\H^k$-a.e. $x\in M$. Moreover, given two locally $\H^k$-rectifiable sets $M_1$ and $M_2$ in $\R^n$, we have $T_xM_1=T_xM_2$ for $\H^k$-a.e. $x\in M_1\cap M_2$.

A Lebesgue measurable set $E\subset\R^n$ is said of {\it locally finite perimeter} in $\R^n$ if there exists a $\R^n$-valued Radon measure $\mu_E$, called the {\it Gauss--Green measure} of $E$, such that
\[
\int_E\nabla\vphi(x)\,dx=\int_{\R^n}\vphi(x)\,d\mu_E(x)\,,\qquad\forall \vphi\in C^1_c(\R^n)\,.
\]
The relative perimeter of $E$ in $A\subset\R^n$ is then defined by setting $P(E;A)=|\mu_E|(A)$, while $P(E)=P(E;\R^n)$ is the perimeter of $E$. The {\it reduced boundary} of $E$ is the set $\pa^*E$ of those $x\in\R^n$ such that
\[
\nu_E(x)=\lim_{r\to 0^+}\,\frac{\mu_E(B(x,r))}{|\mu_E|(B(x,r))}\qquad\mbox{exists and belongs to $S^{n-1}$}\,.
\]
The Borel function $\nu_E:\pa^*E\to S^{n-1}$ is called the {\it measure-theoretic outer unit normal} to $E$. It turns out that $\pa^*E$ is a locally $\H^{n-1}$-rectifiable set in $\R^n$ \cite[Corollary 16.1]{maggiBOOK}, that $\mu_E=\nu_E\,\H^{n-1}\llcorner\pa^*E$, and that
\[
\int_E\nabla\vphi(x)\,dx=\int_{\pa^*E}\vphi(x)\,\nu_E(x)\,d\H^{n-1}(x)\,,\qquad\forall \vphi\in C^1_c(\R^n)\,.
\]
In particular, $P(E;A)=\H^{n-1}(A\cap\pa^*E)$ for every Borel set $A\subset\R^n$. We say that $x\in\R^n$ is a jump point of $E$, if  there exists $\nu\in S^{n-1}$ such that
\begin{equation}
  \label{jump point of E}
  E_{x,r}\toloc H_{0,\nu}^+\,,\qquad\mbox{as $r\to 0^+$}\,,
\end{equation}
and we denote by $\pa^JE$ the set of {\it jump points} of $E$. Notice that we always have $\pa^JE\subset E^{(1/2)}\subset\pae E$. In fact, if $E$ is a set of locally finite perimeter and $x\in\pa^*E$, then \eqref{jump point of E} holds true with $\nu=-\nu_E(x)$, so that $\pa^*E\subset\pa^JE$. Summarizing, if $E$ is a set of locally finite perimeter, we have
\begin{equation}
  \label{inclusioni frontiere}
  \pa^*E\subset\pa^J E\subset E^{1/2}\subset \pae E\,,
\end{equation}
and, moreover, by {\it Federer's theorem} \cite[Theorem 3.61]{AFP}, \cite[Theorem 16.2]{maggiBOOK},
\[
\H^{n-1}(\pae E\setminus\pa^*E)=0\,,
\]
so that $\pae E$ is locally $\H^{n-1}$-rectifiable in $\R^n$. We shall need at several occasions to use the following very fine criterion for finite perimeter, known as {\it Federer's criterion} \cite[4.5.11]{FedererBOOK} (see also \cite[Theorem 1, section 5.11]{EvansGariepyBOOK}): if $E$ is a Lebesgue measurable set in $\R^n$ such that $\pae E$ is locally $\H^{n-1}$-finite, then $E$ is a set of locally finite perimeter.

Given a Lebesgue measurable function $f:\R^n\to\R$ and an open set $\Om\subset\R^n$ we define the {\it total variation of $f$ in $\Om$} as
\[
|Df|(\Om)=\sup\Big\{\int_\Om\,f(x)\,\Div\,T(x)\,dx:T\in C^1_c(\Om;\R^n)\,,|T|\le 1\Big\}\,.
\]
We say that $f\in BV(\Om)$ if $|Df|(\Om)<\infty$ and $f\in L^1(\Om)$, and that $f\in BV_{loc}(\Om)$ if $f\in BV(\Om')$ for every open set $\Om'$ compactly contained in $\Om$. If $f\in BV_{loc}(\R^n)$ then the distributional derivative $Df$ of $f$ is an $\R^n$-valued Radon measure. Notice in particular that $E$ is a set of locally finite perimeter if and only if $1_E\in BV_{loc}(\R^{n})$, and that in this case $\mu_E=-D1_E$. Sets of finite perimeter and functions of bounded variation are related by the fact that, if $f\in BV_{loc}(\R^n)$, then, for a.e. $t\in\R$, $\{f>t\}$ is a set of finite perimeter, and the {\it coarea formula},
\begin{equation}
  \label{coarea bv}
  \int_\R\,P(\{f>t\};G)\,dt=|Df|(G)\,,
\end{equation}
holds true (as an identity in $[0,\infty]$) for every Borel set $G\subset\R^n$. If $f\in BV_{loc}(\R^n)$, then the Radon--Nykodim decomposition of $Df$ with respect to $\H^n$ is denoted by $Df=D^af+D^sf$, where $D^sf$ and $\H^n$ are mutually singular, and where $D^af\ll\H^n$. The density of $D^af$ with respect to $\H^n$ is by convention denoted as $\nabla f$, so that $\nabla\,f\in L^1(\Om;\R^n)$ with $D^af=\nabla f\,d\H^n$. Moreover, for a.e. $x\in\R^n$, $\nabla f(x)$ is the approximate differential of $f$ at $x$. If $f\in BV_{loc}(\R^n)$, then $S_f$ is countably $\H^{n-1}$-rectifiable, with $\H^{n-1}(S_f\setminus J_f)=0$, $[f]\in L^1_{loc}(\H^{n-1}\llcorner J_f)$, and the $\R^n$-valued Radon measure $D^jf$ defined as
\[
D^jf=[f]\,\nu_f\,d\H^{n-1}\llcorner J_f\,,
\]
is called the {\it jump part of $Df$}. Since $D^af$ and $D^jf$ are mutually singular, by setting $D^cf=D^sf-D^jf$ we come to the canonical decomposition of $Df$ into the sum $D^af+D^jf+D^cf$. The $\R^n$-valued Radon measure $D^cf$ is called the {\it Cantorian part} of $Df$. It has the distinctive property that $|D^cf|(M)=0$ if $M$ is $\s$-finite with respect to $\H^{n-1}$. We shall often need to use (in combination with \eqref{gino1} and \eqref{gino2}) the following localization property of Cantorian derivatives.

\begin{lemma}\label{lemma Dcv} If $v\in BV(\R^n)$, then $|D^cv|(\{v^\wedge=0\})=0$. In particular, if $f,g\in BV(\R^n)$ and $f=g$ $\H^n$-a.e. on a Borel set $E$, then $D^cf\llcorner E^{(1)}=D^cg\llcorner E^{(1)}$.
\end{lemma}

\begin{proof}
  {\it Step one:} Let $v\in BV(\R^n)$, and let $K\subset S_v^c$ be a concentration set for $D^cv$ that is $\H^n$-negligible. By the coarea formula,
  \begin{eqnarray*}
  |D^cv|(\{v^\wedge=0\})&=&|D^cv|(K\cap\{v^\wedge=0\})=|Dv|(K\cap\{v^\wedge=0\})
  \\
  &=&\int_\R\H^{n-2}(K\cap\{v^\wedge=0\}\cap\pa^*\{v>t\})\,dt
  \\\mbox{{\small (by $v^\wedge=v^\vee$ on $S_v^c$)}}\qquad
  &=&
  \int_\R\H^{n-2}(K\cap\{\widetilde{v}=0\}\cap\pa^*\{v>t\})\,dt=0\,.
  \end{eqnarray*}
  where in the last identity we have noticed that $\{\widetilde{v}=0\}\cap\pa^*\{v>t\}\cap S_v^c=\emptyset$ if $t\ne 0$.

  \medskip

  \noindent {\it Step two:} Let $f,g\in BV(\R^n)$ with $f=g$ $\H^n$-a.e. on a Borel set $E$. Let $v=f-g$ so that $v\in BV(\R^n)$. Since $v=0$ on $E$ we easily see that $E^{(1)}\subset\{\widetilde{v}=0\}$. Thus $|D^cv|(E^{(1)})=0$ by step one.
\end{proof}

\begin{lemma}
  \label{lemma fgE} If $f,g\in BV(\R^n)$, $E$ is a set of finite perimeter, and $f=1_E\,g$, then
  \begin{eqnarray}\label{sorella1}
    \nabla f=1_E\,\nabla g\,,&&\qquad\mbox{$\H^n$-a.e. on $\R^n$}\,,
    \\\label{sorella2}
    D^cf=D^cg\llcorner E^{(1)}\,,&&
    \\\label{sorella3}
    S_f\cap E^{(1)}=S_g\cap E^{(1)}\,.&&
  \end{eqnarray}
\end{lemma}

\begin{proof}
  Since $f=g$ on $E$ by \eqref{gino2} we find that $\nabla f=\nabla g$ $\H^n$-a.e. on $E$; since $f=0$ on $\R^n\setminus E$, again by \eqref{gino2} we find that $\nabla f=0$ $\H^n$-a.e. on $\R^n\setminus E$; this proves \eqref{sorella1}. For the same reasons, but this time exploiting Lemma \ref{lemma Dcv} in place of \eqref{gino2}, we see that $D^cf\llcorner E^{(1)}=D^cg\llcorner E^{(1)}$ and that $D^cf\llcorner (\R^n\setminus E)^{(1)}=D^cf\llcorner E^{(0)}=0$; since $\pae E$ is locally $\H^{n-2}$-rectifiable, and thus $|D^cf|$-negligible, we come to prove \eqref{sorella2}. Finally, \eqref{sorella3} is an immediate consequence of \eqref{gino1}.
\end{proof}

Given a Lebesgue measurable function $f:\R^n\to\R$ we say that $f$ is a function of {\it generalized bounded variation} on $\R^n$, $f\in GBV(\R^n)$, if $\psi\circ f\in BV_{loc}(\R^n)$ for every $\psi\in C^1(\R)$ with $\psi'\in C^0_c(\R)$, or, equivalently, if $\tau_M(f)\in BV_{loc}(\R^{n})$ for every $M>0$, where $\tau_M$ was defined in \eqref{def tauM}. Notice that, if $f\in GBV(\R^{n})$, then we do not even ask that $f\in L^1_{loc}(\R^n)$, so that the distributional derivative $Df$ of $f$ may even fail to be defined. Nevertheless, the structure theory of $BV$-functions holds true for $GBV$-functions too. Indeed, if $f\in GBV(\R^{n})$, then, see \cite[Theorem 4.34]{AFP}, $\{f>t\}$ is a set of finite perimeter for a.e. $t\in\R$, $f$ is approximately differentiable $\H^n$-a.e. on $\R^n$, $S_f$ is countably $\H^{n-1}$-rectifiable and $\H^{n-1}$-equivalent to $J_f$, and the coarea formula \eqref{coarea bv} takes the form
\begin{equation}
  \label{coarea gbv}
  \int_\R\,P(\{f>t\};G)\,dt=\int_G|\nabla f|\,d\H^n+\int_{G\cap S_f}[f]\,d\H^{n-1}+|D^cf|(G)\,,
\end{equation}
for every Borel set $G\subset\R^n$, where $|D^cf|$ denotes the Borel measure on $\R^n$ defined as the least upper bound of the Radon measures $|D^c(\tau_M(f))|$; and, in fact,
\begin{equation}
  \label{Dc GBV}
  |D^cf|(G)=\lim_{M\to\infty}|D^c(\tau_M(f))|(G)=\sup_{M>0}|D^c(\tau_M(f))|(G)\,,
\end{equation}
whenever $G$ is a Borel set in $\R^n$; see \cite[Definition 4.33]{AFP}.

\section{Characterization of equality cases and barycenter functions}\label{section equality cases steiner} We now prove the results presented in section \ref{section intro baricentro}. In section \ref{section segments sets}, Theorem \ref{lemma u1u2}, we obtain a formula for the perimeter of a set whose sections are segments, which is then applied in section \ref{section barycenter} to study barycenter functions of such sets, and prove Theorem \ref{thm tauM b delta}. Sections \ref{section v in bv} and \ref{section necessary cond barycenter} contain the proof of Theorem \ref{thm characterization with barycenter} concerning the characterization of equality cases in terms of barycenter functions, while Theorem \ref{thm mv per v sbv} is proved in section \ref{section proof charact Mv per v sbv}.

\subsection{Sets with segments as sections}\label{section segments sets} Given $u:\R^{n-1}\to\R\cup\{\pm\infty\}$, let us denote by $\S_u=\{x\in\R^n:\q x>u(\p x)\}$ and $\S^{u}=\{x\in\R^n:\q x<u(\p x)\}$, respectively, the epigraph and the subgraph of $u$. As proved in \cite[Proposition 3.1]{ccdpmGAUSS}, $\S_u$ is a set of locally finite perimeter if and only if $\tau_M(u)\in BV_{loc}(\R^{n-1})$ for every $M>0$. (Note that this does not mean that $u\in GBV(\R^{n-1})$, as here $u$ takes values in $\R\cup\{\pm\infty\}$.) Moreover, it is well known that if $u\in BV_{loc}(\R^{n-1})$, then, for every Borel set $G\subset\R^{n-1}$, the identity
\begin{equation}
  \label{goffmanserrin}
  P(\S_u;G\times\R)=\int_G\sqrt{1+|\nabla u|^2}\,d\H^{n-1}+\int_{G\cap S_u}[u]\,d\H^{n-2}+|D^cu|(G)\,,
\end{equation}
holds true in $[0,\infty]$; see \cite[Chapter 4, Section 1.5 and 2.4]{GMSbook2}. In the study of equality cases for Steiner's inequality, thanks to Theorem \ref{thm ccf1}, we are concerned with sets $E$ of the form $E=\S_{u_1}\cap \S^{u_2}$ corresponding to Lebesgue measurable functions $u_1$ and $u_2$ such that $u_1\le u_2$ on $\R^{n-1}$. A characterization of those pairs of functions $u_1, u_2$ corresponding to sets $E$ of finite perimeter and volume is presented in Proposition \ref{proposition insieme u1u2}. In Theorem \ref{lemma u1u2}, we provide instead a formula for the perimeter of $E$ in terms of $u_1$ and $u_2$ in the case that $u_1,u_2\in GBV(\R^{n-1})$, that is analogous to \eqref{goffmanserrin}.

\begin{theorem}\label{lemma u1u2}
  If $u_1\,,u_2\in GBV(\R^{n-1})$ with $u_1\le u_2$, and $E=\S_{u_1}\cap\S^{u_2}$ has finite volume, then $E$ is a set of locally finite perimeter and, for every Borel set $G\subset\R^{n-1}$,
  \begin{eqnarray}\nonumber
  P(E;G\times\R)
  &=&\int_{G\cap\{u_1<u_2\}}\sqrt{1+|\nabla u_1|^2}\,d\H^{n-1}+\int_{G\cap\{u_1<u_2\}}\sqrt{1+|\nabla u_2|^2}\,d\H^{n-1}
  \\\nonumber
  &&+|D^cu_1|\Big(G\cap\{\widetilde{u}_1<\widetilde{u}_2\}\Big)+|D^cu_2|\Big(G\cap\{\widetilde{u}_1<\widetilde{u}_2\}\Big)
  \\\label{dead0}
  &&+\int_{G\cap(S_{u_1}\cup S_{u_2})}\,\min\Big\{ 2 (\widetilde{u}_2 - \widetilde{u}_1) , [u_1] + [u_2] \Big\} \,d\H^{n-2}\,,
  \end{eqnarray}
  where this identity holds true in $[0,\infty]$, and with the convention that $\widetilde{u}_2 - \widetilde{u}_1=0$ when $\widetilde{u}_2= \widetilde{u}_1=+\infty$.
\end{theorem}

If $E=\S_{u_1}\cap\S^{u_2}$ is of locally finite perimeter, then it is not necessarily true that $u_1\,,u_2\in GBV(\R^{n-1})$. The regularity of $u_1$ and $u_2$ is, in fact, quite minimal, and completely degenerates as we approach the set where $u_1$ and $u_2$ coincide.

\begin{proposition}\label{proposition insieme u1u2}
  Let $u_1,u_2:\R^{n-1}\to\R$ be Lebesgue measurable functions with $u_1\le u_2$ on $\R^{n-1}$. Then $E=\S_{u_1}\cap\S^{u_2}$ is of finite perimeter with $0<|E|<\infty$ if and only if $v=u_2-u_1\in BV(\R^{n-1})$, $v\ne 0$, $\H^{n-1}(\{v>0\})<\infty$, $\{u_2>t>u_1\}$ is of finite perimeter for a.e. $t\in\R$, and $f\in L^1(\R)$ for $f(t)=P(\{u_2>t>u_1\})$, $t\in\R$. In both cases,
  \begin{eqnarray*}
    \int_{\R}P(\{u_2>t>u_1\})\,dt\le P(E)\,,
    \\
    |Dv|(\R^{n-1})\le P(F[v])\,,
    \\
    \H^{n-1}(\{v>0\})\le \frac{P(F[v])}2\,.
  \end{eqnarray*}
  Moreover,
  \begin{figure}
    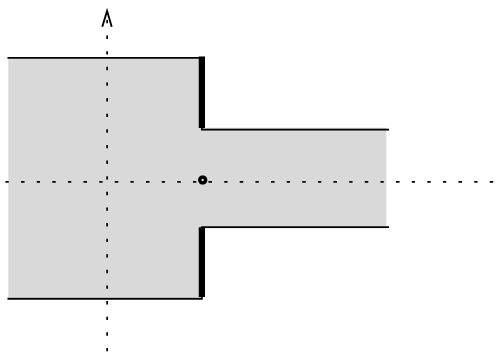\caption{{\small Inclusion \eqref{uuuuu}.}}\label{fig u1u2}
  \end{figure}
  see Figure \ref{fig u1u2},
  \begin{equation}
    \label{uuuuu}
      (\pae E)_z\subset[u_1^\wedge(z),u_1^\vee(z)]\cup[u_2^\wedge(z),u_2^\vee(z)]\,,\qquad\forall z\in\R^{n-1}\,,
  \end{equation}
  and
  \begin{equation}
    \label{casa}
      \Big(S_{u_1}\cup S_{u_2}\Big)\setminus\Big(\{u_2^\vee=u_1^\vee\}\cap\{u_2^\wedge=u_1^\wedge\}\Big)
  \end{equation}
  is countably $\H^{n-2}$-rectifiable, with $\{v^\vee=0\} \subseteq \{u_2^\vee=u_1^\vee\}\cap\{u_2^\wedge=u_1^\wedge\}$.
\end{proposition}

\begin{proof} We first notice that, if we set $E(t)=\{z\in\R^{n-1}:(z,t)\in E\}$, then we have $E(t)=\{u_1<t<u_2\}$ for every $t\in\R$, and that, by Fubini's theorem, $E$ has finite volume if and only if $v\in L^1(\R^{n-1})$; in both cases $|E|=\int_{\R^{n-1}}v$.

\medskip

\noindent   {\it Step one:} Let us assume that $E$ has finite perimeter with $0<|E|<\infty$.  As noticed, we have $v\in L^1(\R^{n-1})$. By Steiner's inequality, $F[v]$ has finite perimeter. By \cite[Proposition 19.22]{maggiBOOK}, since $|F[v]\cap\{x_n>0\}|=\int_{\R^{n-1}}v/2=|E|/2>0$, we have that
\[
\frac{P(F[v])}2\ge P(F[v];\{x_n>0\})\ge \H^{n-1}(F[v]^{(1)}\cap\{x_n=0\})=\H^{n-1}(\{v>0\})\,.
\]
If $T\in C^1_c(\R^{n-1};\R^{n-1})$ with $\sup_{\R^{n-1}}|T|\le 1$, and we set $S\in C^1_c(\R^n;\R^n)$ as $S(x)=(T(\p x),0)$, then by Fubini's theorem and Steiner's inequality we find that
  \begin{eqnarray*}
    \int_{\R^{n-1}} v(z)\,\Div'\,T(z)\,dz=\int_{F[v]}\,\Div\,S\le P(F[v])\le P(E)\,.
  \end{eqnarray*}
  Hence, $v\in BV(\R^{n-1})$, with $|Dv|(\R^{n-1})\le P(F[v])$. If $w_h\in C^1_c(\R^n)$ with $w_h\to 1_E$ in $L^1(\R^n)$ and $|Dw_h|(\R^n)\to P(E)$ as $h\to\infty$, then $w_h(\cdot,t)\to 1_{E(t)}$ in $L^1(\R^{n-1})$ for a.e. $t\in\R$, and, correspondingly
  \[
  \int_{E(t)}\,\Div'T=\lim_{h\to\infty}\int_{\R^{n-1}}w_h\,\Div'T
  =-\lim_{h\to\infty}\int_{\R^{n-1}}T\cdot\nabla w_h\le\lim_{h\to\infty}\int_{\R^{n-1}}|\nabla w_h(z,t)|\,dz\,.
  \]
  Hence, by Fatou's lemma,
  \begin{eqnarray*}
  \int_\R\,\sup\Big\{\Big|\int_{E(t)}\,\Div'T\Big|:T\in C^1_c(\R^{n-1};\R^{n-1})\,,\sup_{\R^{n-1}}|T|\le 1\Big\}\,dt
  &\le&\liminf_{h\to\infty}\int_{\R^n}|\nabla w_h|
  \\
  &=&P(E)\,,
  \end{eqnarray*}
  so that $E(t)$ is of finite perimeter for a.e. $t\in\R$, and $\int_\R P(E(t))\,dt\le P(E)$, as required.

  \medskip

  \noindent {\it Step two:} We have already noticed that $|E|=\int_{\R^{n-1}}v\in(0,\infty)$. If $\vphi\in C^1_c(\R^n)$, then
  \[
  \int_E\,\pa_n\vphi=\int_{\R^{n-1}}\vphi(z,u_2(z))-\vphi(z,u_1(z))\,dz\le 2\,\sup_{\R^n}|\vphi|\,\H^{n-1}(\{v>0\})\,.
  \]
  while
  \begin{eqnarray*}
    \int_E\,\nabla'\vphi&=&\int_{\R}dt\int_{E(t)}\nabla'\vphi(z,t)\,dz
    =\int_{\R}dt\int_{\pa^*E(t)}\vphi(z,t)\,\nu_{E(t)}(z)\,d\H^{n-2}(z)
    \\
    &\le&\sup_{\R^n}\,|\vphi|\int_{\q(\spt\,\vphi)}\,P(E(t))\,dt\,.
  \end{eqnarray*}
  If we set $f(t)=P(E(t))$, then we have just proved
  \[
  \Big|\int_E\nabla\vphi\Big|\le\sup_{\R^n}|\vphi|\Big(2\,\H^{n-1}(\{v>0\})+\|f\|_{L^1(\R)}\Big)\,,
  \]
  so that $E$ has finite perimeter.

\medskip

\noindent {\it Step three:} For every $x\in\R^n$ and $r>0$ we have
  \[
  \H^n(E\cap\C_{x,r})=\int_{\q x-r}^{\q x+r}\,\H^{n-1}(\D_{\p x,r}\cap\{u_1<s\}\cap\{u_2>s\})\,ds\,.
  \]
  If $\q x>u_2^\vee(\p x)$, then given $t\in(u_2^\vee(\p x),\q x)$ and $r<\q x -t$ we find that
  \[
  \H^n(E\cap\C_{x,r})\le 2\,r\,\H^{n-1}(\D_{\p x,r}\cap\{u_2>t\})=o(r^n)\,,
  \]
  so that $x\in E^{(0)}$. By a similar argument, we show that
  \begin{eqnarray*}
  \Big\{x\in\R^n:\q x>u_2^\vee(\p x)\Big\}\cup  \Big\{x\in\R^n:\q x<u_1^\wedge(\p x)\Big\}\subset E^{(0)}\,,
  \\
  \Big\{x\in\R^n:u_1^\vee(\p x)<\q x<u_2^\wedge(\p x)\Big\}\subset E^{(1)}\,.
  \end{eqnarray*}
  We thus conclude that, if $x\in\pae E$, then $u_1^\wedge(\p x)\le \q x\le u_2^\vee(\p x)$ and either $\q x\le u_1^\vee(\p x)$ or $\q x\ge u_2^\wedge(\p x)$.

  \medskip

  \noindent  {\it Step four:} Let $I$ be a countable dense subset of $\R$ such that $\{u_1<t<u_2\}$ is of finite perimeter for every $t\in I$. We claim that
  \begin{equation}
    \label{heyhey}
      \{u_2^\wedge>u_1^{\wedge}\}\cap S_{u_1}\subset\bigcup_{t\in I}\pae \{u_2>t>u_1\}\,.
  \end{equation}
  Indeed, if $\min\{u_2^\wedge(z),u_1^\vee(z)\}>t>u_1^{\wedge}(z)$, then
  \[
  \theta(\{u_2>t\},z)=1\,,\qquad \theta^*(\{u_1<t\},z)>0\,,\qquad \theta_*(\{u_1<t\},z)<1\,,
  \]
  which implies $\theta^*(\{u_1<t<u_2\},z)>0$ and that $\theta_*(\{u_1<t<u_2\},z)<1$, and thus \eqref{heyhey}. In particular, $\{u_2^\wedge>u_1^{\wedge}\}\cap S_{u_1}$ is countably $\H^{n-2}$-rectifiable. By entirely similar arguments, one checks that the sets $\{u_2^\vee>u_1^{\vee}\}\cap S_{u_2}$, $S_{u_1}^c\cap S_{u_2}$ and $S_{u_1}\cap S_{u_2}^c$ are included in the set on the right-hand side of \eqref{heyhey}, and thus complete the proof of \eqref{casa}.

  \medskip

  \noindent {\it Step five:} We prove that $\{v^\vee=0\} \subseteq \{u_2^\vee=u_1^\vee\}\cap\{u_2^\wedge=u_1^\wedge\}$.
Indeed from the general fact that $(f+g)^\vee \leq f^\vee +g^\vee$, we obtain that $ 0\le u_2^\vee - u_1^\vee  \leq (u_2-u_1)^\vee = v^\vee$.
At the same time, $0\le u_2^\wedge-u_1^\wedge=(-u_1)^\vee - (-u_2)^\vee  \leq (-u_1+u_2)^\vee = v^\vee$.
\end{proof}

\begin{proof}
  [Proof of Theorem \ref{lemma u1u2}]  {\it Step one:} We first consider the case that $u_1,u_2\in BV_{loc}(\R^{n-1})$. By \cite[Section 4.1.5]{GMSbook1}, $\S_{u_1}$ and $\S^{u_2}$ are of locally finite perimeter, with
  \begin{eqnarray}\label{Su pastar c}
  \pa^*\S_{u_1}\cap\Big(S_{u_1}^c\times\R\Big)&=_{\H^{n-1}}&\Big\{x\in\R^n:\widetilde{u}_1(\p x)=\q x\Big\}\,,
  \\\label{Su pastar j}
  \pa^*\S_{u_1}\cap\Big(S_{u_1}\times\R\Big)&=_{\H^{n-1}}&\Big\{x\in\R^n:u_1^\wedge(\p x)<\q x<u_1^\vee(\p x)\Big\}\,,
  \end{eqnarray}
  and, by similar arguments, with
  \begin{eqnarray}
  \label{Su 1 c}
  \S_{u_1}^{(1)}\cap\Big(S_{u_1}^c\times\R\Big)&=_{\H^{n-1}}&\Big\{x\in\R^n:\widetilde{u}_1(\p x)<\q x\Big\}\,,
  \\\label{Su 1 j}
  \S_{u_1}^{(1)}\cap\Big(S_{u_1}\times\R\Big)&=_{\H^{n-1}}&\Big\{x\in\R^n:u_1^\vee(\p x)<\q x\Big\}\,,
  \\
  \label{Su 1 cx}
  (\S^{u_2})^{(1)}\cap\Big(S_{u_2}^c\times\R\Big)&=_{\H^{n-1}}&\Big\{x\in\R^n:\widetilde{u}_2(\p x)>\q x\Big\}\,,
  \\
  \label{Su 1 jx}
  (\S^{u_2})^{(1)}\cap\Big(S_{u_2}\times\R\Big)&=_{\H^{n-1}}&\Big\{x\in\R^n:u_2^\wedge(\p x)>\q x\Big\}\,.
  \end{eqnarray}
  Let us now recall that, by \cite[Theorem 16.3]{maggiBOOK}, if $F_1$, $F_2$ are sets of locally finite perimeter, then
  \begin{equation}
    \label{yeeah}
      \pa^*(F_1\cap F_2)=_{\H^{n-1}}\Big(F_1^{(1)}\cap\pa^*F_2\Big)\cup \Big(F_2^{(1)}\cap\pa^*F_1\Big)\cup\Big(\pa^*F_1\cap \pa^*F_2\cap
  \{\nu_{F_1}=\nu_{F_2}\}\Big)\,;
  \end{equation}
  moreover, if $F_1\subset F_2$, then $\nu_{F_1}=\nu_{F_2}$ $\H^{n-1}$-a.e. on $\pa^*F_1\cap\pa^*F_2$. Since $u_1\le u_2$ implies $\S_{u_2}\subset\S_{u_1}$ and $\S^{u_2}=\R^n\setminus \S_{u_2}$, so that $\mu_{\S_{u_2}}=-\mu_{\S^{u_2}}$, we thus find
  \begin{equation}
    \label{wd furbo}
      \nu_{\S_{u_1}}=-\nu_{\S^{u_2}}\,,\qquad\mbox{$\H^{n-1}$-a.e. on $\pa^*\S_{u_1}\cap\pa^*\S^{u_2}$}\,.
  \end{equation}
  By \eqref{yeeah} and \eqref{wd furbo}, since $E=\S_{u_1}\cap\S^{u_2}$ we find
  \[
  \pa^*E=_{\H^{n-1}}\Big(\pa^*\S_{u_1}\cap(\S^{u_2})^{(1)}\Big)\cup \Big(\pa^*\S^{u_2}\cap(\S_{u_1})^{(1)}\Big)\,.
  \]
  We now apply \eqref{Su pastar c} to $u_1$ and \eqref{Su 1 cx} to $u_2$ to find
  \begin{eqnarray}\nonumber
  &&\Big(\pa^*\S_{u_1}\cap(\S^{u_2})^{(1)}\Big)\cap\Big((S_{u_1}^c\cap S_{u_2}^c)\times\R\Big)
  \\\label{wd1}
  &&=_{\H^{n-1}}\Big\{(z,\widetilde{u}_1(z)):z\in (S_{u_1}^c\cap S_{u_2}^c)\,,\widetilde{u}_1(z)<\widetilde{u}_2(z)\Big\}\,.
  \end{eqnarray}
  We combine \eqref{Su pastar j} applied to $u_1$ and \eqref{Su 1 cx} applied to $u_2$ to find
  \begin{eqnarray}\nonumber
  &&\Big(\pa^*\S_{u_1}\cap(\S^{u_2})^{(1)}\Big)\cap\Big((S_{u_1}\cap S_{u_2}^c)\times\R\Big)
  \\\label{wd2}
  &&=_{\H^{n-1}}\Big\{(z,t):z\in (S_{u_1}\cap S_{u_2}^c)\,,u_1^\wedge(z)<t<\min\{u_1^\vee(z),\widetilde{u}_2(z)\}\Big\}\,.
  \end{eqnarray}
  We combine \eqref{Su pastar j} applied to $u_1$ and \eqref{Su 1 jx} applied to $u_2$ to find
  \begin{eqnarray}\nonumber
  &&\Big(\pa^*\S_{u_1}\cap(\S^{u_2})^{(1)}\Big)\cap\Big((S_{u_1}\cap S_{u_2})\times\R\Big)
  \\\label{wd3}
  &&=_{\H^{n-1}}\Big\{(z,t):z\in (S_{u_1}\cap S_{u_2})\,,u_1^\wedge(z)<t<\min\{u_1^\vee(z),u_2^\wedge(z)\}\Big\}\,.
  \end{eqnarray}
  We finally apply \eqref{Su pastar c} to $u_1$ and \eqref{Su 1 jx} to $u_2$ to find
  \begin{eqnarray}\nonumber
  &&\Big(\pa^*\S_{u_1}\cap(\S^{u_2})^{(1)}\Big)\cap\Big((S_{u_1}^c\cap S_{u_2})\times\R\Big)
  \\\label{wd4}
  &&=_{\H^{n-1}}\Big\{(z,\widetilde{u}_1(z)):z\in (S_{u_1}^c\cap S_{u_2})\,,\widetilde{u}_1(z)<u_2^\wedge(z)\Big\}\,.
  \end{eqnarray}
  This gives, by \eqref{goffmanserrin},
  \begin{eqnarray*}
  &&\H^{n-1}\Big(\pa^*\S_{u_1}\cap(\S^{u_2})^{(1)}\cap(G\times\R)\Big)
  \\
  \mbox{{\small by \eqref{wd1}}}\qquad&=&\int_{G\cap\{u_1<u_2\}}\sqrt{1+|\nabla u_1|^2}\,d\H^{n-1}
  +|D^cu_1|\Big(G\cap\{\widetilde{u}_1<\widetilde{u}_2\}\Big)
  \\
  \mbox{{\small by \eqref{wd2} and \eqref{wd3}}}\qquad&&+\int_{G\cap S_{u_1}}\,
  \Big(\min\{u_1^\vee,u_2^\wedge\}-u_1^\wedge\Big)_+\,d\H^{n-2}\,,
  \end{eqnarray*}
  where we have also taken into account that, as a consequence of \eqref{wd4}, we simply have
  \[
  \H^{n-1}\Big(\Big(\pa^*\S_{u_1}\cap(\S^{u_2})^{(1)}\Big)\cap\Big((S_{u_1}^c\cap S_{u_2})\times\R\Big)\Big)=0\,,
  \]
  by \cite[3.2.23]{FedererBOOK}. Also, by symmetry,
  \begin{eqnarray*}
  &&\H^{n-1}\Big(\pa^*\S^{u_2}\cap(\S_{u_1})^{(1)}\cap(G\times\R)\Big)
  \\
  &=&\int_{G\cap\{u_1<u_2\}}\sqrt{1+|\nabla u_2|^2}\,d\H^{n-1}
  +|D^cu_2|\Big(G\cap\{\widetilde{u}_1<\widetilde{u}_2\}\Big)
  \\
  &&+\int_{G\cap S_{u_2}}\,\Big(u_2^\vee-\max\{u_2^\wedge,u_1^\vee\}\Big)_+\,d\H^{n-2}\,.
  \end{eqnarray*}
  In conclusion we have proved
  \begin{eqnarray}\nonumber
  P(E;G\times\R)
  &=&\int_{G\cap\{u_1<u_2\}}\Big(\sqrt{1+|\nabla u_1|^2}\,+\sqrt{1+|\nabla u_2|^2}\Big)\,d\H^{n-1}
  \\\label{dead3}
  &&+|D^cu_1|\Big(G\cap\{\widetilde{u}_1<\widetilde{u}_2\}\Big)+|D^cu_2|\Big(G\cap\{\widetilde{u}_1<\widetilde{u}_2\}\Big)
  \\\nonumber
  &&+\int_{G\cap(S_{u_1}\cup S_{u_2})}\,\Big(\min\{u_1^\vee,u_2^\wedge\}-u_1^\wedge\Big)_+
  +\Big(u_2^\vee-\max\{u_2^\wedge,u_1^\vee\}\Big)_+\,d\H^{n-2}\,.
  \end{eqnarray}
  We thus deduce \eqref{dead0} by means of \eqref{dead3} and thanks to the identity,
  \begin{eqnarray*}
  \min\Big\{ 2 (\widetilde{u}_2 - \widetilde{u}_1) , [u_1] + [u_2]\Big\}&=&
  \min\Big\{ u_2^{\vee} + u_2^{\wedge} - ( u_1^{\vee} + u_1^{\wedge} ),
  u_1^{\vee} - u_1^{\wedge} + u_2^{\vee} - u_2^\wedge\Big\}
  \\
  &=& u_2^{\vee} - u_1^{\wedge} + \min\Big\{u_2^{\wedge} - u_1^{\vee} , u_1^{\vee} - u_2^{\wedge}\Big\}
  \\
  &=& u_2^{\vee} - u_1^{\wedge} + \min\{ u_2^{\wedge} , u_1^{\vee}\} - \max\{ u_2^{\wedge} , u_1^{\vee}\}
  \\
  &=&\Big(\min\{u_1^\vee,u_2^\wedge\}-u_1^\wedge\Big)_+
  +\Big(u_2^\vee-\max\{u_2^\wedge,u_1^\vee\}\Big)_+\,.
  \end{eqnarray*}
  This completes the proof of the theorem in the case that $u_1,u_2\in BV_{loc}(\R^{n-1})$.

  \medskip

  \noindent {\it Step two:} We now address the general case. If $u_1,u_2\in GBV(\R^{n-1})$, then $\S_{u_1}$ and $\S^{u_2}$ are sets of locally finite perimeter by \cite[Proposition 3.1]{ccdpmGAUSS}, and thus $E$ is of locally finite perimeter. We now prove \eqref{dead0}. To this end, since \eqref{dead0} is an identity between Borel measures on $\R^{n-1}$, it suffices to consider the case that $G$ is {\it bounded}. Given $M>0$, let $E_M=\S_{\tau_M(u_1)}\cap \S^{\tau_M(u_2)}$. Since $\tau_Mu_i\in BV_{loc}(\R^{n-1})$ for every $M>0$, $i=1,2$, by step one we find that $E_M$ is a set of locally finite perimeter, and that \eqref{dead0} holds true on $E_M$ with $\tau_M(u_1)$ and $\tau_M(u_2)$ in place of $u_1$ and $u_2$. We are thus going to complete the proof of the theorem by showing that,
  \begin{eqnarray}\label{dead11}
    P(E;G\times\R)&=&\lim_{M\to\infty}P(E_M;G\times\R)\,,
    \\
    \int_{G\cap\{u_1<u_2\}}\sqrt{1+|\nabla u_i|^2}\,d\H^{n-1}&=&\lim_{M\to\infty}\int_{G\cap\{\tau_M(u_1)<\tau_M(u_2)\}}
    \hspace{-1.5cm}\sqrt{1+|\nabla \tau_M(u_i)|^2}\,d\H^{n-1}\,,
    \label{dead12}
    \\\label{dead13}
    |D^cu_i|\Big(G\cap\{\widetilde{u}_1<\widetilde{u}_2\}\Big)&=&\lim_{M\to\infty}|D^c\tau_M(u_i)|\Big(G\cap\{\widetilde{\tau_M(u_1)}<
    \widetilde{\tau_M(u_2)}\}\Big)\,,\hspace{0.5cm}
  \end{eqnarray}
  and that
  \begin{eqnarray}
    \label{dead14}
    &&\int_{G\cap(S_{u_1}\cup S_{u_2})}\,\min\Big\{ 2 (\widetilde{u}_2 - \widetilde{u}_1) , [u_1] + [u_2] \Big\} \,d\H^{n-2}
    \\\nonumber
    &=&\lim_{M\to\infty}
    \int_{G\cap(S_{\tau_M(u_1)}\cup S_{\tau_M(u_2)})}\,\min\Big\{ 2 (\widetilde{\tau_M(u_2)} - \widetilde{\tau_M(u_1)}) , [\tau_M(u_1)] + [\tau_M(u_2)] \Big\} \,d\H^{n-2}\,.
  \end{eqnarray}
  Let us set $f_M(a,b)=\tau_M(b)-\tau_M(a)$ for $a,b\in\R\cup\{\pm\infty\}$. By \eqref{tauM wedge}, we can write the right-hand side of \eqref{dead14} as $\int_G\,h_M\,d\H^{n-2}$, where
  \[
  h_M=1_{S_{\tau_M(u_1)}\cup S_{\tau_M(u_2)}}\,\,
  \g\Big(f_M(u_1^\vee,u_2^\vee),f_M(u_1^\wedge,u_2^\wedge),f_M(u_1^\wedge,u_1^\vee),f_M(u_2^\wedge,u_2^\vee)\Big)\,,
  \]
  for a function $\g:\R\times\R\times\R\times\R\to[0,\infty)$ that is increasing in each of its arguments. Since, for every $a,b\in\R\cup\{\pm\infty\}$ with $a\le b$, the quantity $f_M(a,b)$ is increasing in $M$, with
  \[
  \lim_{M\to\infty}f_M(a,b)=\left\{\begin{array}
    {l l}
    0\,,&\mbox{if $a=b=+\infty$ or $a=b=-\infty$}\,,
    \\
    b-a\,,&\mbox{if else}\,,
  \end{array}\right .
  \]
  we see that $\{S_{\tau_M(u_i)}\}_{M>0}$ is a monotone increasing family of sets whose union is $S_{u_i}$, $\{h_M\}_{M>0}$ is an increasing family of functions on $\R^{n-1}$, and that
  \[
  \lim_{M\to\infty} h_M=1_{S_{u_1}\cup S_{u_2}}\,\min\Big\{ 2 (\widetilde{u}_2 - \widetilde{u}_1) , [u_1] + [u_2] \Big\}\,,
  \]
  where the convention that $\widetilde{u_2}-\widetilde{u_1}=0$ if $\widetilde{u_2}=\widetilde{u_1}=+\infty$ was also taken into account; we have thus completed the proof of \eqref{dead14}. Similarly, on noticing that
  \begin{eqnarray*}
  \{\widetilde{\tau_M(u_1)}<\widetilde{\tau_M(u_2)}\}&=&\{f_M(u_1^\vee,u_2^\vee)+f_M(u_1^\wedge,u_2^\wedge)>0\}
  \\
  &=&\{f_M(u_1^\vee,u_2^\vee)>0\}\cup\{f_M(u_1^\wedge,u_2^\wedge)>0\}\,,
  \end{eqnarray*}
  we see that $\{\{\widetilde{\tau_M(u_1)}<\widetilde{\tau_M(u_2)}\}\}_{M>0}$ is a monotone increasing family of sets whose union is $\{u_2^\vee>u_1^\vee\}\cup\{u_2^\wedge>u_1^\wedge\}$. Therefore, by definition of $|D^cu_i|$, we find, for $i=1,2$,
  \begin{eqnarray*}
    \lim_{M\to\infty}|D^c\tau_Mu_i|\Big(G\cap\{\widetilde{\tau_M(u_1)}<\widetilde{\tau_M(u_2)}\}\Big)
    &=&|D^cu_i|\Big(G\cap(\{u_2^\vee>u_1^\vee\}\cup\{u_2^\wedge>u_1^\wedge\})\Big)
    \\
    &=&|D^cu_i|\Big(G\cap\{\widetilde{u}_1<\widetilde{u}_2\}\Big)\,,
  \end{eqnarray*}
  where in the last identity we have taken into account that $S_{u_1}\cup S_{u_2}$ is countably $\H^{n-2}$-rectifiable, and thus $|D^cu_i|$-negligible for $i=1,2$. This proves \eqref{dead13}. Next, we notice that
  \[
  |\nabla \tau_M(u_i)|=1_{\{|u_i|<M\}}\,|\nabla u_i|\,,\qquad\mbox{$\H^{n-1}$-a.e. on $\R^{n-1}$}\,,
  \]
  so that \eqref{dead12} follows again by monotone convergence. By \eqref{dead0} applied to $E_M$ this shows in particular that the limit as $M\to\infty$ of $P(E_M;G\times\R)$ exists in $[0,\infty]$. Thus, in order to prove \eqref{dead11} it suffices to show that $P(E;G\times\R)$ is the limit of $P(E_{M_h};G\times\R)$ as $h\to\infty$, where $\{M_h\}_{h\in\N}$ has been chosen in such a way that
  \begin{equation}
    \label{dead15}
    \lim_{h\to\infty}\H^{n-1}\Big(E^{(1)}\cap\{|x_n|=M_h\}\Big)=0\,,\qquad \H^{n-1}\Big(\pae  E\cap\{|x_n|=M_h\}\Big)=0\,,\qquad\forall h\in\N\,.
  \end{equation}
  (Notice that the choice of  $\{M_h\}_{h\in\N}$ is made possible by the fact that $|E|<\infty$, and since $\H^{n-1}\llcorner\pae E$ is a Radon measure.) Indeed, by  $E_M=E\cap\{|x_n|<M\}$, by \eqref{dead15}, and by \cite[Theorem 16.3]{maggiBOOK}, we have that
  \[
  \pae  E_{M_h}=\Big(\{|x_n|<M_h\}\cap\pae  E\Big)\cup\Big(\{|x_n|=M_h\}\cap E^{(1)}\Big)\,,\qquad\forall h\in\N\,,
  \]
  so that, by the first identity in \eqref{dead15} we find $P(E;G\times\R)=\lim_{h\to\infty}P(E_{M_h};G\times\R)$, as required. This completes the proof of the theorem.
\end{proof}

In practice, we shall always apply Theorem \ref{lemma u1u2} in situations where the sets under consideration are described in terms of their barycenter and slice length functions.

\begin{corollary}\label{lemma W}
  If $v\in (BV\cap L^\infty)(\R^{n-1};[0,\infty))$, $b\in GBV(\R^{n-1})$, and
  \begin{equation}
    \label{definizione W}
      W=W[v,b]=\Big\{x\in\R^n:|\q x- b(\p x)|<\frac{v(\p x)}2\Big\}\,,
  \end{equation}
  then $u_1=b-(v/2)\in GBV(\R^{n-1})$, $u_2=b+(v/2)\in GBV(\R^{n-1})$, $W$ is a set of locally finite perimeter with finite volume, and for every Borel set $G\subset\R^{n-1}$ we have
  \begin{eqnarray}\label{formula W}
    P(W;G\times\R)&=&\int_{G\cap\{v>0\}}\sqrt{1+\Big|\nabla\Big(b+\frac{v}2\Big)\Big|^2}+\sqrt{1+\Big|\nabla\Big(b-\frac{v}2\Big)\Big|^2}\,d\H^{n-1}
    \\\nonumber
    &&+\int_{G\cap(S_v\cup S_b)}\,\min\Big\{v^\vee+v^\wedge,\max\Big\{[v],2\,[b]\Big\}\Big\}\,d\H^{n-2}
    \\\nonumber
    &&+\Big|D^c\Big(b+\frac{v}2\Big)\Big|\Big(G\cap\{\widetilde{v}>0\}\Big)+\Big|D^c\Big(b-\frac{v}2\Big)\Big|\Big(G\cap\{\widetilde{v}>0\}\Big)\,,
  \end{eqnarray}
  where this identity holds true in $[0,\infty]$.
\end{corollary}

\begin{proof} It is easily seen that $(BV\cap L^\infty)+GBV\subset GBV$. By Theorem \ref{lemma u1u2}, $W=\S_{u_1}\cap\S^{u_2}$ is of locally finite perimeter, and $P(W;G\times\R)$ can be computed by means of \eqref{dead0} for every Borel set $G\subset\R^{n-1}$. We are thus left to prove that, $\H^{n-2}$-a.e. on $S_{u_1}\cup S_{u_2}$,
\begin{equation}
  \label{furbini}
\min\Big\{ 2 (\widetilde{u}_2 - \widetilde{u}_1) , [u_1] + [u_2] \Big\}=
\min\Big\{v^\vee+v^\wedge,\max\Big\{[v],2\,[b]\Big\}\Big\}\,.
\end{equation}
On $J_{u_1}\cap J_{u_2}\cap\{\nu_{u_1}=\nu_{u_2}\}$ we have that
\begin{eqnarray*}
  &&b^\vee=\frac{u_1^\vee+u_2^\vee}2\,,\qquad v^\vee=\max\Big\{u_2^\vee-u_1^\vee,u_2^\wedge-u_1^\wedge\Big\}\,,
  \\
  &&b^\wedge=\frac{u_1^\wedge+u_2^\wedge}2\,,\qquad v^\wedge=\min\Big\{u_2^\vee-u_1^\vee,u_2^\wedge-u_1^\wedge\Big\}\,,
\end{eqnarray*}
while on $J_{u_1}\cap J_{u_2}\cap\{\nu_{u_1}=-\nu_{u_2}\}$ we find
\begin{eqnarray*}
  &&b^\vee=\max\Big\{\frac{u_2^\vee+u_1^\wedge}2,\frac{u_2^\wedge+u_1^\vee}2\Big\}\,,\qquad v^\vee=u_2^\vee-u_1^\wedge\,,
  \\
  &&b^\wedge=\min\Big\{\frac{u_2^\vee+u_1^\wedge}2,\frac{u_2^\wedge+u_1^\vee}2\Big\}\,,\qquad  v^\wedge=u_2^\wedge-u_1^\vee\,,
\end{eqnarray*}
so that \eqref{furbini} is proved through an elementary case by case argument on $J_{u_1}\cap J_{u_2}$, and thus, $\H^{n-2}$-a.e. on $S_{u_1}\cap S_{u_2}$. At the same time, on $S_{u_1}\cap S_{u_2}^c$ we have
\begin{eqnarray*}
  &&b^\vee=\frac{\widetilde{u}_2+u_1^\vee}2\,,\qquad v^\vee=\widetilde{u}_2-u_1^\wedge\,,
  \\
  &&b^\wedge=\frac{\widetilde{u}_2+u_1^\wedge}2\,,\qquad v^\wedge=\widetilde{u}_2-u_1^\vee\,,
\end{eqnarray*}
from which we easily deduce \eqref{furbini} on $S_{u_1}\cap S_{u_2}^c$; by symmetry, we notice the validity of \eqref{furbini} on $S_{u_1}^c\cap S_{u_2}$, and thus conclude the proof of the corollary.
\end{proof}

\begin{corollary}\label{lemma W2}
  Let $v:\R^{n-1}\to[0,\infty)$ be Lebesgue measurable. Then, $F[v]$ is of finite perimeter and volume if and only if $v\in BV(\R^{n-1};[0,\infty))$ and $\H^{n-1}(\{v>0\})<\infty$. In this case, if $F=F[v]$, then for every $z\in\R^{n-1}$ we have
  \begin{eqnarray}
   \label{F[v] densita 1}
   \Big(-\frac{v^\wedge(z)}2,\frac{v^\wedge(z)}2\Big)&\subset (F^{(1)})_z&\subset \Big[-\frac{v^\wedge(z)}2,\frac{v^\wedge(z)}2\Big]\,,
   \\
   \label{F[v] paM}
   \Big\{t\in\R: \frac{v^\wedge(z)}2< |t|< \frac{v^\vee(z)}2\Big\}
   &\subset (\pae F)_z&
   \subset \Big\{t\in\R: \frac{v^\wedge(z)}2\le |t|\le \frac{v^\vee(z)}2\Big\}\,,\hspace{0.4cm}
  \end{eqnarray}
  while, for every Borel set $G\subset\R^{n-1}$,
  \begin{equation}
    \label{perimetro di F}
    P(F;G\times\R)=2\int_{G\cap\{v>0\}}\sqrt{1+\Big|\frac{\nabla v}2\Big|^2}\,d\H^{n-1}
    +\int_{G\cap S_v}\,[v]\,d\H^{n-2}+|D^cv|(G)\,.
  \end{equation}
\end{corollary}

\begin{proof} By Proposition \ref{proposition insieme u1u2} and the coarea formula \eqref{coarea bv}, we see that $F[v]$ is of finite perimeter if and only if $v\in BV(\R^{n-1};[0,\infty))$ and $\H^{n-1}(\{v>0\})<\infty$. By arguing as in step three of the proof of Proposition \ref{proposition insieme u1u2}, we easily prove \eqref{F[v] densita 1} and \eqref{F[v] paM}. Finally, by applying Theorem \ref{lemma u1u2} to $u_2=v/2$ and $u_1=-v/2$, we prove \eqref{perimetro di F} with $|D^cv|(G\cap\{\widetilde{v}>0\})$ in place of $|D^cv|(G)$. We conclude the proof of the corollary by Lemma \ref{lemma Dcv}.
\end{proof}

We close this section with the proof of Proposition \ref{proposition dsv}.

\begin{proof}[Proof of Proposition \ref{proposition dsv}]  We want to prove that, if $\l\in[0,1]\setminus\{1/2\}$ and
  \begin{equation}
    \label{malefico*}
      E=\Big\{x\in\R^n: -\l\,v_2(\p x)-\frac{v_1(\p x)}2\le \q x\le \frac{v_1(\p x)}2+(1-\l)\,v_2(\p x)\Big\}\,,
  \end{equation}
  then $E\in\M(v)$ and $\H^n(E\Delta(t\,e_n+F[v]))>0$ for every $t\in\R$. By Corollary \ref{lemma W2},
  \begin{equation}
    \label{perimetro simmetrizzato steiner}
  P(F[v])=2\,\int_{\R^{n-1}}\sqrt{1+\Big|\nabla\Big(\frac{v_1}2\Big)\Big|^2}+|D^sv_2|(\R^{n-1})\,.
  \end{equation}
  At the same time, $E=W[v,b]$, where $b=((1/2)-\l)\,v_2$. Since $D^sv_1=0$, $D^{a}v_2=0$, and $v^\vee+v^\wedge\ge [v]=[v_2]\ge 2[b]$ $\H^{n-2}$-a.e. on $\R^{n-1}$, we easily find
  \begin{eqnarray*}
    &&\nabla\Big(b\pm\frac{v}2\Big)=\pm\nabla\Big(\frac{v_1}2\Big)\,,\qquad\mbox{$\H^{n-1}$-a.e. on $\R^{n-1}$}\,,
    \\
    &&\min\Big\{v^\vee+v^\wedge,\max\Big\{[v],2\,[b]\Big\}\Big\}=[v_2]\,,\qquad\mbox{$\H^{n-2}$-a.e. on $\R^{n-1}$}\,,
    \\
    &&D^c\Big(b+\frac{v}2\Big)=(1-\l)\,D^cv_2\,,\qquad D^c\Big(b-\frac{v}2\Big)=-\l\,D^cv_2\,.
  \end{eqnarray*}
  Since $S_b\cup S_v=_{\H^{n-2}}S_{v_2}$, we find $P(E)=P(F[v])$ by \eqref{perimetro simmetrizzato steiner} and \eqref{formula W}. At the same time,
  \[
  \H^n(E\Delta(t\,e_n+F[v]))=2\,\int_{\{v>0\}}\Big|t-\Big(\frac12-\l\Big)v_2\Big|\,d\H^{n-1}\,,\qquad\forall t\in\R\,,
  \]
  so that $\H^n(E\Delta(t\,e_n+F[v]))>0$ as $\l\ne 1/2$ and $v_2$ is non-constant on $\{v>0\}$.
\end{proof}

\subsection{A fine analysis of the barycenter function}\label{section barycenter} We now prove Theorem \ref{thm tauM b delta}, stating in particular that $b_E\,1_{\{v>\de\}}\in GBV(\R^{n-1})$ whenever $E$ is a $v$-distributed set of finite perimeter and $\{v>\de\}$ is of finite perimeter. We first discuss some examples showing that this is the optimal degree of regularity we can expect for the barycenter. (Let us also recall that the regularity of barycenter functions in arbitrary codimension, but under ``no vertical boundaries'' and ``no vanishing sections'' assumptions, was addressed in \cite[Theorem 4.3]{barchiesicagnettifusco}.)

\begin{remark}\label{remark fifi}
  {\rm In the case $n=2$, as it will be clear from the proof of Theorem \ref{thm tauM b delta}, conclusion \eqref{barycenter sharp regularity} can be strengthened to $1_{\{v>\de\}}\,b_E \in (BV\cap L^\infty)(\R^{n-1})$. The localization on $\{v>\de\}$ is necessary. Indeed, let us define $E\subset\R^2$ as
  \[
  E = \bigcup_{h \in \mathbb{N}}\Big\{x\in\R^2:\frac1{h+1}<\p x<\frac1h\,,\Big|\q x-(-1)^h\Big|<\frac1{h^2}\Big\}\,,
  \]
  so that $E$ has finite perimeter and volume, and has segments as sections. However,
  \[
  b_E(z)=\sum_{h\in\N}(-1)^h\,1_{((h+1)^{-1},h^{-1})}(z)\,,\qquad z\in\R\,,
  \]
  so that $b_E\in L^\infty(\R)\setminus BV(\R)$. We also notice that, in the case $n\ge 3$, the use of generalized functions of bounded variation is necessary. For example, let $E_\a\subset\R^3$ be such that
  \[
  E_{\alpha} = \bigcup_{h \in \mathbb{N}}\Big\{x\in\R^3: \frac1{(h+1)^2}<|\p x|<\frac1{h^2}\,,\Big|\q x-h^\a\Big|<\frac12\Big\}\,,\qquad \a>0\,.
  \]
  In this way, $E_\a$ has always finite perimeter and volume, with $v(z)=1$ if $|z|<1$ and
  \[
  1_{\{v>\de\}}(z)\,b_{E_\a}(z)=b_{E_\a}(z)=\sum_{h\in\N}1_{((h+1)^{-2},h^{-2})}(|z|)\,h^\a\,,\qquad \forall z\in\R^2\,,\,0<\de<1\,.
  \]
  In particular, $1_{\{v>\de\}}\,b_{E_2}\in L^1(\R^2)\setminus BV(\R^2)$ and $1_{\{v>\de\}}\,b_{E_4}\not\in L^1_{loc}(\R^2)$. Hence, without truncation, $1_{\{v>\de\}}\,b_{E}$ may either fail to be of bounded variation (even if it is locally summable), or it may just fail to be locally summable.}
\end{remark}

Before entering into the proof of Theorem \ref{thm tauM b delta}, we shall need to prove that the momentum function $m_E$ of a vertically bounded set $E$ is of bounded variation; see Lemma \ref{lemma Dv} below. Given $E\subset\R^n$, we say that $E$ is vertically bounded (by $M>0$) if
\[
E\subset_{\H^n}\Big\{x\in\R^n:|\q x|<M\Big\}\,.
\]

\begin{lemma}\label{lemma Dv} If $v\in BV(\R^{n-1};[0,\infty))$ and $E$ is a vertically bounded, $v$-distributed set of finite perimeter, then $m_E\in (BV\cap L^\infty)(\R^{n-1})$, where
\[
m_E(z)=\int_{E_z}t\,d\H^1(t)\,,\qquad\forall z\in\R^{n-1}\,.
\]
\end{lemma}

\begin{proof}
If $E$ is vertically bounded by $M>0$, then $v\in L^\infty(\R^{n-1})$, $|m_E|\le M\,v$, and $m_E\in L^\infty(\R^{n-1})$. Moreover, $m_E\in BV(\R^{n-1})$ as, for every $\vphi\in C^1_c(\R^{n-1})$,
\begin{eqnarray*}
  \int_{\R^{n-1}}m_E\,\nabla'\vphi \,d\H^{n-1}&=&\int_{E}\nabla'(\vphi(\p x)\,\q x)\,d\H^n(x)
  =\int_{\pa^*E}\vphi(\p x)\,\q x\,\p\nu_E(x)\,d\H^{n-1}(x)
  \\
  &\le& M\,\sup_{\R^{n-1}}|\vphi|\,P(E)\,.\hspace{6.8cm}\qedhere
\end{eqnarray*}
\end{proof}

\begin{proof}[Proof of Theorem \ref{thm tauM b delta}] {\it Step one:} Let us decompose $z\in\R^{n-1}$ as $z=(z_1,z')\in\R\times\R^{n-2}$. For every fixed $z'\in\R^{n-2}$,  $f:\R^{n-1}\to\R$, $G\subset\R^{n-1}$, and $E\subset\R^n$, we define
\begin{eqnarray*}
f^{z'}:\R\to\R\,,\qquad f^{z'}(z_1)=f(z_1,z')\,,
\\
G^{z'}=\{z_1\in\R:(z_1,z)\in G\}\,,
\\
E^{z'}=\{(z_1,t)\in\R^2:(z_1,z',t)\in E\}\,.
\end{eqnarray*}
We now consider $v$ and $E$ as in the statement, and identify a set $I\subset(0,1)$ such that $\H^1((0,1)\setminus I)=0$ and if $\de\in I$, then $\{v>\de\}$ is a set of finite perimeter. We now fix $\de\in I$, and consider a set $J\subset\R^{n-2}$ (depending on $\de$, and whose existence is a consequence of Theorem \ref{thm volpert} in section \ref{section F indecomposable}) such that $\H^{n-2}(\R^{n-2}\setminus J)=0$ and, for every $z'\in J$, $E^{z'}$ is a set of finite perimeter in $\R^2$ (hence, $v^{z'}\in BV(\R)$) and
\[
\{v>\de\}^{z'}=\{v^{z'}>\de\}
\]
is a set of finite perimeter in $\R$. As we shall see in step three, for every $z'\in J$,
\[
\Big|D\,\Big(\tau_M\Big(1_{\{v^{z'}>\de\}}\,b_{E^{z'}}\Big)\Big)\Big|(\R)\le C(M,\de)\Big\{P\Big(\Big\{v^{z'}>\de\Big\}\Big)+P(E^{z'})\Big\}\,.
\]
If we thus take into account that
\[
\Big(\tau_M\Big(1_{\{v>\de\}}\,b_E\Big)\Big)^{z'}=\tau_M\Big(1_{\{v^{z'}>\de\}}\,b_{E^{z'}}\Big)\,,
\]
we thus conclude that
\begin{eqnarray*}
  &&\int_{\R^{n-2}}\bigg|D\bigg(\Big(\tau_M\Big(1_{\{v>\de\}}\,b_E\Big)\Big)^{z'}\bigg)\bigg|(\R)\,d\H^{n-2}(z')
  \\
  &\le&C(M,\de)\,\int_{\R^{n-2}}\,\Big\{P(\{v^{z'}>\de\})+P(E^{z'})\Big\}\,d\H^{n-2}(z')
  \\
  &\le&C(M,\de)\,\Big\{P(\{v>\de\})+P(E)\Big\}\,,
\end{eqnarray*}
where in the last step we have used \cite[Proposition 14.5]{maggiBOOK}. We can repeat this argument along each coordinate direction in $\R^{n-1}$ and combine it with \cite[Remark 3.104]{AFP} to conclude that $\tau_M(1_{\{v>\de\}}\,b_E)\in (BV\cap L^\infty)(\R^{n-1})$, with
\[
\Big|D\Big(\tau_M\Big(1_{\{v>\de\}}\,b_E\Big)\Big)\Big|(\R^{n-1})\le\,C(M,\de)\,\Big\{P(\{v>\de\})+P(E)\Big\}\,.
\]
The proof of \eqref{barycenter sharp regularity} will then be completed in the following two steps.

\medskip

\noindent {\it Step two:} Let $n=2$. We claim that $P(E^s)<\infty$ implies $v\in L^\infty(\R)$, while $P(E)<\infty$ implies $b_E\in L^\infty(\{v>\s\})$ for every $\s>0$. The first claim follows by Corollary \ref{lemma W2}: indeed, $P(E^s)<\infty$ implies $v\in BV(\R)$, and thus, trivially, $v\in L^\infty(\R)$. To prove the second claim, let us recall from step two in the proof of \cite[Theorem 19.15]{maggiBOOK} that if $a,b\in\R$ are such that $a\ne b$ and
\begin{eqnarray*}
  \H^1(E^{(1)}_a)+\H^1(E^{(1)}_b)<\infty\,,\quad \H^1(E^{(1)}_a\cap E^{(1)}_b)=0\,,
\quad
  \H^1(\pa^*E^{(1)}_a)=\H^1(\pa^*E^{(1)}_b)=0\,,
\end{eqnarray*}
then one has
\begin{equation}
  \label{finedellavoro}
  \H^1(E^{(1)}_a)+\H^1(E^{(1)}_b)\le P(E;\{a<x_1<b\})\,.
\end{equation}
Should $b_E$ fail to be essentially bounded on $\{v>\s\}$ for some $\s>0$, then we may construct a strictly increasing sequence $\{a_h\}_{h\in\N}\subset\R$ with $\s\le \H^1(E^{(1)}_{a_h})<\infty$, $\H^1(\pa^*E^{(1)}_{a_h})=0$, and $\H^1(E^{(1)}_{a_h}\cap E^{(1)}_{a_k})=0$ if $h\ne k$. Correspondingly, by \eqref{finedellavoro}, we would get
\[
2\s\le P(E;\{a_h<x_1<a_{h+1}\})\,,\qquad\forall h\in\N\,,
\]
and thus conclude that $P(E)=+\infty$.

\medskip

\noindent {\it Step three:} Let $v\in BV(\R)$, $E$ be a $v$-distributed set of finite perimeter in $\R^2$ such that $E_z$ is a segment for $\H^1$-a.e. $z\in\R$, and let $\de>0$ be such that $\{v>\de\}$ is a set of finite perimeter in $\R$. According to step one, in order to complete the proof of \eqref{barycenter sharp regularity} we are left to show that, if $M>0$, then
\begin{eqnarray}
  \label{stimetta M delta}
\Big|D\,\Big(\tau_M\Big(1_{\{v>\de\}}\,b_E\Big)\Big)\Big|(\R)\le C(M,\de)\Big\{P(\{v>\de\})+P(E)\Big\}\,.
\end{eqnarray}
By step two, $v\in L^\infty(\R)$ and $b_E\in L^\infty(\{v>\de\})$. In particular, $E$ is vertically bounded above $\{v>\de\}$, that is, there exists $L(\de)>0$ such that
\begin{equation}
  \label{vertically bounded sigma}
  E(\de)=E\cap\Big(\{v>\de\}\times\R\Big)\subset_{\H^2}\Big\{x\in\R^2:v(\p x)>\de\,,|\q x|<L(\de)\Big\}\,.
\end{equation}
Let us now set $v_\de=1_{\{v>\de\}}\,v$. Since $\{v>\de\}$ is of finite perimeter, we have
\[
v_\de\in (BV\cap L^\infty)(\R)\,,\qquad \{v_\de>0\}=\{v>\de\}\,.
\]
Concerning $E(\de)$, we notice that, since $\{v>\de\}\times\R$ is of locally finite perimeter, then $E(\de)$ is, at least, a $v_\de$-distributed set of locally finite perimeter such that $E(\de)_z$ is a segment for $\H^1$-a.e. $z\in\R$. But in fact, \eqref{vertically bounded sigma} implies $\{|x_n|>L(\de)\}\subset E(\de)^{(0)}$, while at the same time we have the inclusion
\[
\pae E(\de)\subset \Big[\pae E\cap\Big(\{v>\de\}^{(1)}\times\R\Big)\Big]\cup\Big[\Big(\pae \{v>\de\}\times\R\Big)\cap\Big(E^{(1)}\cup\pae E\Big)\Big]\,;
\]
in particular, $E(\de)$ is of finite perimeter by Federer's criterion, as
\[
\H^{n-1}(\pae E(\de))\le P(E;\{v>\de\}^{(1)}\times\R)+2\,L(\de)\,P(\{v>\de\})\,.
\]
We now notice that $b_{E(\de)}=1_{\{v>\de\}}\,b_E\in L^\infty(\R)$, with $P(E(\de);\{v>\de\}^{(1)}\times\R)\le P(E)$; hence, \eqref{stimetta M delta} follows by showing
\[
\Big|D\,\Big(\tau_M\,(b_{E(\de)})\Big)\Big|(\R)\le C(M,\de)\Big\{P(\{v_\de>0\})+P\Big(E(\de);\{v_\de>0\}^{(1)}\times\R\Big)\Big\}\,,
\]
for every $M>0$. It is now convenient to reset notation.

\medskip

\noindent {\it Step four:} As shown in step three, the proof of \eqref{barycenter sharp regularity} is completed by showing that if $v\in (BV\cap L^\infty)(\R)$ is such that, for some $\de>0$, $\{v>0\}=\{v>\de\}$ is a set of finite perimeter in $\R$, and $E$ is a vertically bounded, $v$-distributed set of finite perimeter in $\R^2$ with $b_E\in L^\infty(\R)$, then, for every $M>0$,
\begin{eqnarray}
  \label{stimetta M delta3}
|D\,(\tau_M\,(b_E))|(\R)\le C(M,\de)\Big\{P(\{v>0\})+P\Big(E;\{v>0\}^{(1)}\times\R\Big)\Big\}\,.
\end{eqnarray}
We start by noticing that, since $E$ is vertically bounded, then by Lemma \ref{lemma Dv} we have $m_E\in (BV\cap L^\infty)(\R)$. Moreover, if we set
\[
w=\frac{1_{\{v>0\}}}v=\frac{1_{\{v>\de\}}}v\,,
\]
then we have $w \in (BV\cap L^\infty)(\R)$, and thus $b_E=w\,m_E\in (BV\cap L^\infty)(\R)$. We now notice that, since $\{v=0\}\subset\{\tau_M(b_E)=0\}$, we have $\{v=0\}^{(1)}\subset\{\tau_M(b_E)=0\}^{(1)}$; at the same time, a simple application of the co-area formula shows that
\begin{eqnarray}
  \nonumber
  0&=&|D(\tau_M(b_E))|\Big(\{\tau_M(b_E)=0\}^{(1)}\Big)\ge|D(\tau_M(b_E))|\Big(\{v=0\}^{(1)}\Big)
  \\\label{clarinetto1}
  &=&|D(\tau_M(b_E))|\Big(\{v>0\}^{(0)}\Big)\,.
\end{eqnarray}
Moreover, since $\{v>0\}$ is a set of finite perimeter, we know that $\pae \{v>0\}$ is a finite set, so that
\begin{eqnarray}
  |D(\tau_M(b_E))|\Big(\pae \{v>0\}\Big)=\int_{S_{\tau_M(b_E)}\cap\pae \{v>0\}}[\tau_M(b_E)]\,d\H^0
  \label{clarinetto2}
  \le2M\,P(\{v>0\})\,,
\end{eqnarray}
where we have noticed that $[\tau_M(b_E)]\le 2M$ since $|\tau_M(b_E)|\le M$ on $\R^{n-1}$. By \eqref{clarinetto1} and \eqref{clarinetto2}, in order to achieve \eqref{stimetta M delta3}, we are left to prove that
\begin{eqnarray}
  \label{stimetta M delta3k}
|D\,(\tau_M\,(b_E))|\Big(\{v>0\}^{(1)}\Big)\le C(M,\de)\,P\Big(E;\{v>0\}^{(1)}\times\R\Big)\,.
\end{eqnarray}
By \eqref{dens3} and by $\{v>\de\}=\{v>0\}$ we have
\[
\{v^\wedge>0\}\subset\{v>0\}^{(1)}=\{v>\de\}^{(1)}\subset\{v^\wedge\ge\de\}\subset \{v^\wedge>0\}\,,
\]
that is $\{v>0\}^{(1)}=\{v^\wedge>0\}$. By applying Corollary \ref{lemma W} to $G=\{v>0\}^{(1)}=\{v^\wedge>0\}$,
\begin{eqnarray}\label{clarinetto5}
    &&P(E;\{v>0\}^{(1)}\times\R)
    \\\nonumber
    &=&\int_{\{v>0\}}\sqrt{1+\Big|\Big(b_E+\frac{v}2\Big)'\Big|^2}+\sqrt{1+\Big|\Big(b_E-\frac{v}2\Big)'\Big|^2}\,d\H^1
    \\\nonumber
    &&+\int_{\{v>0\}^{(1)}\cap(S_v\cup S_{b_E})}\,\min\Big\{v^\vee+v^\wedge,\max\Big\{[v],2\,[b_E]\Big\}\Big\}\,d\H^0
    \\\nonumber
    &&+\Big|D^c\Big(b_E+\frac{v}2\Big)\Big|\Big(\{v^\wedge>0\}\cap\{\widetilde{v}>0\}\Big)+\Big|D^c\Big(b_E-\frac{v}2\Big)\Big|\Big(\{v^\wedge>0\}\cap
    \{\widetilde{v}>0\}\Big)\,.
\end{eqnarray}
Since $\{v^\wedge=0\}=\{\widetilde{v}=0\}\cup\{v^\vee>0=v^\wedge\}$ where $\{v^\vee>0=v^\wedge\}\subset_{\H^0}J_v$, we find that $\{v^\wedge=0\}$ is $|D^cf|$-equivalent to $\{\widetilde{v}=0\}$ for every $f\in BV_{loc}(\R^{n-1})$: hence,
\begin{equation}
  \label{clarinetto4}
  \Big|D^c\Big(b_E\pm\frac{v}2\Big)\Big|\Big(\{v^\wedge>0\}\cap\{\widetilde{v}>0\}\Big)
=\Big|D^c\Big(b_E\pm\frac{v}2\Big)\Big|\Big(\{v^\wedge>0\}\Big)\,.
\end{equation}
By \eqref{clarinetto5}, \eqref{clarinetto4}, the triangular inequality, and as $v^\wedge\ge\de$ on $\{v>0\}^{(1)}=\{v>\de\}^{(1)}$,
\begin{eqnarray}\nonumber
    P(E;\{v>0\}^{(1)}\times\R)&\ge&2\int_{\{v>0\}}|b_E'|\,d\H^1
    +2\,\int_{\{v>0\}^{(1)}\cap S_{b_E}}\,\min\{\de,[b_E]\}\,d\H^0
    \\\label{bri1}
    &&+2\,|D^cb_E|(\{v^\wedge>0\})\,.
  \end{eqnarray}
At the same time, by \cite[Theorem 3.99]{AFP}, for every $M>0$ we have
\begin{eqnarray}\label{bbq}
&&|D(\tau_M(b_E))|(\{v>0\}^{(1)})
\\\nonumber
&=&\int_{\{|b_E|<M\}\cap\{v>0\}}\,|b_E'|\,d\H^1+|D^cb_E|\Big(\Big\{|\widetilde{b}_E|<M\Big\}\cap\{v>0\}^{(1)}\Big)
\\\nonumber
&&+ \int_{S_{b_E}\cap\{b_E^\wedge<M\}\cap\{b_E^\vee>-M\}\cap \{v>0\}^{(1)}}\,\min\{M,b_E^\vee\}-\max\{-M,b_E^\wedge\}\,d\H^0\,.
\end{eqnarray}
As it is easily seen by arguing on a case by case basis,
\begin{equation}
  \label{bri2}
  \min\{M,b_E^\vee\}-\max\{-M,b_E^\wedge\}\le \max\Big\{1,\frac{2M}\de\Big\}\,\min\{\de,[b_E]\}\,,\qquad\mbox{on $S_{b_E}$}\,.
\end{equation}
By combining \eqref{bri1}, \eqref{bbq}, and \eqref{bri2} we conclude the proof of \eqref{stimetta M delta3k}, thus of step four. The proof of \eqref{barycenter sharp regularity} is now complete.

\medskip

\noindent {\it Step five:} Since $\{v>\de\}$ is of finite perimeter for a.e. $\de>0$, we find that, correspondingly, $b_\de=1_{\{v>\de\}}\,b_E\in GBV(\R^{n-1})$ for a.e. $\de>0$. In particular, $b_\de$ is approximately differentiable at $\H^{n-1}$-a.e. $x\in\R^{n-1}$. Since $b_\de=b_E$ on $\{v>\de\}$, by \eqref{gino2} it follows that
\begin{equation}\label{gradienti bdelta be}
  \nabla b_E(x)=\nabla b_\de(x)\,,\qquad \mbox{$\H^{n-1}$-a.e. $x\in\{v>\de\}$}\,.
\end{equation}
By considering $\de_h\to 0$ as $h\to\infty$ with $\{v>\de_h\}$ of finite perimeter for every $h\in\N$, we find that $b_E$ is approximately differentiable at $\H^{n-1}$-a.e. $x\in \{v>0\}$. Since, trivially, $b_E$ is approximately differentiable at {\it every} $x\in\{v=0\}^{(1)}$ with $\nabla b_E(x)=0$, we conclude that $b_E$ is approximately differentiable at $\H^{n-1}$-a.e. $x\in\R^{n-1}$. By \cite[Theorem 4.34]{AFP}, for every Borel $G\subset\R^{n-1}$ we have
\begin{equation}
    \label{coarea-a-ah delta}
      \int_{\R}\H^{n-2}(G\cap\pae \{b_\de>t\})\,dt=\int_G|\nabla b_\de|\,d\H^{n-1}+\int_{G\cap S_{b_\de}}[b_\de]\,d\H^{n-2}+|D^cb_\de|(G)\,.
\end{equation}
Let us notice that, by \eqref{gino1}, $[b_\de]=[b_E]$ on $\{v>\de\}^{(1)}$, and thus $S_{b_\de}\cap\{v>\de\}^{(1)}=S_{b_E}\cap\{v>\de\}^{(1)}$. By \eqref{gradienti bdelta be} and by applying \eqref{coarea-a-ah delta} to $G\cap\{v>\de\}^{(1)}$ where $G\subset\R^{n-1}$ is a Borel set, we find
\begin{eqnarray}
    \label{coarea-a-ah delta 1}
      &&\int_{\R}\H^{n-2}(G\cap\{v>\de\}^{(1)}\cap\pae \{b_\de>t\})\,dt
      \\\nonumber
      &=&\int_{G\cap\{v>\de\}}|\nabla b_E|\,d\H^{n-1}
      +\int_{G\cap S_{b_E}\cap\{v>\de\}^{(1)}}[b_E]\,d\H^{n-2}
      +|D^cb_\de|(G\cap\{v>\de\}^{(1)})\,.
\end{eqnarray}
Since $\tau_M b_\de=1_{\{v>\de\}}\,\tau_M b_\de$, by applying Lemma \ref{lemma fgE} we find that, for every $G\subset\R^{n-1}$,
\begin{equation}
  \label{lalalla}
  |D^cb_\de|(G\cap\{v>\de\}^{(1)})= \lim_{M\to \infty} |D^c \tau_Mb_\de|(G\cap\{v>\de\}^{(1)})= \lim_{M\to \infty} |D^c\tau_Mb_\de|(G)=
  |D^cb_\de|(G)\,.
\end{equation}
At the same time, since $\{v>\de\}\cap\{b_\de>t\}=\{v>\de\}\cap\{b_E>t\}$ for every $t\in\R$, we have
\[
\{v>\de\}^{(1)}\cap\pae \{b_\de>t\}=\{v>\de\}^{(1)}\cap\pae  \{b_E>t\}\,,\qquad\forall t\in\R\,,
\]
and thus
\[
\int_\R\H^{n-2}(G\cap\{v>\de\}^{(1)}\cap\pae \{b_\de>t\})\,dt=\int_\R\,\H^{n-2}(G\cap\{v>\de\}^{(1)}\cap\pae \{b_E>t\})\,dt\,.
\]
If we now set $\de=\de_h$ into \eqref{coarea-a-ah delta 1} and then let $h\to\infty$, then by
\begin{equation}
  \label{herbert0}\{v^\wedge>0\}=\bigcup_{h\in\N}\{v>\de_h\}^{(1)}\,,
\end{equation}
(which follows by \eqref{dens3}), by \eqref{lalalla}, and thanks to the definition \eqref{DcbE} of $|D^cb_E|^+$, we find that \eqref{coarea-a-ah} holds true for every Borel set $G\subset\{v^\wedge>0\}$, as required. We have thus completed the proof of Theorem \ref{thm tauM b delta}.
\end{proof}

\subsection{Characterization of equality cases, part one}\label{section v in bv} In this section we prove the necessary conditions for equality cases in Steiner's inequality stated in Theorem \ref{thm characterization with barycenter}. We premise the following simple lemma to the proof.

\begin{lemma}\label{lemma misure triangolare}
  If $\mu$ and $\nu$ are $\R^{n-1}$-valued Radon measures on $\R^{n-1}$, then
  \begin{equation}
    \label{triangular measures}
      2\,|\mu|(G)\le|\nu+\mu|(G)+|\nu-\mu|(G)\,,
  \end{equation}
  for every Borel set $G\subset\R^{n-1}$. Moreover, equality holds in \eqref{triangular measures} for every bounded Borel set $G\subset\R^{n-1}$ if and only if there exists a Borel function $f:\R^{n-1}\to[-1,1]$ with
  \[
  \nu(G)=\int_G\,f\,d\mu\,,\qquad\mbox{for every bounded Borel set $G\subset\R^{n-1}$.}
  \]
\end{lemma}

\begin{proof}
  The validity of \eqref{triangular measures} follows immediately from the fact that if $G$ is a Borel set in $\R^{n-1}$, then $|\mu|(G)$ is the supremum of the sums $\sum_{h\in\N}|\mu(G_h)|$ along partitions $\{G_h\}_{h\in\N}$ of $G$ into bounded Borel sets. From the same fact, we immediately deduce that $|\nu+\mu|(G)=|\nu-\mu|(G)=|\nu|(G)$ whenever $|\mu|(G)=0$: therefore, if $G$ is such that $|\mu|(G)=0$ and \eqref{triangular measures} holds as an equality, then $|\nu|(G)=0$. In particular, if equality holds in \eqref{triangular measures} for every bounded Borel set $G\subset\R^{n-1}$, then $|\nu|$ is absolutely continuous with respect to $|\mu|$. By the Radon-Nykod\'im theorem we have that $\nu=g\,d|\mu|$ for a $|\mu|$-measurable function $g:\R^{n-1}\to\R^{n-1}$, as well as $\mu=h\,d|\mu|$, for a $|\mu|$-measurable function $h:\R^{n-1}\to S^{n-2}$.
  In particular $\nu\pm\mu=(g\pm h)\,d|\mu|$, and thus, since equality holds in \eqref{triangular measures},
  \begin{eqnarray*}
    2|\mu|(G)=|\nu+\mu|(G)+|\nu-\mu|(G)
    =\int_G\,|g+h|\,d|\mu|+\int_G\,|g-h|\,d|\mu|\,,
  \end{eqnarray*}
  for every Borel set $G\subset\R^{n-1}$,  which gives
  \[
  |g+h|+|h-g|=2=2|h|\,,\qquad\mbox{$|\mu|$-a.e. on $\R^{n-1}$}\,.
  \]
  Thus, there exists $\l:\R^{n-1}\to[0,\infty)$ such that $(h-g)=\l\,(g+h)$ $|\mu|$-a.e. on $\R^{n-1}$, i.e.
  \[
  g=\frac{1-\l}{1+\l}\,h\,,\qquad\mbox{$|\mu|$-a.e. on $\R^{n-1}$}\,.
  \]
  This proves that $\nu=f\,d\mu$ where $f=(1-\l)/(1+\l)$. By Borel regularity of $|\mu|$, we can assume without loss of generality that $f$ is Borel measurable. The proof is complete.
\end{proof}

\begin{proof}
  [Proof of Theorem \ref{thm characterization with barycenter}, necessary conditions]  Let  $E\in\M(v)$. By Theorem \ref{thm ccf1}, we have that $E_z$ is $\H^1$-equivalent to a segment for $\H^{n-1}$-a.e. $z\in\R^{n-1}$, that is \eqref{finale0}. As a consequence, by Theorem \ref{thm tauM b delta}, we have $b_\de=1_{\{v>\de\}}\,b_E\in GBV(\R^{n-1})$ whenever $\{v>\de\}$ is of finite perimeter. Let us set
  \begin{eqnarray}
  \label{defn:I}
    I&=&\Big\{\de>0:\mbox{$\{v>\de\}$ and $\{v<\de\}$ are sets of finite perimeter}\Big\}\,,
    \\
    \label{defn:J}
    J_\de&=&\Big\{M>0:\mbox{$\{b_\de< M\}$ and $\{b_\de> -M\}$ are sets of finite perimeter}\Big\}\,,
  \end{eqnarray}
  and notice that $\H^1((0,\infty)\setminus I)=0$ since $v\in BV(\R^{n-1})$, and that $\H^1((0,\infty)\setminus J_\de)=0$ for every $\de\in I$, as  $b_\de\in GBV(\R^{n-1})$ whenever $\de\in I$. By taking total variations in \eqref{finale3 e}, we find $2\,|D^c(\tau_Mb_\de)|(G)\le|D^cv|(G)$ for every bounded Borel set $G\subset\R^{n-1}$. By letting first $M\to\infty$ (in $J_\de$) and then $\de\to 0$ (in $I$) we prove \eqref{2DcbE le Dcv}. Let us also notice that \eqref{coarea-a-ah ottimali} is an immediate corollary of \eqref{coarea-a-ah} and \eqref{2DcbE le Dcv}, once \eqref{finale1 E}, \eqref{finale2 E}  have been proved. Summarizing, these remarks show that we only need to prove the validity of \eqref{finale1 E}, \eqref{finale2 E}, and \eqref{finale3 e} (for $\de\in I$ and $M\in J_\de$) in order to complete the proof of the necessary conditions for equality cases. This is accomplished in various steps.

  \medskip

  \noindent {\it Step one:} Let us fix $\de,L\in I$ and $M\in J_\de$, and set
  \[
  \S_{\de,L,M}=\{\de<v<L\}\cap\{|b_E|<M\}=\{|b_\de|<M\}\cap\{ \de < v<L\}\,,
  \]
  so that $\S_{\de,L,M}$  is a set of finite perimeter. Since $\tau_Mb_\de\in (BV\cap L^\infty)(\R^{n-1})$ (see the end of step one in the proof of Theorem \ref{thm tauM b delta}), $1_{\S_{\de,L,M}}\in (BV\cap L^\infty)(\R^{n-1})$, and $\tau_Mb_\de=b_\de=b_E$ on $\S_{\de,L,M}$, we have
  \[
  b_{\de,L,M}=1_{\S_{\de,L,M}}\,b_E\in (BV\cap L^\infty)(\R^{n-1})\,.
  \]
  We now claim that there exists a Borel function $f_{\de,L,M}:\R^{n-1}\to[-1/2,1/2]$ such that
  \begin{eqnarray}
  \label{bdelta ac xx}
  \nabla b_{\de,L,M}(z)=0\,,&&\quad\mbox{for $\H^{n-1}$-a.e. $z\in\S_{\de,L,M}$}\,,
  \\
  \label{bdelta cantor xx}
  D^cb_{\de,L,M}(G)=\int_G f_{\de,L,M}\,d\,(D^cv)\,,&&\quad\mbox{for every bounded Borel set $G\subset\S_{\de,L,M}^{(1)}$}\,.\hspace{0.5cm}
 \end{eqnarray}
 Indeed, let us set $v_{\de,L,M}=1_{\S_{\de,L,M}}\,v$. Since $v_{\de,L,M}\,,b_{\de,L,M}\in (BV\cap L^\infty)(\R^{n-1})$, we can apply Corollary \ref{lemma W} to $W=W[v_{\de,L,M},b_{\de,L,M}]$. Since $W[v_{\de,L,M},b_{\de,L,M}]=E\cap(\S_{\de,L,M}\times\R)$, and thus
\[
\pae E\cap(\S_{\de,L,M}^{(1)}\times\R)=\pae W[v_{\de,L,M},b_{\de,L,M}]\cap(\S_{\de,L,M}^{(1)}\times\R)\,,
\]
we find that, for every Borel set $G\subset\S_{\de,L,M}^{(1)}\setminus(S_{v_{\de,L,M}}\cup S_{b_{\de,L,M}})$,
\begin{eqnarray}\nonumber
    P(E;G\times\R)&=&P(W[v_{\de,L,M},b_{\de,L,M}];G\times\R)
    \\\nonumber
    &=&\int_{G}\sqrt{1+\Big|\nabla\Big(b_{\de,L,M}+\frac{v_{\de,L,M}}2\Big)\Big|^2}
    +\sqrt{1+\Big|\nabla\Big(b_{\de,L,M}-\frac{v_{\de,L,M}}2\Big)\Big|^2}\,d\H^{n-1}
    \\\label{sveglia2*}
    &&+\Big|D^c\Big(b_{\de,L,M}+\frac{v_{\de,L,M}}2\Big)\Big|(G)+\Big|D^c\Big(b_{\de,L,M}-\frac{v_{\de,L,M}}2\Big)\Big|(G)\,.
  \end{eqnarray}
  By Lemma \ref{lemma fgE} applied with $v_{\de,L,M}=1_{\S_{\de,L,M}}v$, we find that
  \begin{eqnarray*}
  \nabla v_{\de,L,M}&=&1_{\S_{\de,L,M}}\,\nabla v\,,\qquad\mbox{$\H^{n-1}$-a.e. on $\R^{n-1}$}\,,
  \\
  D^cv_{\de,L,M}&=&D^cv\llcorner\S_{\de,L,M}^{(1)}\,,
  \\
  S_{v_{\de,L,M}}\cap\S_{\de,L,M}^{(1)}&=&S_v\cap\S_{\de,L,M}^{(1)}\,.
  \end{eqnarray*}
  By \eqref{sveglia2*}, we thus find that
  \begin{eqnarray}\nonumber
    P(E;G\times\R)&=&\int_{G}\sqrt{1+\Big|\nabla\Big(b_{\de,L,M}+\frac{v}2\Big)\Big|^2}
    +\sqrt{1+\Big|\nabla\Big(b_{\de,L,M}-\frac{v}2\Big)\Big|^2}\,d\H^{n-1}
    \\\label{sveglia2**}
    &&+\Big|D^c\Big(b_{\de,L,M}+\frac{v}2\Big)\Big|(G)+\Big|D^c\Big(b_{\de,L,M}-\frac{v}2\Big)\Big|(G)\,,
\end{eqnarray}
for every  Borel set $G\subset \S_{\de,L,M}^{(1)}\setminus(S_v\cup S_{b_{\de,L,M}})$. By Corollary \ref{lemma W2}, for every Borel set $G\subset\R^{n-1}$,
\begin{eqnarray}\label{sveglia3*}
    P(F[v];G\times\R)=2\,\int_{G}\sqrt{1+\Big|\frac{\nabla v}2\Big|^2}\,d\H^{n-1}
    +\int_{G\cap S_v}\,[v]\,d\H^{n-2}
    +|D^cv|(G)\,.
\end{eqnarray}
Taking into account that $P(E;G\times\R)=P(F[v];G\times\R)$ for every Borel set $G\subset\R^{n-1}$, we combine \eqref{sveglia2**} and \eqref{sveglia3*}, together with the convexity of $\xi\in\R^{n-1}\mapsto\sqrt{1+|\xi|^2}$ and \eqref{triangular measures}, to find that, if $G\subset \S_{\de,L,M}^{(1)}\setminus(S_v\cup S_{b_{\de,L,M}})$, then
\begin{equation}\label{G ac}
  0=\int_{G}\sqrt{1+\Big|\nabla\Big(b_{\de,L,M}+\frac{v}2\Big)\Big|^2}
  +\sqrt{1+\Big|\nabla\Big(b_{\de,L,M}-\frac{v}2\Big)\Big|^2}-2\,\sqrt{1+\Big|\frac{\nabla v}2\Big|^2}\,d\H^{n-1}\,,
\end{equation}
\begin{equation}\label{G cantor}
  0=\Big|D^c\Big(b_{\de,L,M}+\frac{v}2\Big)\Big|(G)+\Big|D^c\Big(b_{\de,L,M}-\frac{v}2\Big)\Big|(G)-|D^cv|(G)\,.
\end{equation}
Since $\S_{\de,L,M}^{(1)}\setminus(S_v\cup S_{b_{\de,L,M}})$ is $\H^{n-1}$-equivalent to $\S_{\de,L,M}$, by \eqref{G ac} and by the strict convexity of $\xi\in\R^{n-1}\mapsto\sqrt{1+|\xi|^2}$, we come to prove \eqref{bdelta ac xx}. By applying Lemma \ref{lemma misure triangolare} to
\[
\mu=\frac{D^cv}2\,,\qquad \nu=D^cb_{\de,L,M}\llcorner \Big(\S_{\de,L,M}^{(1)}\setminus(S_v\cup S_{b_{\de,L,M}})\Big)
=D^cb_{\de,L,M}\llcorner \S_{\de,L,M}^{(1)}\,,
\]
we prove \eqref{bdelta cantor xx}. This completes the proof of \eqref{bdelta ac xx} and \eqref{bdelta cantor xx}.

\medskip

\noindent {\it Step two:} We prove \eqref{finale3 e}. Let $\de,L\in I$ and $M\in J_\de$. Since $b_{\de,L,M}=1_{\S_{\de,L,M}}\tau_Mb_\de$, by Lemma \ref{lemma fgE} we have
\[
D^cb_{\de,L,M}=D^c(\tau_Mb_\de)\llcorner \S_{\de,L,M}^{(1)}\,.
\]
We combine this fact with \eqref{bdelta cantor xx} to find a Borel function $f_{\de,M}:\R^{n-1}\to[-1/2,1/2]$ with
\[
D^c\tau_Mb_\de(G)=\int_G f_{\de,M}\,d\,(D^cv)\,,\qquad\mbox{for every bounded Borel set $G\subset\S_{\de,L,M}^{(1)}$}\,.
\]
As a consequence, the Radon measures $D^c\tau_Mb_\de$ and $f_{\de,M}\,D^cv$ coincide on every bounded Borel set contained into
\begin{eqnarray*}
\bigcup_{L\in I}\,\S_{\de,L,M}^{(1)}&=&\bigcup_{L\in I}\,\{v>\de\}^{(1)}\cap\{|b_E|<M\}^{(1)}\cap\{v<L\}^{(1)}
\\
&=&\Big(\{v>\de\}^{(1)}\cap\{|b_E|<M\}^{(1)}\Big)\cap \bigcup_{L\in I}\,\{v<L\}^{(1)}
\\
&=&\{v>\de\}^{(1)}\cap\{|b_E|<M\}^{(1)}\cap \{v^\vee<\infty\}\,,
\end{eqnarray*}
where in the last identity we have used \eqref{dens2}. Since $\H^{n-2}(\{v^\vee=\infty\})=0$ by \cite[4.5.9(3)]{FedererBOOK}, the set $\{v^\vee=\infty\}$ is negligible with respect to both $|D^c\tau_Mb_\de|$ and $|D^cv|$. We have thus proved that, for every bounded Borel set $G \subset \{v>\de\}^{(1)}\cap\{|b_E|<M\}^{(1)}$,
\begin{equation}
\label{rachel}
 D^c(\tau_Mb_\de)(G)=\int_G f_{\de,M}\,d\,(D^cv)\,.
\end{equation}
Since for every $M'>M$ and $\de'<\de$ we have that $\tau_Mb_\de = \tau_{M'}b_{\de'}$ on  $\{v>\de\}\cap\{|b_E|<M\}$, by Lemma \ref{lemma Dcv} we obtain that
$$
D^c(\tau_Mb_\de)\llcorner \{v>\de\}^{(1)}\cap\{|b_E|<M\}^{(1)} = D^c(\tau_{M'} b_{\de'})\llcorner \{v>\de\}^{(1)}\cap\{|b_E|<M\}^{(1)}\,,
$$
and therefore \eqref{rachel} can be rewritten with a function $f$ independent of $M$ and $\de$; thus,
\begin{equation}
\label{rachelnew}
 D^c(\tau_Mb_\de)(G)=\int_G  f\,d\,(D^cv)\,,
\end{equation}
for every bounded Borel set $G \subset \{v>\de\}^{(1)}\cap\{|b_E|<M\}^{(1)}$. We next notice that, if $\de\in I$ and $M\in J_\de$, then
  \[
  \tau_M\,b_\de=M\,1_{\{b_\de\ge M\}}-M\,1_{\{b_\de\le -M\}}+1_{\{|b_\de|<M\}\cap\{v>\de\}}\tau_Mb_\de\,,\qquad\mbox{on $\R^{n-1}$}\,,
  \]
  is an identity between $BV$ functions. By \cite[Example 3.97]{AFP} we thus find
  \begin{eqnarray}
  \nonumber
  D^c\tau_Mb_\de&=&D^c\Big(1_{\{|b_\de|<M\}\cap\{v>\de\}}\tau_Mb_\de\Big)=1_{(\{|b_\de|<M\}\cap\{v>\de\})^{(1)}}D^c(\tau_Mb_\de)
  \\\label{donna}
  &=&D^c(\tau_Mb_\de)\llcorner\Big(\{|b_\de|<M\}^{(1)}\cap\{v>\de\}^{(1)}\Big)\,.
  \end{eqnarray}
Since by \eqref{donna} the measure $D^c(\tau_Mb_\de)$ is concentrated on $ \{v>\de\}^{(1)}\cap\{|b_E|<M\}^{(1)}$,
we deduce from \eqref{rachelnew} that for every bounded Borel set $G \subset \R^{n-1}$
\[
 D^c(\tau_Mb_\de)(G) =  D^c(\tau_Mb_\de)\Big(G \cap \{v>\de\}^{(1)}\cap\{|b_E|<M\}^{(1)}\Big) =\int_{G\cap\{v>\de\}^{(1)}\cap\{|b_E|<M\}^{(1)}} \hspace{-2cm}f\,d\,(D^cv)\,,
\]
which proves \eqref{finale3 e}.

\medskip

\noindent {\it Step three:} We prove \eqref{finale1 E}. Let $\de,L\in I$ and $M\in J_\de$. Since $b_{\de,L,M}=b_E$ on $\S_{\de,L,M}$, by \eqref{bdelta ac xx} and by \eqref{gino2} we find that $\nabla b_E=0$ $\H^{n-1}$-a.e. on $\S_{\de,L,M}$. By taking a union first on $M\in J_\de$, and then on $\de,L\in I$, we find that $\nabla b_E=0$ $\H^{n-1}$-a.e. on $\{v>0\}$. At the same time $b_E=0$ on $\{v=0\}$ by definition, and thus, again by \eqref{gino2}, we have $\nabla b_E=0$ $\H^{n-1}$-a.e. on $\{v=0\}$. This completes the proof of \eqref{finale1 E}.

\medskip

\noindent {\it Step four:} We prove \eqref{finale2 E}. We fix $\de,L\in I$ and define $\S_{\de,L}=\{\de<v<L\}$, $b_{\de,L}=1_{\S_{\de,L}}\,b_E$, and $v_{\de,L}=1_{\S_{\de,L}}\,v$. Since $\S_{\de,L}$ is a set of finite perimeter, it turns out that $b_{\de,L}\in GBV(\R^{n-1})$, while, by construction, $v_{\de,L}\in (BV\cap L^\infty)(\R^{n-1})$. We are in the position to apply Corollary \ref{lemma W} to obtain a formula for the perimeter of $W[v_{\de,L},b_{\de,L}]$ relative to cylinders $G\times\R$ for Borel sets $G\subset\R^{n-1}$. In particular, if $G\subset \S_{\de,L}^{(1)}\cap( S_{v_{\de,L}}\cup S_{b_{\de,L}})$, then
\begin{eqnarray*}
  P(E;G\times\R)&=&P(W[v_{\de,L},b_{\de,L}];G\times\R)
  \\
  &=&\int_{G}\,\min\Big\{v_{\de,L}^\vee+v_{\de,L}^\wedge,\max\Big\{[v_{\de,L}],2\,[b_{\de,L}]\Big\}\Big\}\,d\H^{n-2}\,.
\end{eqnarray*}
Since, by \eqref{gino1}, $\S_{\de,L}^{(1)}\cap S_{v_{\de,L}}=\S_{\de,L}^{(1)}\cap S_v$ with
$v_{\de,L}^\vee=v^\vee$, $v_{\de,L}^\wedge=v^\wedge$, and $[v_{\de,L}]=[v]$ on $\S_{\de,L}^{(1)}$, we have
\begin{eqnarray*}
  P(E;G\times\R)=\int_{G}\,\min\Big\{v^\vee+v^\wedge,\max\Big\{[v],2\,[b_{\de,L}]\Big\}\Big\}\,d\H^{n-2}\,,
\end{eqnarray*}
whenever $G\subset \S_{\de,L}^{(1)}\cap (S_v\cup S_{b_{\de,L}})$. Since $P(E;G\times\R)=P(F[v];G\times\R)$, by \eqref{sveglia3*},
\[
\min\Big\{v^\vee+v^\wedge,\max\Big\{[v],2\,[b_{\de,L}]\Big\}\Big\}=[v]\,,\qquad
\mbox{$\H^{n-2}$-a.e. on $(S_{b_{\de,L}}\cup S_v)\cap\S_{\de,L}^{(1)}$}\,.
\]
Since $v^\wedge\ge\de$ on $\S_{\de,L}^{(1)}$, we deduce that $v^\vee+v^\wedge>[v]$ on $\S_{\de,L}^{(1)}$, and thus the above condition immediately implies that
\[
2\,[b_{\de,L}]\le [v]\,,\qquad\mbox{$\H^{n-2}$-a.e. on $(S_{b_{\de,L}}\cup S_v)\cap\S_{\de,L}^{(1)}$}\,.
\]
In particular, $S_{b_{\de,L}}\cap\S_{\de,L}^{(1)}\subset_{\H^{n-2}} S_v$, and we have proved
\[
2\,[b_{\de,L}]\le[v]\,,\qquad\mbox{$\H^{n-2}$-a.e. on $\S_{\de,L}^{(1)}$}\,.
\]
By \eqref{gino1}, $[b_{\de,L}]=[b_E]$ on $\S_{\de,L}^{(1)}$. By taking the union of $\S_{\de,L}^{(1)}$ on $\de, L\in I$, and by taking \eqref{dens2} and \eqref{dens3} into account, we thus find that
\[
2\,[b_E]\le[v]\,,\qquad\mbox{$\H^{n-2}$-a.e. on $\{v^\wedge>0\}\cup\{v^\vee<\infty\}$}\,.
\]
Since, as noticed above, $\{v^\vee=\infty\}$ is $\H^{n-2}$-negligible, we have proved \eqref{finale2 E}.
\end{proof}

\subsection{Characterization of equality cases, part two}\label{section necessary cond barycenter} We now complete the proof of Theorem \ref{thm characterization with barycenter}, by showing that if a $v$-distributed set of finite perimeter $E$ satisfies \eqref{finale0}, \eqref{finale1 E}, \eqref{finale2 E}, and \eqref{finale3 e}, then $E\in\M(v)$. The following proposition will play a crucial role.

\begin{proposition}\label{corollario v uguale 0}
  If $v\in BV(\R^{n-1};[0,\infty))$, $\H^{n-1}(\{v>0\})<\infty$, and $E$ is a $v$-distributed set of finite perimeter with section as segments, then
  \begin{equation}
    \label{superpippo}
      P(E;\{v^\wedge=0\}\times\R)=P(F[v];\{v^\wedge=0\}\times\R)=\int_{\{v^\wedge=0\}}v^\vee\,d\H^{n-2}\,.
  \end{equation}
\end{proposition}

\begin{remark}
  {\rm With Proposition \ref{corollario v uguale 0} at hand, one can actually go back to Corollary \ref{lemma W} and obtain a formula for $P(E;G\times\R)$ in terms of $v$ and $b_E$ whenever $E$ is a $v$-distributed set of finite perimeter with section as segments. Since such a formula may be of independent interest, we have included its proof in Appendix \ref{section perimeter formula}.}
\end{remark}

\begin{proof}[Proof of Proposition \ref{corollario v uguale 0}]
  Let $I=\{t>0:\mbox{$\{v>t\}$ and $\{v<t\}$ is of finite perimeter}\}$, so that we have as usual $\H^1((0,\infty)\setminus I)=0$. Since
  \[
  \int_0^\infty P(\{v>t\})\,dt=\int_0^\infty P(\{v< t\})\,dt=|Dv|(\R^{n-1})<\infty\,,
  \]
  we can find two sequences $\{\de_h\}_{h\in\N},\{L_h\}_{h\in\N}\subset I$ such that
  \begin{eqnarray}
    \label{superpippo0}
    \lim_{h\to\infty}\de_h=0\,,&&\qquad \lim_{h\to\infty}\de_h\,P(\{v>\de_h\})=0\,,
    \\\label{superpippo0x}
    \lim_{h\to\infty}L_h=\infty\,,&&\qquad    \lim_{h\to\infty}L_h\,P(\{v<L_h\})=0\,.
  \end{eqnarray}
  Let us set $\S_h=\{L_h>v>\de_h\}$ and $E_h=E\cap(\S_h\times\R)$. Notice that $E_h$ is, trivially, a set of locally finite perimeter. Since $E_h$ locally converges to $E$ as $h\to\infty$, since $P(E_h;\S_h^{(0)}\times\R)=0$, and since $\pae E_h\cap(\S_h^{(1)}\times\R)=\pae E\cap(\S_h^{(1)}\times\R)$ we have
  \begin{equation}
    \label{superpippo1}
    P(E)\le \liminf_{h\to\infty}P(E_h)=\liminf_{h\to\infty}P(E;\S_h^{(1)}\times\R)+P(E_h;\pae\S_h\times\R)\,.
  \end{equation}
  By \eqref{dens2} and \eqref{dens3},
  \[
  \lim_{h\to\infty}1_{\S_h^{(1)}}(z)=1_{\{v^\wedge>0\}\cap\{v^\vee<\infty\}}(z)\,,\qquad\forall z\in\R^{n-1}\,,
  \]
  so that by dominated convergence and thanks to the fact that $E$ has finite perimeter,
  \[
  \lim_{h\to\infty}P(E;\S_h^{(1)}\times\R)= P\Big(E;\Big(\{v^\wedge>0\}\cap\{v^\vee<\infty\}\Big)\times\R\Big)=P(E;\{v^\wedge>0\}\times\R)\,.
  \]
  (In the last identity we have first used \cite[4.5.9(3)]{FedererBOOK} to infer that $\H^{n-2}(\{v^\vee=\infty\})=0$, and then \cite[2.10.45]{FedererBOOK} to conclude that $\H^{n-1}(\{v^\vee=\infty\}\times\R)=0$.) Hence, by \eqref{superpippo1},
  \begin{equation}
    \label{superpippo2}
    P(E;\{v^\wedge=0\}\times\R)\le\liminf_{h\to\infty}P(E_h;\pae\S_h\times\R)\,.
  \end{equation}
  Since $\de_h,L_h\in I$, we have $v_h=1_{\S_h}\,v\in (BV\cap L^\infty)(\R^{n-1})$ and $a_h=1_{\S_h}\,b_E\in GBV(\R^{n-1})$ (indeed, $a_h=1_{\{v<L_h\}}\,b_{\de_h}$ where $b_{\de_h}=1_{\{v>\de_h\}}\,b_E\in GBV(\R^{n-1})$ thanks to Theorem \ref{thm tauM b delta}). Since $E_h=W[v_h,a_h]$ according to \eqref{definizione W}, we can apply \eqref{formula W} in Corollary \ref{lemma W} to $G=\pae\S_h$ to find that
  \begin{equation}
    \label{superpippo4}
      P(E_h;\pae\S_h\times\R)=\int_{\pae\S_h\cap(S_{v_h}\cup S_{a_h})}\min\Big\{v_h^\vee+v_h^\wedge,\max\{[v_h],2[a_h]\}\Big\}\,d\H^{n-2}\,.
  \end{equation}
  (Notice that, since $\pae\S_h$ is countably $\H^{n-2}$-rectifiable, we are only interested in the ``jump'' contribution in \eqref{formula W}.) Let us now set
  \[
  K_h^1=\pae\S_h\cap\pae\{v>\de_h\}\,,\qquad K_h^2=\pae\S_h\setminus\pae\{v>\de_h\}\subset \pae\{v<L_h\}\,.
  \]
  The key remark to exploit \eqref{superpippo4} is that, as one can check with standard arguments,
  \begin{eqnarray}\label{superpippo3}
  v_h^\vee=v^\vee\ge\de_h\ge v^\wedge\,,\qquad v_h^\wedge=0\,,\qquad\mbox{$\H^{n-2}$-a.e. on $K_h^1$}\,,
  \\\label{superpippo3x}
  v^\vee\ge L_h \ge v^\wedge=v_h^\vee\,,\qquad v_h^\wedge=0\,,\qquad\mbox{$\H^{n-2}$-a.e. on $K_h^2$}\,.
  \end{eqnarray}
  (For example, in order to prove \eqref{superpippo3x}, we argue as follows. First, we notice that we always have $v^\vee\ge L_h\ge v^\wedge$ and $v_h^\wedge=0$ on $\pae\{v<L_h\}$. In particular, $\widetilde{v}=L_h$ on $S_v^c\cap \pae\{v<L_h\}$, and this immediately implies $v_h^\vee=L_h$ on $S_v^c\cap\pae\{v< L_h\}$. Finally, by taking into account that $v_h=1_{\S_h}v$ with $\S_h\subset\{v<L_h\}$, one checks that $v^\wedge=v_h^\vee$ $\H^{n-2}$-a.e. on $J_v\cap\pa^*\{v<L_h\}$.) By \eqref{superpippo3} and \eqref{superpippo3x} we have
  \begin{eqnarray}
    \min\{v_h^\vee+v_h^\wedge,\max\{[v_h],2[a_h]\}\}=v^\vee\,,&&\qquad \mbox{$\H^{n-2}$-a.e. on $K_h^1$}\,,
    \\
    \min\{v_h^\vee+v_h^\wedge,\max\{[v_h],2[a_h]\}\}=v^\wedge\,,&&\qquad\mbox{$\H^{n-2}$-a.e. on $K_h^2$}\,,
  \end{eqnarray}
  so that, by \eqref{superpippo4}, and since, again by \eqref{superpippo3}, $K_h^1\subset_{\H^{n-2}}S_{v_h}$, we find
  \begin{equation}
    \label{superpippo4x}
      P(E_h;\pae\S_h\times\R)\le\int_{K_h^1}v^\vee\,d\H^{n-2}+\int_{K_h^2}v^\wedge\,d\H^{n-2}\,.
  \end{equation}
  By \eqref{superpippo3x} and \eqref{superpippo0x}, we have
  \begin{equation}
    \label{superpippoX1}
      \limsup_{h\to\infty}\int_{K_h^2}v^\wedge\,d\H^{n-2}\le  \limsup_{h\to\infty} L_h\,\H^{n-2}(K_h^2)
  \le  \limsup_{h\to\infty} L_h\,P(\{v<L_h\})=0\,.
  \end{equation}
  We are now going to prove that
  \begin{equation}
        \label{superpippoX2}
        \lim_{h\to\infty}\int_{\pae\{v>\de_h\}}v^\vee\,d\H^{n-2}=\int_{\{v^\wedge=0\}}v^\vee\,d\H^{n-2}\,.
  \end{equation}
  (This will be useful in the estimate of the right-hand side of \eqref{superpippo4} as $K_h^1\subset \pae\{v>\de_h\}$.) Since $\{v^\wedge=0\}\cap\pae\{v>\de_h\}=\{v^\wedge=0\}\cap S_v\cap\pae\{v>\de_h\}=\{v^\wedge=0\}\cap\{[v]\ge\de_h\}$, we have that, monotonically as $h\to\infty$,
  \[
  v^\vee\,1_{\{v^\wedge=0\}\cap\pae\{v>\de_h\}}\to v^\vee\,1_{\{v^\wedge=0\}\cap S_v}\,,\qquad\mbox{pointwise on $\R^{n-1}$.}
  \]
  Hence,
  \begin{equation}
    \label{superpippo6}
      \lim_{h\to\infty}\int_{\{v^\wedge=0\}\cap\pae\{v>\de_h\}}v^\vee\,d\H^{n-2}=\int_{\{v^\wedge=0\}\cap S_v}v^\vee\,d\H^{n-2}=\int_{\{v^\wedge=0\}}v^\vee\,d\H^{n-2}\,.
  \end{equation}
  We now claim that
  \begin{equation}
    \label{superpippo7}
          \lim_{h\to\infty} \int_{\{v^\wedge>0\}\cap\pae\{v>\de_h\}}v^\vee\,d\H^{n-2}=0\,.
  \end{equation}
  Indeed, since $v^\vee=v^\wedge=\de_h$ on $S_v^c\cap\pae\{v>\de_h\}$, we find that
  \[
    \int_{S_v^c\cap \{v^\wedge>0\}\cap\pae\{v>\de_h\}}v^\vee\,d\H^{n-2}\le \de_h\,\H^{n-2}(\pae\{v>\de_h\})=\de_h\,P(\{v>\de_h\})\,,
  \]
  so that, by \eqref{superpippo0},
  \begin{eqnarray}\nonumber
    &&\limsup_{h\to\infty} \int_{\{v^\wedge>0\}\cap \pae\{v>\de_h\}}v^\vee\,d\H^{n-2}
    \\\nonumber
    &=&\limsup_{h\to\infty} \int_{S_v\cap\{v^\wedge>0\}\cap\pae\{v>\de_h\}}v^\vee\,d\H^{n-2}
    \\\nonumber
    &=&\limsup_{h\to\infty} \int_{S_v\cap\{v^\wedge>0\}\cap\pae\{v>\de_h\}}[v]+v^\wedge\,d\H^{n-2}
    \\\nonumber
    {\small \mbox{(by \eqref{superpippo3})}}\qquad
    &\le&\limsup_{h\to\infty} \int_{S_v\cap\{v^\wedge>0\}\cap\pae\{v>\de_h\}}[v]\,d\H^{n-2}+ \de_h\,\H^{n-2}(\pae\{v>\de_h\})
    \\\label{superpippo8}
    {\small \mbox{(by \eqref{superpippo0})}}\qquad&=&\limsup_{h\to\infty} \int_{S_v\cap\{v^\wedge>0\}\cap\pae\{v>\de_h\}}[v]\,d\H^{n-2}\,.
  \end{eqnarray}
  Now, if $z\in\{v^\wedge>0\}$, then $z\in\{v>\de\}^{(1)}$ for every $\de<v^\wedge(z)$, so that,
  \[
  1_{S_v\cap\{v^\wedge>0\}\cap\pae\{v>\de_h\}}\to 0\,,\qquad\mbox{pointwise on $\R^{n-1}$}\,,
  \]
  as $h\to\infty$. Since $[v] \in L^1(\H^{n-2}\llcorner S_v)$, by dominated convergence we find
  \begin{equation}
    \label{superpippo9}
    \lim_{h\to\infty}\int_{S_v\cap\{v^\wedge>0\}\cap\pae\{v>\de_h\}}[v]\,d\H^{n-2}=0\,.
  \end{equation}
  By combining \eqref{superpippo8} and \eqref{superpippo9}, we prove \eqref{superpippo7}. By \eqref{superpippo6} and \eqref{superpippo7}, we deduce \eqref{superpippoX2}. By $K_h^1\subset \pae\{v>\de_h\}$, \eqref{superpippo4x}, \eqref{superpippoX1}, and \eqref{superpippoX2}, we deduce that
  \[
  \limsup_{h\to\infty}P(E_h;\pae\S_h\times\R)\le \int_{\{v^\wedge=0\}}v^\vee\,d\H^{n-2}\,.
  \]
  By combining this last inequality with \eqref{superpippo2} we find
  \begin{eqnarray*}
    P(E;\{v^\wedge=0\}\times\R)&\le&\int_{\{v^\wedge=0\}}v^\vee\,d\H^{n-2}
    \\
    &=&P(F[v];\{v^\wedge=0\}\times\R)\le P(E;\{v^\wedge=0\}\times\R)\,,
  \end{eqnarray*}
  where the identity follows by \eqref{perimetro di F}, and the final inequality is, of course, \eqref{steiner inequality}. This completes the proof of \eqref{superpippo}.
\end{proof}

\begin{remark}\label{remark filippo omega}
  {\rm Let $v\in BV(\R^{n-1};[0,\infty))$ with $\H^{n-1}(\{v>0\})<\infty$, and let $E$ be a $v$-distributed set with segments as section. Then, $E$ is of finite perimeter if and only if $\sup_{h\in\N}P(E_h)<\infty$, where
  \[
  E_h=E\cap(\S_h\times\R)\,,\qquad \S_h=\{L_h>v>\de_h\}\,,
  \]
  and $\{\de_h\}_{h\in\N},\{L_h\}_{h\in\N}\subset(0,\infty)$ are such that
  \begin{eqnarray*}
    \lim_{h\to\infty}\de_h=0\,,&&\qquad \lim_{h\to\infty}\de_h\,P(\{v>\de_h\})=0\,,
    \\
    \lim_{h\to\infty}L_h=\infty\,,&&\qquad    \lim_{h\to\infty}L_h\,P(\{v<L_h\})=0\,.
  \end{eqnarray*}
  The fact that $P(E)<\infty$ implies $\sup_{h\in\N}P(E_h)<\infty$ is implicit in the proof of Proposition \ref{corollario v uguale 0}. Conversely, if $\{E_h\}_{h\in\N}$ is defined as above, then $E_h\to E$ as $h\to\infty$, and thus $\sup_{h\in\N}P(E_h)<\infty$ implies $P(E)<\infty$ by lower semicontinuity of perimeter.}
\end{remark}

\begin{lemma}\label{voglioCinfinito}
If $v\in (BV \cap L^\infty) (\R^{n-1})$, $b:\R^{n-1}\to\R$ is such that $\tau_M b \in (BV\cap L^\infty)(\R^{n-1})$ for a.e. $M>0$, and $\mu$ is a   $\R^{n-1}$-valued Radon measure such that
\begin{equation}
\label{eqn:totalvar}
\lim_{M\to \infty}  |\mu - D^c \tau_M b|(G)=0\,, \qquad \mbox{for every bounded Borel set $G \subseteq \R^{n-1}$\,,}
\end{equation}
then,
\begin{equation}\label{vitalitheo}
|D^c (b+v)|(G) \leq |\mu+D^cv|(G)\,, \qquad \mbox{for every Borel set $G \subseteq \R^{n-1}$.}
\end{equation}
\end{lemma}

\begin{proof}
Let us assume that $|v| \leq L$ $\H^{n-1}$-a.e. on $\R^{n-1}$. If $f\in BV (\R^{n-1})$, then
\[
\tau_M\,f=M\,1_{\{f>M\}}-M\,1_{\{f<-M\}}+1_{\{|f|<M\}}\tau_Mf\in (BV\cap L^\infty)(\R^{n-1})\,,
\]
for every $M$ such that $\{f>M\}$ and $\{f<-M\}$ are of finite perimeter, and thus, by \cite[Example 3.97]{AFP},
\[
D^c\tau_Mf=D^c\Big(1_{\{|f|<M\}}\tau_Mf\Big)=1_{\{|f|<M\}^{(1)}}D^c(\tau_Mf)=D^c(\tau_Mf)\llcorner\{|f|<M\}^{(1)}\,;
\]
in particular,
\begin{equation} \label{eqn:tauf}
|D^c\tau_Mf| = |D^cf | \llcorner \{|f|<M\}^{(1)} \leq |D^cf |\, .
\end{equation}
From the equality $\tau_M(\tau_{M+L}(b)+v) = \tau_M(b+v)$ and from \eqref{eqn:tauf} applied with $f=\tau_{M+L}(b)+v$ it follows that, for every Borel set $G \subseteq \R^{n-1}$,
\begin{equation}\label{seminario}
 |D^c(\tau_M(b+v))| (G) = |D^c(\tau_M(\tau_{M+L}(b)+v))|(G) \leq
|D^c( \tau_{M+L}(b)+v)|(G)\, .
\end{equation}
By \eqref{eqn:totalvar},
$$
\lim_{M\to \infty}  |D^c( \tau_{M+L}(b)+v)|(G) = |\mu+D^cv|(G)\,.
$$
We let $M\to \infty$ in \eqref{seminario}, and by definition of $|D^c(b+v)|$ we obtain \eqref{vitalitheo}.
\end{proof}

\begin{proof}
  [Proof of Theorem \ref{thm characterization with barycenter}, sufficient conditions] Let $E$ be a $v$-distributed set of finite perimeter satisfying \eqref{finale0}, \eqref{finale1 E}, \eqref{finale2 E}, and \eqref{finale3 e}. Let $I$ and $J_\de$ be defined as in \eqref{defn:I} and \eqref{defn:J}. If $\de,S \in I$ and we set $b_{\de,S}=1_{\{\de<v<S\}}\,b_E =1_{\{\de<v<S\}}\,b_\de$, then, for every $M\in J_\de$, we have $\tau_Mb_\de\in (BV\cap L^\infty)(\R^{n-1})$ (see the end of step one in the proof of Theorem \ref{thm tauM b delta}), and so we obtain that $\tau_M b_{\de,S} \in (BV\cap L^\infty)(\R^{n-1})$. Let us consider the $\R^{n-1}$-valued Radon measure $\mu_{\de,S}$ on $\R^{n-1}$ defined as
  \[
  \mu_{\de,S}(G) =  \int_{G\cap\{\de<v<S\}^{(1)} \cap \{|b_E|^\vee <\infty\}} f \, dD^cv\,,
  \]
  for every bounded Borel set $G \subset \R^{n-1}$. Since $\tau_Mb_{\de,S}= 1_{\{v<S\}}\tau_Mb_\de$, by Lemma \ref{lemma fgE} we have $D^c [\tau_Mb_{\de,S}]= 1_{\{v<S\}^{(1)}}D^c [\tau_Mb_\de]$, and thus, for every Borel set $G \subset \R^{n-1}$,
\begin{eqnarray*}
 \lim_{M\to \infty}   |\mu_{\de,S} - D^c [\tau_M b_{\de,S}]|(G)
 &=& \lim_{M\to \infty}  |\mu_{\de,S} - D^c [\tau_M b_\de]| (G \cap\{v<S\}^{(1)})
 \\
 &\le&  \lim_{M\to \infty} \int_{G\cap\{\de<v<S\}^{(1)}\cap[\{|b_E|^\vee <\infty\}\setminus \{|b_E|<M\}^{(1)}]}\hspace{-1cm} |f| \, d|D^cv| = 0 \,,
 \end{eqnarray*}
where the inequality follows by \eqref{finale3 e}, and the last equality follows from the fact that $\{\{|b_E|<M\}^{(1)}\}_{M\in I}$ is an increasing family of sets whose union is $\{|b_E|^\vee <\infty\}$. By applying Lemma~\ref{voglioCinfinito} to $b_{\de,S}$ and $\pm v_{\de,S}/2$ (with $v_{\de,S}=1_{\{\de<v<S\}}\,v$), and Lemma \ref{lemma misure triangolare} to $\mu_{\de,S}$ and $\pm D^c v_{\de,S}/2$ and having \eqref{finale3 e} in mind, we find that, for every bounded Borel set $G \subset \R^{n-1}$,
\begin{eqnarray}\nonumber
 \Big|D^c\Big(b_{\de,S}+\frac{v_{\de,S}}2\Big)\Big|(G)+\Big|D^c\Big(b_{\de,S}-\frac{v_{\de,S}}2\Big)\Big|(G)
&\le& \Big|\mu_{\de,S}+\frac{D^cv_{\de,S}}2\Big|(G) + \Big|\mu_{\de,S}-\frac{D^cv_{\de,S}}2\Big|(G)\\
\label{stef3}
&=&|D^cv_{\de,S}|(G)\,.
\end{eqnarray}
Since $b_{\de,S}\in GBV(\R^{n-1})$ and $v_{\de,S}\in (BV\cap L^\infty)(\R^{n-1})$, if $W=W[v_{\de,S},b_{\de,S}]$, then we can compute $P(W;G\times\R)$ for every Borel set $G\subset\R^{n-1}$ by Corollary \ref{lemma W}. In particular, if $G \subset \{\de<v<S\}^{(1)}$, then by $E\cap(\{\de<v<S\}\times\R)=W\cap(\{\de<v<S\}\times\R)$, we find that
  \begin{eqnarray}\nonumber
    \hspace{-1cm}P(E;G\times\R)&=&   P(W;G\times\R)\\
    \label{stef}
    &=&\int_{G}\sqrt{1+\Big|\nabla\Big(b_{\de,S}+\frac{v_{\de,S}}2\Big)\Big|^2}
    +\sqrt{1+\Big|\nabla\Big(b_{\de,S}-\frac{v_{\de,S}}2\Big)\Big|^2}\,d\H^{n-1}
    \\\label{stef6}
    &&+\int_{G\cap(S_{v_{\de,S}}\cup S_{b_{\de,S}})}\,\min\Big\{v^\vee_{\de,S}+v^\wedge_{\de,S},\max\Big\{[v_{\de,S}],2\,[b_{\de,S}]\Big\}\Big\}\,d\H^{n-2}
    \\\label{stef99}
    &&+\Big|D^c\Big(b_{\de,S}+\frac{v_{\de,S}}2\Big)\Big|(G)+\Big|D^c\Big(b_{\de,S}-\frac{v_{\de,S}}2\Big)\Big|(G)\,.
\end{eqnarray}
We can also compute $P(F[v_{\de,S}];G\times\R)$ thorough Corollary \ref{lemma W2}. Taking also into account that $F[v]\cap(\{\de<v<S\}\times\R)=F[v_{\de,S}]\cap(\{\de<v<S\}\times\R)$, we thus conclude that
\begin{eqnarray*}
P(F;G\times\R) &=& P(F[v_{\de,S}];G\times\R)\\
   &=&2\,\int_{G}\sqrt{1+\Big|\frac{\nabla v_{\de,S}}2\Big|^2}\,d\H^{n-1}
    +\int_{G\cap S_{v_{\de,S}}}\,[v_{\de,S}]\,d\H^{n-2}
    +|D^cv_{\de,S}|(G)\,. \hspace{0.8cm}
\end{eqnarray*}
From \eqref{finale1 E} and \eqref{finale2 E} we deduce that (applying \eqref{gino1} and \eqref{gino2} to $b_E$ and $v$)
\begin{eqnarray}\label{stef2}
  \nabla b_{\de,S}(z)= \nabla b_E=0\,,&&\qquad\mbox{for $\H^{n-1}$-a.e. $z\in  \{\de<v<S\}$}\,,
\\
\label{stef5}
2\,[b_{\de,S}] =2\, [b_E] \le [v]=[v_{\de,S}]\,,&&\qquad\mbox{$\H^{n-2}$-a.e. on $ \{\de<v<S\}^{(1)}$}\,.
\end{eqnarray}
Substituting \eqref{stef2} into \eqref{stef}, \eqref{stef5} into \eqref{stef6},  and \eqref{stef3} into \eqref{stef99}, we find that
\begin{equation}\label{stef7}
  P(E;\{\de<v<S\}^{(1)}\times\R)\le P(F;\{\de<v<S\}^{(1)}\times\R)\,,
\end{equation}
where, in fact, equality holds thanks to \eqref{steiner inequality}. By \eqref{dens3} it follows that
  \begin{equation}\label{herbert}
   \bigcup_{M\in I}\,\{v<M\}^{(1)} = \{v^\vee<\infty\}=_{\H^{n-2}}\R^{n-1}\,,
  \end{equation}
  as $\H^{n-2}(\{v^\vee=\infty\})=0$ by \cite[4.5.9(3)]{FedererBOOK}. By taking a union over $\de_h\in I$ and $S_h\in I$ such that $\de_h\to 0$ and $S_h\to\infty$ as $h\to\infty$, we deduce from \eqref{stef7}, \eqref{herbert0} and \eqref{herbert} that
  \[
  P(E;\{v^\wedge>0\}\times\R)=P(F;\{v^\wedge>0\}\times\R)\,.
  \]
  By Proposition \ref{corollario v uguale 0}, $P(E;\{v^\wedge=0\}\times\R)=P(F;\{v^\wedge=0\}\times\R)$, and thus $P(E)=P(F)$, as required.
\end{proof}

\subsection{Equality cases by countably many vertical translations}\label{section proof charact Mv per v sbv} We finally address the problem of characterizing the situation when equality cases are necessarily obtained by countably many vertical translations of parts of $F[v]$, see \eqref{E traslato a pezzi}. In particular, we want to show this situation is characterized by the assumptions that $v\in SBV(\R^{n-1};[0,\infty))$ with $\H^{n-1}(\{v>0\})<\infty$ and $S_v$ locally $\H^{n-2}$-rectifiable. We shall need the following theorem:

\begin{theorem}\label{thm afp cool}
  Let $u:\R^{n-1}\to\R$ be Lebesgue measurable. Equivalently,
  \begin{enumerate}
    \item[(i)] $u\in GBV(\R^{n-1})$ with $|D^cu|=0$, $\nabla u=0$ $\H^{n-1}$-a.e. on $\R^{n-1}$, and $S_u$ locally $\H^{n-2}$-finite;
    \item[(ii)] there exist an at most countable set $I$, $\{c_h\}_{h\in I}\subset\R$, and a partition $\{G_h\}_{h\in I}$ of $\R^{n-1}$ into Borel sets, such that
  \begin{eqnarray}\label{uGh}
          u=\sum_{h\in I}c_h\,1_{G_h}\,,\qquad\mbox{$\H^{n-1}$-a.e. on $\R^{n-1}$}\,,
  \end{eqnarray}
  and $\sum_{h\in I}P(G_h\cap B_R)<\infty$ for every $R>0$.
  \end{enumerate}
  Moreover, if we assume that $c_h\ne c_k$ for $h\ne k\in I$, then in both cases
  \begin{eqnarray}\label{gomez1}
    S_u\subset_{\H^{n-2}}\bigcup_{h\ne k\in I}\pae G_h\cap\pae G_k\,,
  \end{eqnarray}
  with $[u]=|c_h-c_k|$ $\H^{n-2}$-a.e. on $\pae G_h\cap\pae G_k$. In particular,
  \[
      \sum_{h\in I}P(G_h;B_R)=2\H^{n-2}(S_u\cap B_R)\,,\qquad\forall R>0\,.
  \]
\end{theorem}

\begin{proof}[Proof of Theorem \ref{thm afp cool}] {\it Step one:} We recall that, by \cite[Definitions 4.16 and 4.21, Theorem 4.23]{AFP}, for every
open set $\Omega$ and $u\in L^\infty(\Omega)$, the following two conditions are equivalent:
  \begin{enumerate}
    \item[(j)] there exist an at most countable set $I$, $\{c_h\}_{h\in I}\subset\R$,
    a partition $\{G_h\}_{h\in I}$ of $\Omega$ such that $\sum_{h\in I}P(G_h; \Omega)<\infty$, and
        \begin{eqnarray}\label{uGh2}
          u=\sum_{h\in I}c_h\,1_{G_h}\,,\qquad\mbox{$\H^{n-1}$-a.e. on $\Omega$}\,.
        \end{eqnarray}

    \item[(jj)] $u\in BV_{loc}(\Omega)$, $Du=Du\llcorner S_u$, and $\H^{n-2}(S_u \cap \Omega)<\infty$
  \end{enumerate}
  In both cases, we have $2\H^{n-2}(S_u \cap \Omega)=\sum_{h\in I}P(G_h ; \Omega)$.

  \medskip

  \noindent {\it Step two:} Let us prove that (i) implies (ii).
  Let $u\in GBV(\R^{n-1})$ with $|D^cu|=0$, $\nabla u=0$ $\H^{n-1}$-a.e. on $\R^{n-1}$, and $S_u$ locally $\H^{n-2}$-finite. For every $R,M>0$, we have,
  by the definition of $GBV$, that
 $\tau_Mu\in BV(B_R)$.
 Moreover, $|D^c \tau_Mu|=0$, $\nabla \tau_Mu=0$, and $S_{\tau_Mu} \cap B_R \subset B_R\cap S_u$ is $\H^{n-2}$-finite.
 By step one, there exist an at most countable set $I_{R,M}$, $\{c_{R,M,h}\}_{h\in I_{R,M}}\subset\R$, and a partition $\{G_{R,M,h}\}_{h\in I_{R,M}}$ of $B_R$ into sets of finite perimeter such that $\sum_{h\in I_{R,M}}P(G_{R,M,h}; B_R)<\infty$,
 and
  \[
 \tau_Mu=\sum_{h\in I_{R,M}}c_{R,M,h}\,1_{G_{R,M,h}}\,,\qquad\mbox{$\H^{n-1}$-a.e. on $B_R$}\,.
  \]
  By a simple monotonicity argument we find \eqref{uGh}. By \eqref{uGh}, if we set $J_M=\{h\in\N:|c_h|\le M\}$, then, $\H^{n-1}$-a.e. on $\R^{n-1}$,
  \begin{eqnarray}\label{zanzare}
 \tau_M\,u=M\,1_{\{u> M\}\cap B_R}-M\,1_{\{u<-M\}\cap B_R}+\sum_{h\in J_M}c_h\,1_{G_h\cap B_R}\,,\quad\mbox{$\H^{n-1}$-a.e. on $B_R$}\,.
   \end{eqnarray}
  By step one,
  \[
  P(\{u> M\}; B_R)+P(\{u<- M\}; B_R)+\sum_{h\in J_M}P(G_h; B_R)=2\,\H^{n-2}(S_{\tau_Mu} \cap B_R)\,.
  \]
  Thus,
  \[
  \sum_{h\in J_M}P(G_h; B_R) \leq 2\,\H^{n-2}(S_{\tau_Mu} \cap B_R)
  \leq 2\,\H^{n-2}(S_{u} \cap B_R)\,.
  \]
Since $\bigcup_{M>0}J_M=I$, letting $M$ go to $\infty$, we find that $\sum_{h\in I}P(G_h;B_R)<\infty$,
which clearly implies $\sum_{h\in I}P(G_h \cap B_R)<\infty$.

  \medskip

  \noindent {\it Step three:} We prove that (ii) implies (i). We easily see that, for every $R,M > 0$,
  $\tau_Mu$ satisfies the assumptions (jj) in step one in $B_R$. Thus, $\tau_Mu \in BV (B_R)$
  with $D\tau_Mu=D\tau_Mu\llcorner S_{\tau_Mu}$ in $B_R$, and
  \[
  2\H^{n-2}(S_{\tau_Mu} \cap B_R)=\sum_{h\in J_{M}}P(G_h;B_R)\le \sum_{h\in I}P(G_h\cap B_R)<\infty \, ,
  \]
  where, as before,  $J_M=\{h\in\N:|c_h|\le M\}$.
 This shows that $u\in GBV(\R^{n-1})$ with $|D^cu|=0$ and $\nabla u=0$ $\H^{n-1}$ a.e. on $\R^{n-1}$.
 Since $\cup_{M > 0} S_{\tau_M u} = S_u$, this immediately implies that $S_u$ is locally $\H^{n-2}$-finite.

  \medskip

  \noindent {\it Step four:} We now complete the proof of the theorem. Since $\{G_h\}_{h\in I}$ is an at most countable Borel partition of $\R^{n-1}$ with $\sum_{h\in\N}P(G_h\cap B_R)<\infty$, we have that
  \[
  \R^{n-1}=_{\H^{n-2}}\bigcup_{h\in I}G_h^{(1)}\cup \bigcup_{h\ne k\in I}\pae G_h\cap \pae G_k\,,
  \]
  compare with \cite[Theorem 4.17]{AFP}. Since $S_u\cap G_h^{(1)}=\emptyset$ for every $h\in I$, this proves \eqref{gomez1}. If we now exploit the fact that, for every $h\ne k\in I$, $c_h\ne c_k$, $G_h$ and $G_k$ are disjoint sets of locally finite perimeter, then, by a blow-up argument we easily see that $[u]=|c_h-c_k|$ $\H^{n-2}$-a.e. on $\pae G_h\cap\pae G_k$ as required. This completes the proof of theorem.
\end{proof}

\begin{proof}
  [Proof of Theorem \ref{thm mv per v sbv}]  {\it Step one:} We prove that if $E\in\M(v)$, then there exist a finite or countable set $I$, $\{c_h\}_{h\in I}\subset\R$, and $\{G_h\}_{h\in I}$ a $v$-admissible partition of $\{v>0\}$, such that $b_E=\sum_{h\in I}c_h\,1_{G_h}$ $\H^{n-1}$-a.e. on $\R^{n-1}$ (so that $E$ satisfies \eqref{E traslato a pezzi}, see Remark \ref{remark checca}), $|D^cb_E|^+=0$, and $2[b_E]\le [v]$ $\H^{n-2}$-a.e. on $\{v^\wedge>0\}$. The last two properties of $b_E$ follow immediately by Theorem \ref{thm characterization with barycenter} since $D^cv=0$. We now prove that $b_E=\sum_{h\in I}c_h\,1_{G_h}$ $\H^{n-1}$-a.e. on $\R^{n-1}$. Let $\de>0$ be such that $\{v>\de\}$ is a set of finite perimeter, and let $b_\de=1_{\{v>\de\}}\,b_E$. By Theorem \ref{thm tauM b delta} and by \eqref{finale1 E}, \eqref{finale2 E}, and \eqref{2DcbE le Dcv}, taking into account \eqref{gino1}, \eqref{gino2} and the definition of $|D^cb_E|^+$ we have that $b_\de\in GBV(\R^{n-1})$ with
  \begin{eqnarray}
  \label{finale1v}
  \nabla b_\de(z)=0\,,&&\qquad\mbox{for $\H^{n-1}$-a.e. $z\in\{v>\de\}$}\,,
  \\
  \label{finale2v}
  2[b_\de]\le[v]\,,&&\qquad\mbox{$\H^{n-2}$-a.e. on $\{v>\de\}^{(1)}$}\,,
  \\\label{finale3v}
  2|D^cb_\de|(G)\le|D^cv|(G)\,,&&\qquad\mbox{for every Borel set $G\subset \R^{n-1}$.}
  \end{eqnarray}
  Since $D^cv=0$ we have that $|D^cb_\de|=0$ on Borel sets by \eqref{finale3v}. Since, trivially, $\nabla b_\de=0$ $\H^{n-1}$-a.e. on $\{v\le\de\}$, by \eqref{finale1v} we have that  $\nabla b_\de=0$ $\H^{n-1}$-a.e. on $\R^{n-1}$. Finally, by \eqref{finale2v} we have that
  \begin{equation}
    \label{pizzaefichi}
      S_{b_{\de}}\subset_{\H^{n-2}}\Big(S_v\cap\{v>\de\}^{(1)}\Big)\cup\pae \{v>\de\}\subset
  \Big(S_v\cap\{v^\wedge>0\}\Big)\cup\pae \{v>\de\}\,,
  \end{equation}
  so that $S_{b_\de}$ is locally $\H^{n-2}$-finite. We can thus apply Theorem \ref{thm afp cool} to $b_\de$, to find a finite or countable set $I_\de$, $\{c^\de_h\}_{h\in I_\de}\subset\R$, and a Borel partition $\{G_h^\de\}_{h\in I_\de}$ of $\{v>\de\}$ with
  \[
  b_\de=\sum_{h\in I_\de}c_h^\de\,1_{G_h^\de}\,,\qquad\mbox{$\H^{n-1}$-a.e. on $\{v>\de\}$}\,.
  \]
  By a diagonal argument over a sequence $\de_h\to 0$ as $h\to\infty$ with $\{v>\de_h\}$ of finite perimeter for every $h\in\N$, we prove the existence of $I$, $\{c_h\}_{h\in I}$ and $\{G_h\}_{h\in I}$ as in \eqref{E traslato a pezzi} such that $b_E=\sum_{h\in I}c_h 1_{G_h}$ $\H^{n-1}$-a.e. on $\{v>0\}$ (and thus, $\H^{n-1}$-a.e. on $\R^{n-1}$). This means that
  \[
  b_\de=\sum_{h\in I_\de}c_h\,1_{G_h\cap\{v>\de\}}\,,\qquad\mbox{$\H^{n-1}$-a.e. on $\R^{n-1}$}\,,
  \]
  and thus, again by Theorem \ref{thm afp cool}, $\sum_{h\in I}P(G_h\cap\{v>\de\}\cap B_R)<\infty$. This shows that $\{G_h\}_{h\in\N}$ is $v$-admissible and completes the proof.

  \medskip

  \noindent {\it Step two:} We now assume that  $E$ is a $v$-distributed set of finite perimeter such that \eqref{E traslato a pezzi} holds true,
  with $\{ G_h \}_{h \in I}$ $v$-admissible, and $2[b_E]\le[v]$ $\H^{n-2}$-a.e. on $\{v^\wedge>0\}$, and aim to prove that $E\in\M(v)$. Since $E$ is $v$-distributed with segments as sections and $\{ G_h \}_{h \in I}$ is  $v$-admissible, we see that $b_\de$ satisfies the assumption (ii) of Theorem \ref{thm afp cool} for a.e. $\de>0$. By applying that theorem, and then by letting $\de\to0^+$, we deduce that $\nabla b_E=0$ $\H^{n-1}$-a.e. on $\R^{n-1}$ and that $|D^cb_E|^+=0$. Hence, by applying Theorem \ref{thm characterization with barycenter}, we deduce that $E\in\M(v)$.
\end{proof}

\section{Rigidity in Steiner's inequality}\label{section rigidity in steiner} In this section we discuss the rigidity problem for Steiner's inequality. We begin in section \ref{section pizzauk} by proving the general sufficient condition for rigidity stated in Theorem \ref{thm sufficient bv}. We then present our characterizations of rigidity: in section \ref{section slice} we prove Theorem \ref{thm characterization sbv} (characterization of rigidity for $v\in SBV(\R^{n-1};[0,\infty))$ with $S_v$ locally $\H^{n-2}$-finite), while section \ref{section positive jump} and \ref{section proof of charact W11} deal with the cases of generalized polyhedra and with the ``no vertical boundaries'' case. (Note that the equivalence between the indecomposability of $F[v]$ and the condition that $\{v^\wedge>0\}$ does not essentially disconnect $\{v>0\}$ is proved in section \ref{section F indecomposable}). Finally, in section \ref{section planar sets} we address the proof of Theorem \ref{thm characterization R2} about the characterization of equality cases for planar sets.

\subsection{A general sufficient condition for rigidity}\label{section pizzauk} The general sufficient condition of Theorem \ref{thm sufficient bv} follows quite easily from Theorem \ref{thm characterization with barycenter}.

\begin{proof}
  [Proof of Theorem \ref{thm sufficient bv}] Let $E\in\M(v)$, so that, by Theorem \ref{thm characterization with barycenter}, we know that
  \begin{equation}
    \label{coarea-a-ah ottimali xx}
      \int_{\R}\H^{n-2}(G\cap\pae \{b_E>t\})\,dt=\int_{G\cap S_{b_E}\cap S_v}[b_E]\,d\H^{n-2}+|D^cb_E|^+(G\cap K)\,,
  \end{equation}
  whenever $G$ is a Borel subset of $\{v^\wedge>0\}$ and $K$ is a Borel set of concentration for $|D^cb_E|^+$. If $b_E$ is not constant on $\{v>0\}$, then there exists a Lebesgue measurable set $I\subset\R$ such that $\H^1(I)>0$ and for every $t\in I$ the Borel sets $G_+=\{b_E>t\}\cap\{v>0\}$ and $G_-=\{b_E\le t\}\cap\{v>0\}$ define a non-trivial Borel partition $\{G_+,G_-\}$ of $\{v>0\}$. Since
  \[
  \{v>0\}^{(1)}\cap\pae G_+\cap\pae G_-=\{v>0\}^{(1)}\cap \pae \{b_E>t\}\,,
  \]
  by  \eqref{general sufficient condition}, we deduce that
  \begin{equation}
    \label{pizzeria}
      \H^{n-2}\Big(\Big(\{v>0\}^{(1)}\cap\pae \{b_E>t\}\Big)\setminus\Big(\{v^\wedge=0\}\cup S_v\cup K\Big)\Big)>0\,,\qquad\forall t\in I\,.
  \end{equation}
  At the same time, by plugging $G=\{v>0\}^{(1)}\setminus(\{v^\wedge=0\}\cup S_v\cup K)\subset\{v^\wedge>0\}$ into \eqref{coarea-a-ah ottimali xx}, we find
  \[
  \int_\R\H^{n-2}\Big(\Big(\{v>0\}^{(1)}\cap\pae \{b_E>t\}\Big)\setminus\Big(\{v^\wedge=0\}\cup S_v\cup K\Big)\Big)\,dt=0\,.
  \]
  This is of course in contradiction with \eqref{pizzeria} and $\H^1(I)>0$.
\end{proof}

\begin{remark}
  {\rm By the same argument used in the proof of Theorem \ref{thm sufficient bv} one easily sees that if a Borel set $G\subset\R^m$ is essentially connected and $f\in BV(\R^m)$ is such that $|Df|(G^{(1)})=0$, then there exists $c\in\R$ such that $f=c$ $\H^m$-a.e. on $G$. In the case $G$ is an indecomposable set, this property was proved in \cite[Proposition 2.12]{dolzmannmu}.}
\end{remark}

\subsection{Characterization of rigidity for $v$ in $SBV$ with locally finite jump}\label{section slice} This section contains the proof of Theorem \ref{thm characterization sbv}.

\begin{proof}
  [Proof of Theorem \ref{thm characterization sbv}] {\it Step one:} We first prove that the mismatched stairway property is sufficient to rigidity. We argue by contradiction, and assume the existence of $E\in\M(v)$ such that $\H^n(E\Delta(t\,e_n+F[v]))>0$ for every $t\in\R$. By Theorem \ref{thm mv per v sbv}, there exists a finite or countable set $I$, $\{c_h\}_{h\in I}\subset\R$, $\{G_h\}_{h\in I}$ a $v$-admissible partition of $\{v>0\}$
such that $b_E=\sum_{h\in I}c_h\,1_{G_h}$ $\H^{n-1}$-a.e. on $\R^{n-1}$, $E=_{\H^n}W[v,b_E]$, and
  \begin{eqnarray}
  \label{grand canyon}
  2[b_E]\le [v]\,,&&\qquad\mbox{$\H^{n-2}$-a.e. on $\{v^\wedge>0\}$}\,.
  \end{eqnarray}
Of course, we may assume without loss of generality that $\H^{n-1}(G_h)>0$ for every $h\in I$ and that $c_h\ne c_k$ for every $h,k\in I$, $h\ne k$ (if any). In fact, $\#\,I\ge 2$, because if $\#\,I=1$, then we would have $\H^n(E\Delta(c\,e_n+F[v]))=0$ for some $c\in\R$.
  We can apply the mismatched stairway property to $I$, $\{G_h\}_{h\in I}$ and $\{c_h\}_{h\in I}$, to find $h_0,k_0\in I$, $h_0\ne k_0$, and a Borel set $\S$ with $\H^{n-2}(\S)>0$, such that
  \begin{eqnarray}  \label{grand canyon2}
    \S\subset\pae G_{h_0}\cap\pae G_{k_0}\cap\{v^\wedge>0\}\,,
  \qquad [v](z)< 2|c_{h_0}-c_{k_0}|\,,\qquad\forall z\in\S\,.
  \end{eqnarray}
  Since $b_E^\vee\ge\max\{c_{h_0},c_{k_0}\}$ and $b_E^\wedge\le\min\{c_{h_0},c_{k_0}\}$ on $\pae G_{h_0}\cap\pae G_{k_0}$, \eqref{grand canyon} implies
  \[
  2\,|c_{h_0}-c_{k_0}|\le [v]\,,\qquad\mbox{$\H^{n-2}$-a.e. on $\pae G_{h_0}\cap\pae G_{k_0}\cap\{v^\wedge>0\}$}\,.
  \]
  a contradiction to \eqref{grand canyon2} and $\H^{n-2}(\S)>0$.

 \medskip

 \noindent {\it Step two:} We show that the failure of the mismatched stairway property implies the failure of rigidity. Indeed, let us assume the existence of a $v$-admissible partition $\{G_h\}_{h\in I}$ of $\{v>0\}$ and of $\{c_h\}_{h\in I}\subset\R$ with $c_h\ne c_k$ for every $h,k\in I$, $h\ne k$, such that
  \begin{equation}\label{t0}
  2|c_h-c_k|\le[v]\,,\qquad\mbox{$\H^{n-2}$-a.e. on $\pae G_h\cap\pae G_k\cap\{v^\wedge>0\}$}\,,
  \end{equation}
  whenever $h,k\in I$ with $h\ne k$. We now claim that $E\in\M(v)$, where
  $$
  E= \bigcup_{h\in I}\Big(c_h\,e_n+(F[v]\cap(G_h\times\R))\Big).
  $$
  To this end, let $\de >0$ be such that $\{ v > \de \}$ is a set of finite perimeter. By Theorem \ref{thm afp cool}, $b_{\de}= b_E 1_{\{v > \delta\}} \in GBV (\R^{n-1})$ with $\nabla b_\de = 0$ $\H^{n-1}$-a.e. on $\R^{n-1}$, $|D^c b_\de|= 0$, $S_{b_\de}$ is locally $\H^{n-2}$-finite, and
  \begin{eqnarray}\label{gomez33}
    &&\{v>\de\}^{(1)}\cap S_{b_\de}\subset_{\H^{n-2}}\bigcup_{h\ne k\in I}\pae G_{h,\de}\cap\pae G_{k,\de}\,,
    \\\label{gomez34}
    &&[b_\de]=|c_h-c_k|\,,\qquad\mbox{$\H^{n-2}$-a.e. on $\pae G_{h,\de}\cap\pae G_{k,\de}\cap \{v>\de\}^{(1)}$, $h\ne k\in I$\,,}
  \end{eqnarray}
  where $G_{h,\de}=G_h\cap\{v>\de\}$ for every $h\in I$. By \eqref{t0}, \eqref{gomez33}, and \eqref{gomez34}, we find
  \begin{equation}
    \label{gomez35}
    2[b_\de]\le[v]\,,\qquad\mbox{$\H^{n-2}$-a.e. on $S_{b_\de}\cap\{v>\de\}^{(1)}$\,.}
  \end{equation}
  Let now $\{ \delta_h \}_{h \in \mathbb{N}}$ and $\{L_h \}_{h \in \mathbb{N}}$ be two sequences satisfying \eqref{superpippo0} and \eqref{superpippo0x}, and set $E_h = E \cap (\{ \delta_h < v < L_h \} \times \R)$, $\S_h=\{ \delta_h < v < L_h \}$, $b_h=1_{\S_h}b_E=1_{\{v<L_h\}}b_{\de_h}$ and $v_h=1_{\S_h}\,v$. Since $v_h\in (BV\cap L^\infty)(\R^{n-1})$ and $b_h\in GBV(\R^{n-1})$, we can apply Corollary \ref{lemma W} to compute $P(E_h;\S_h^{(1)}\times\R)$, to get (by taking into account that $\nabla b_\de = 0$ $\H^{n-1}$-a.e. on $\R^{n-1}$, $|D^c b_\de|= 0$, and \eqref{gomez35}), that
  \[
  P (E_h ; \S_h^{(1)} \times \mathbb{R}) = P (F[v] ; \S_h^{(1)} \times \mathbb{R}), \qquad \forall h \in \mathbb{N}\,;
  \]
  in particular,
  \[
  \lim_{h\to\infty}  P (E_h ; \S_h^{(1)} \times \mathbb{R})=P(F[v];\{v^\wedge>0\}\times\R)\,.
  \]
   Moreover, by repeating the argument used
 in the proof of Proposition \ref{corollario v uguale 0} we have
 $$
 \lim_{h \to \infty} P (E_h ; \pa^e \S_h \times \mathbb{R}) = P (F[v] ; \{ v^{\wedge} = 0 \} \times \mathbb{R}).
 $$
 We thus conclude that
 \[
 P(E)\le\liminf_{h\to\infty}P(E_h)=P(F[v])\,,
 \]
 that is, $E$ is of finite perimeter with $E\in\M(v)$.
\end{proof}

\subsection{Characterization of rigidity on generalized polyhedra}\label{section positive jump} We now prove Theorem \ref{thm characterization poliedri SUPERIOR}. We premise the following lemma to the proof of the theorem.

\begin{lemma}\label{lemma remark noioso}
    If $v\in BV(\R^{n-1};[0,\infty))$ with $\H^{n-1}(\{v>0\})<\infty$ is such that
    \begin{eqnarray}
    \label{geom hp 1}
    &&\mbox{$\{v>0\}$ is of finite perimeter}\,,
    \\
    \label{geom hp 2}
    &&\mbox{$\{v^\vee=0\}\cap\{v>0\}^{(1)}$ and $S_v$ are $\H^{n-2}$-finite}\,,
    \end{eqnarray}
    and if there exists $\e>0$ such that $\{v^\wedge=0\}\cup\{[v]>\e\}$ essentially disconnects $\{v>0\}$, then there exists $E\in\M(v)$ such that $\H^n(E\Delta(t\,e_n+F[v]))>0$ for every $t\in\R$.
\end{lemma}

\begin{proof}
  If $\e>0$ is such that $\{v^\wedge=0\}\cup \{[v]>\e\}$ essentially disconnects $\{v>0\}$, then there exist a non-trivial Borel partition $\{G_+,G_-\}$ of $\{v>0\}$ modulo $\H^{n-1}$ such that
  \begin{equation}
    \label{no1}
      \{v>0\}^{(1)}\cap\pae G_+\cap\pae G_-\subset_{\H^{n-2}}\{v^\wedge=0\}\cup\{[v]>\e\}\,.
  \end{equation}
  We are now going to show that the set $E$ defined by
  \[
  E=\Big(\Big(\frac\e2\,e_n+F[v]\Big)\cap(G_+\times\R)\Big)\cup\Big(F[v]\cap(G_-\times\R)\Big)\,,
  \]
  satisfies $E\in\M(v)$: this will prove the lemma. To this end we first prove that $G_+$ is a set of finite perimeter. Indeed, by $G_+\subset\{v>0\}$, we have
  \begin{equation}\label{no2}
    \pae G_+\subset(\pae G_+\cap\{v>0\}^{(1)})\cup\pae\{v>0\}\,,
  \end{equation}
  where $\pae G_+\cap\{v>0\}^{(1)}=\pae G_+\cap\pae G_-\cap\{v>0\}^{(1)}$, and thus, by \eqref{no1},
  \begin{eqnarray}\nonumber
    \pae G_+\cap\{v>0\}^{(1)}&\subset_{\H^{n-2}}&\pae G_+\cap \{v>0\}^{(1)}\cap \Big(\{v^\wedge=0\}\cup\{[v]>\e\}\Big)
    \\\label{no3}
    &\subset&    \Big(\pae G_+\cap\{v^\vee=0\}\cap\{v>0\}^{(1)}\Big)\cup S_v\,.
  \end{eqnarray}
  By combining \eqref{geom hp 1}, \eqref{geom hp 2} \eqref{no2}, and \eqref{no3} we conclude that $\H^{n-2}(\pae G_+)<\infty$, and thus, by Federer's criterion, that $G_+$ is a set of finite perimeter. Since $b_E=(\e/2)\,1_{G_+}$, we thus have $b_E\in BV(\R^{n-1})$, and thus $E=W[v,b_E]$ is of finite perimeter with segments as sections. Since $\nabla b_E=0$ $\H^{n-1}$-a.e. on $\R^{n-1}$ and $D^cb_E=0$, we are only left to check that $2[b_E]\le[v]$ $\H^{n-2}$-a.e. on $\{v^\wedge>0\}$ in order to conclude that $E\in\M(v)$ by means of Theorem \ref{thm characterization with barycenter}. Indeed, since $b_E=(\e/2)1_{G_+}$, we have $S_{b_E}=\pae G_+$ with $[b_E]=\e/2$ $\H^{n-2}$-a.e. on $\pae G_+$. By \eqref{dens3} and by \eqref{no1},
  \begin{eqnarray*}
  S_{b_E}\cap\{v^\wedge>0\}&=&\pae G_+\cap\{v^\wedge>0\}
  \\
  &=&\pae G_+\cap\pae G_-\cap\{v>0\}^{(1)}\cap\{v^\wedge>0\}
  \subset_{\H^{n-2}}\{[v]>\e\}\,.
  \end{eqnarray*}
  This concludes the proof of the lemma.
\end{proof}

\begin{proof}
  [Proof of Theorem \ref{thm characterization poliedri SUPERIOR}] {\it Step one:} We prove that, if $F[v]$ is a generalized polyhedron, then $v\in SBV(\R^{n-1})$, $S_v$ and $\{v^\vee=0\} \setminus\{v=0\}^{(1)}$ are $\H^{n-2}$-finite, and $\{v>0\}$ is of finite perimeter. Indeed, by assumption, there exist a finite disjoint family of indecomposable sets of finite perimeter and volume $\{A_j\}_{j\in J}$ in $\R^{n-1}$, and a family of functions $\{v_j\}_{j\in J}\subset W^{1,1}(\R^{n-1})$, such that
  \begin{eqnarray}\label{gp1*}
  \hspace{1.6cm}v=\sum_{j\in J}v_j\,1_{A_j}\,,\qquad \Big(\{v^\wedge=0\} \setminus\{v=0\}^{(1)} \Big)\cup S_v\subset_{\H^{n-2}}\bigcup_{j\in J}\pae A_j\,.
  \end{eqnarray}
  By \cite[Example 4.5]{AFP}, $v_j\,1_{A_j}\in SBV(\R^{n-1})$ for every $j\in J$, so that $v\in SBV(\R^{n-1})$ as $J$ is finite. Similarly, \eqref{gp1*} gives that $\{v^\wedge=0\}\setminus\{v=0\}^{(1)}$ and $S_v$ are both $\H^{n-2}$-finite. Since $\{v^\vee=0\}\setminus\{v=0\}^{(1)}$ and $\pae \{v>0\}$ are both subsets of $\{v^\wedge=0\}\setminus\{v=0\}^{(1)}$, we deduce that $\{v^\vee=0\}\setminus\{v=0\}^{(1)}$ and $\pae \{v>0\}$ are $\H^{n-2}$-finite. In particular, by Federer's criterion, $\{v>0\}$ is a set of finite perimeter.

  \medskip

  \noindent {\it Step two:} By step one,  if $F[v]$ is a generalized polyhedron, then $v$ satisfies the assumptions of Lemma \ref{lemma remark noioso}. In particular, if $\{v^\wedge=0\}\cup\{[v]>\e\}$ essentially disconnects $\{v>0\}$, then rigidity fails. This shows the implication $(i)\Rightarrow(ii)$ in the theorem.

  \medskip

  \noindent {\it Step three:} We show that if rigidity fails, then $\{v^\wedge=0\}\cup\{[v]>\e\}$ essentially disconnects $\{v>0\}$. By step one, if $F[v]$ is a generalized polyhedron, then $v$ satisfies the assumptions of Theorem \ref{thm mv per v sbv}. In particular, if $E\in\M(v)$, then $\nabla b_E=0$, $S_{b_E}\cap\{v^\wedge>0\}\subset S_v$, $2\,[b_E]\le [v]$ $\H^{n-2}$-a.e. on $\{v^\wedge>0\}$, and $|D^cb_E|^+=0$, so that, by \eqref{gp2} and \eqref{coarea-a-ah ottimali} we find
  \begin{eqnarray}
    \label{la0}
    S_{b_E}\subset_{\H^{n-2}}\bigcup_{j\in J}\pae A_j\,,\hspace{3cm}
  \\
   \label{la1}
      \int_{\R}\H^{n-2}(G\cap\pae \{b_E>t\})\,dt=\int_{G\cap S_{b_E}}[b_E]\,d\H^{n-2}\,,
  \end{eqnarray}
  for every Borel set $G\subset\{v^\wedge>0\}$.  We now combine \eqref{la0} and \eqref{la1} to deduce that
  \[
  \int_{\R}\H^{n-2}(A_j^{(1)}\cap\pae \{b_E>t\})\,dt=0\,,\qquad\forall j\in J\,.
  \]
  Since each $A_j$ indecomposable, by arguing as in the proof of Theorem \ref{thm sufficient bv}, we see that there exists $\{c_j\}_{j\in J}\subset\R$ such that $b_E=\sum_{j\in J}c_j\,1_{A_j}$ $\H^{n-1}$-a.e. on $\R^{n-1}$. In particular, we have $b_E=\sum_{j\in J_0}a_j\,1_{B_j}$ $\H^{n-1}$-a.e. on $\R^{n-1}$, where $\#\,J_0\le\#\,J$, $\{a_j\}_{j\in J_0}\subset\R$ with $a_j\ne a_i$ if $i,j\in J_0$, $i\ne j$, and $\{B_j\}_{j\in J_0}$ is a partition modulo $\H^{n-1}$ of $\R^{n-1}$ into sets of finite perimeter. (Notice that each $B_j$ may fail to be indecomposable.) Let us now assume, in addition to $E\in\M(v)$, that $\H^n(E\Delta(t\,e_n+F[v]))>0$ for every $t\in\R$. In this case the formula for $b_E$ we have just proved implies that $\#\,J_0\ge 2$. We now set,
  \[
  \e=\min\Big\{|a_i-a_j|:i,j\in J_0\,,i\ne j\Big\}\,,
  \]
  so that $\e>0$, and, for some $j_0\in J_0$, we set $G_+=B_{j_0}$ and $G_-=\bigcup_{j\in J_0\,,j\ne j_0}B_j$. In this way $\{G_+,G_-\}$ defines a non-trivial Borel partition $\{G_+,G_-\}$ of $\{v>0\}$ modulo $\H^{n-1}$, with the property that
  \[
  [v]\ge 2\,[b_E]\ge 2\,\e\,,\qquad\mbox{$\H^{n-2}$-a.e. on $\{v^\wedge>0\}\cap\pae G_+\cap\pae G_-$}\,.
  \]
  Thus, $\{v^\wedge=0\}\cup\{[v]>\e\}$ essentially disconnects $\{v>0\}$, and the proof of Theorem \ref{thm characterization poliedri SUPERIOR} is  complete.
\end{proof}

\subsection{Characterization of indecomposability on Steiner's symmetrals}\label{section F indecomposable} We show here that asking $\{v^\wedge=0\}$ does not essentially disconnect $\{v>0\}$ is in fact equivalent to saying that $F[v]$ is an indecomposable set of finite perimeter. This result shall be used to provide a second type of characterization of rigidity when $F[v]$ has no vertical parts, as well as in the planar case; see Theorem \ref{thm characterization no vertical SUFF} and Theorem \ref{thm characterization R2}.

\begin{theorem}
  \label{thm indecomponibili} If $v\in BV(\R^{n-1};[0,\infty))$ with $\H^{n-1}(\{v>0\})<\infty$, then $F[v]$ is indecomposable if and only if $\{v^\wedge=0\}$ does not essentially disconnect $\{v>0\}$.
\end{theorem}

We start by recalling a version of Vol'pert's theorem, see \cite[Theorem 2.4]{barchiesicagnettifusco}.

\begin{theoremletter}\label{thm volpert}
  If $E$ is a set of  finite perimeter in $\R^n$, then there exists a Borel set $G_E\subset\{v>0\}$ with $\H^{n-1}(\{v>0\}\setminus G_E)=0$ such that $E_z$ is a set of finite perimeter in $\R$ with $\pa^*(E_z)=(\pa^*E)_z$ for every $z\in G_E$. Moreover, if  $z\in G_E$ and $s\in\pa^*E_z$, then
    \begin{equation}
    \label{volpert 2}
  \q\nu_E(z,s)\ne 0\,,\qquad \nu_{E_z}(s)=\frac{\q\nu_E(z,s)}{|\q\nu_E(z,s)|}\,.
  \end{equation}
\end{theoremletter}

\begin{proof}[Proof of Theorem \ref{thm indecomponibili}]
  The fact that, if $F=F[v]$ is indecomposable, then $\{v^\wedge=0\}$ does not essentially disconnect $\{v>0\}$, is proved in Lemma \ref{lemma yayyy} below. We prove here the reverse implication. Precisely, let us assume the existence of a non-trivial partition $\{F_+,F_-\}$ of $F$ into sets of finite perimeter such that
  \begin{equation}
    \label{2steps1}
    0=\H^{n-1}(F^{(1)}\cap\pae F_+\cap\pae F_-)=\H^{n-1}(F^{(1)}\cap\pae F_+)\,.
  \end{equation}
  We aim to prove that if we set
  \[
  G_+=\{z\in\R^{n-1}:\H^1((F_+)_z)>0\}\,,\qquad G_-=\{z\in\R^{n-1}:\H^1((F_-)_z)>0\}\,,
  \]
  then $\{G_+,G_-\}$ defines a non-trivial Borel partition modulo $\H^{n-1}$ of $\{v>0\}$ such that
  \begin{equation}
    \label{2steps2}
  \{v>0\}^{(1)}\cap\pae G_+\cap\pae G_-\subset_{\H^{n-2}}\{v^\wedge=0\}\,.
  \end{equation}

  \noindent {\it Step one:} We prove that $\{G_{+},G_{-}\}$ is a non-trivial Borel partition (modulo $\H^{n-1}$) of $\{v>0\}$. The only non-trivial fact to obtain is that $\H^{n-1}(G_{+}\cap G_{-})=0$. By Theorem \ref{thm volpert} there exists $G_+^*\subset G_+$ with $\H^{n-1}(G_+\setminus G_+^*)=0$ such that, if $z\in G_+^*$, then
  \begin{eqnarray*}
    &&\mbox{$(F_+)_z$ is a set of finite perimeter in $\R$ with $(\pa^*F_+)_z=\pa^*((F_+)_z)$}\,,
    \\
    &&\mbox{$(F_-)_z$ is a set of finite perimeter in $\R$}\,,
    \\
    &&\mbox{$\{(F_+)_z,(F_-)_z\}$ is a partition modulo $\H^1$ of $(F^{(1)})_z$}\,,
  \end{eqnarray*}
  where the last property follows by Fubini's theorem and $\H^n(F\Delta F^{(1)})=0$. Let now
  \[
  G_+^{**}=\{z\in G_+^*:\H^1((F^{(1)})_z\setminus (F_+)_z)>0\}=G_+^*\cap G_-\,.
  \]
  If $z\in G_+^{**}$, then $\{(F_+)_z,(F_-)_z\}$ is a non-trivial partition modulo $\H^1$ of $(F^{(1)})_z$ into sets of finite perimeter. Since $(F^{(1)})_z$ is an interval for every $z\in\R^{n-1}$ (see \cite[Lemma 14.6]{maggiBOOK}), we thus have
  \[
  \H^0\Big([(F^{(1)})_z]^{(1)}\cap\pa^*((F_+)_z)\cap\pa^*((F_-)_z)\Big)\ge 1\,,\qquad\forall z\in G_+^{**}\,.
  \]
  In particular, since $(\pa^*F_+)_z=\pa^*((F_+)_z)$, $[(F^{(1)})_z]^{(1)}\subset(F^{(1)})_z$, and $(A\cap B)_z=A_z\cap B_z$ for every $A, B\subset\R^n$, we have
  \[
  \H^0\Big((F^{(1)}\cap\pa^*F_+)_z\Big)\ge 1\,,\qquad\forall z\in G_+^{**}\,.
  \]
  Hence, $G_+^{**}\subset\p(F^{(1)}\cap\pa^*F_+)$, and by \eqref{2steps1} and \cite[Proposition 3.5]{maggiBOOK} we conclude
  \begin{eqnarray*}
   0=\H^{n-1}(F^{(1)}\cap\pa^*F_+)\ge \H^{n-1}(\p(F^{(1)}\cap\pa^*F_+))\ge\H^{n-1}(G_+^{**})=\H^{n-1}(G^*_+\cap G_-)\,,
  \end{eqnarray*}
  that is,  $\H^{n-1}(G_+\cap G_-)=0$.

 \medskip

  \noindent {\it Step two:} We now show that
  \begin{equation}
    \label{jino}
      F^{(1)}\cap\Big(\Big(\pae G_{+}\cap\pae G_{-}\Big)\times\R\Big)\subset \pae F_{+}\cap \pae F_{-}\,.
  \end{equation}
  Indeed, let $(z,s)$ belong to the set on the left-hand side of this inclusion. If, by contradiction, $(z,s)\not\in\pae F_{+}\cap \pae F_{-}$, then either $(z,s)\in F_{-}^{(1)}$ or $(z,s)\in F_{+}^{(1)}$. In the former case,
  \begin{eqnarray*}
    \H^n(\C_{(z,s),r})=\H^n(F_{-}\cap \C_{(z,s),r})+o(r^n)
    \le 2r\,\H^{n-1}(G_{-}\cap\D_{z,r})+o(r^n)\,,
  \end{eqnarray*}
  that is $z\in G_{-}^{(1)}$, against $z\in\pae G_-$; the latter case is treated analogously.

  \medskip

  \noindent {\it Step three:} We conclude the proof. Arguing by contradiction, we can assume that
  \begin{eqnarray*}
  0&<&\H^{n-2}(  \{v>0\}^{(1)}\cap\pae G_+\cap\pae G_-\setminus \{v^\wedge=0\}) \\
  &=&\H^{n-2}(  \pae G_+\cap\pae G_- \cap \{v^\wedge>0\} )
  \\
  &=& \lim_{\e \to 0^+} \H^{n-2}( \pae G_+\cap\pae G_-\cap \{v^\wedge>\e\} )\,,
  \end{eqnarray*}
  where it should be noticed that all these measures could be equal to $+\infty$. However, by \cite[Theorem 8.13]{mattila}, if $\e$ is sufficiently small, then there exists a compact set $K$ with $0<\H^{n-2}(K)<\infty$ and $K\subset \pae G_+\cap\pae G_-\cap \{v^\wedge>\e\} $. Therefore,  by \eqref{jino},
  \begin{eqnarray*}
    \H^{n-1}\Big(F^{(1)}\cap\pae F_{+}\cap \pae F_{-}\Big)
    &\ge&
    \H^{n-1}\Big(F^{(1)}\cap\Big(\Big(\pae G_{+}\cap\pae G_{-}\Big)\times\R\Big)\Big)
    \\
    &\ge&
    \H^{n-1}(F^{(1)}\cap (K\times\R))
    \\
        {\small \mbox{ by \eqref{F[v] densita 1}}}\qquad
   &\ge&
    \H^{n-1}\Big(\Big\{x\in\R^n:\p x \in K\,,|\q x| <\frac{v^\wedge(\p x)}2\Big\}\Big)
    \\
    {\small \mbox{ by $K\subset\{v^\wedge>\e\}$}}\qquad
    &\ge&
    \H^{n-1}\Big(\Big\{x\in\R^n:\p x \in K\,,|\q x| <\frac{\e}2\Big\}\Big)
    \\
    {\small \mbox{ by \cite[2.10.45]{FedererBOOK}}}\qquad &\ge& c(n)\,\H^{n-2}(K) \, \e>0\,,
  \end{eqnarray*}
  a contradiction to \eqref{2steps1}.
\end{proof}

\begin{lemma}\label{lemma yayyy}
  Let $v\in BV(\R^{n-1};[0,\infty))$ with $\H^{n-1}(\{v>0\})<\infty$. If $\{G_+,G_-\}$ is a Borel partition of $\{v>0\}$ such that
  \begin{equation}
    \label{cool}
      \{v>0\}^{(1)}\cap\pae G_+\cap\pae G_-\subset_{\H^{n-2}}\{v^\wedge=0\}\,,
  \end{equation}
  then $F_+=F[v]\cap(G_+\times\R)$ and $F_-=F[v]\cap(G_-\times\R)$ are sets of finite perimeter, with
  \[
  P(F_+)+P(F_-)=P(F[v])\,.
  \]
\end{lemma}

\begin{proof} {\it Step one:} We prove that $F_+$ is a set of finite perimeter (the same argument works of course in the case of $F_-$). Indeed, let $G_{+\,0}=G_+\cup\{v=0\}$. Since $F[v]\cap(G_{+\,0}\times\R)=F_+\cap(G_{+\,0}\times\R)$, we find that
\begin{eqnarray}\label{cool6}
  \H^{n-1}(\pae F\cap(G_{+\,0}^{(1)}\times\R))=\H^{n-1}(\pae F_+\cap(G_{+\,0}^{(1)}\times\R))\,,
\end{eqnarray}
where we have set $F=F[v]$. Since $\pae F_+\cap(G_{+\,0}^{(0)}\times\R)=\emptyset$, we find
\begin{equation}\label{cool1000}
  \H^{n-1}(\pae F_+\cap(G_{+\,0}^{(0)}\times\R))=0\,.
\end{equation}
We now notice that
\begin{eqnarray*}
\R^{n-1}\setminus(G_{+\,0}^{(1)}\cup G_{+\,0}^{(0)})=\pae  G_{+\,0}=\pae G_-\,.
\end{eqnarray*}
Since $\{v>0\}^{(0)}\cap\pae G_-=\emptyset$, $\pae \{v>0\}\subset\{v^\wedge=0\}$, and $\{v>0\}^{(1)}\cap\pae G_+\cap\pae G_-=\{v>0\}^{(1)}\cap\pae G_-$, by \eqref{cool} we find that
\begin{equation}
    \label{cool2}
    \pae G_-\subset_{\H^{n-2}}\{v^\wedge=0\}\,.
\end{equation}
Thus, by \eqref{cool6}, \eqref{cool1000}, \eqref{cool2}, and by Federer's criterion, in order to prove that $F_+$ is a set of finite perimeter, we are left to show that
\begin{equation}
    \label{cool5}
\H^{n-1}\Big(\pae F_+\cap(\{v^\wedge=0\}\times\R)\Big)<\infty\,.
\end{equation}
Since $(\pae F_+)_z=\emptyset$ whenever $z\in\{v=0\}^{(1)}$, we find that
\begin{equation}
  \label{cool3}
\H^{n-1}\Big(\pae F_+\cap(\{v=0\}^{(1)}\times\R)\Big)=0\,.
\end{equation}
Since $F_+\subset F$, then $\pae F_+\subset F^{(1)}\cup \pae F$. At the same time, if $z\in\{v^\vee=0\}$, then $(\pae F)_z\cup(F^{(1)})_z\subset\{0\}$ by \eqref{F[v] densita 1} and \eqref{F[v] paM}, so that, if $G\subset\{v^\vee=0\}$, then
\[
\H^{n-1}\Big(\pae F_+\cap(G\times\R)\Big)\le\H^{n-1}(G\times\{0\})=\H^{n-1}(G)\,.
\]
By Lebesgue's points theorem, $\H^{n-1}(\{v^\vee=0\}\setminus\{v=0\}^{(1)})=0$, thus, if we plug in the above identity $G=\{v^\vee=0\}\setminus\{v=0\}^{(1)}$, then \eqref{cool3} gives
\begin{equation}
  \label{cool4}
\H^{n-1}\Big(\pae F_+\cap(\{v^\vee=0\}\times\R)\Big)=0\,.
\end{equation}
Finally, if $z\in\{v^\wedge=0<v^\vee\}$, then $(F^{(1)})_z\subset\{0\}$ and $(\pae F)_z\subset[-v^\vee(z)/2,v^\vee(z)/2]$ by Corollary \ref{lemma W2}. Since $\{v^\wedge=0<v^\vee\}$ is countably $\H^{n-2}$-rectifiable, by \cite[3.2.23]{FedererBOOK} and \eqref{perimetro di F} we find
\begin{eqnarray}\nonumber
\H^{n-1}\Big(\pae F_+\cap(G\times\R)\Big)&=&\int_{G}\H^1((\pae F_+)_z)\,d\H^{n-2}(z)\le
\int_{G}v^\vee\,d\H^{n-2}
\\\label{cool9}
&=&P(F;G\times\R)\,,
\end{eqnarray}
for every Borel set $G\subset \{v^\wedge=0<v^\vee\}$. By combining \eqref{cool9} (with $G=\{v^\wedge=0<v^\vee\}$) and \eqref{cool4}, we prove \eqref{cool5} for $F_+$. The proof for $F_-$ is of course entirely analogous.

\medskip

\noindent {\it Step two:} We now prove that $P(F_+)+P(F_-)=P(F)$. Since $F$ is $\H^n$-equivalent to $F_+\cup F_-$, by \cite[Lemma 12.22]{maggiBOOK} it suffices to prove that $P(F_+)+P(F_-)\le P(F)$. By \eqref{cool6}, \eqref{cool4}, and the analogous relations for $F_-$, we are actually left to show that
\begin{equation}\label{cool*}
P(F_+;G\times\R)+P(F_-;G\times\R)\le P(F;G\times\R)\,,
\end{equation}
for every Borel set $G\subset\{v^\wedge=0<v^\vee\}$. Since $F_+=F[1_{G_+}v]$ is of finite perimeter, by Corollary \ref{lemma W2} we have $v_+=1_{G_+}v\in BV(\R^{n-1})$, with
  \begin{equation}
    \label{perimetro di F+}
    P(F_+;G\times\R)=2\int_{G\cap\{v_+>0\}}\sqrt{1+\Big|\frac{\nabla v_+}2\Big|^2}
    +\int_{G\cap S_{v_+}}\,[v_+]\,d\H^{n-2}+|D^cv_+|(G)\,,
  \end{equation}
for every Borel set $G\subset\R^{n-1}$. Since $\{v^\wedge=0<v^\vee\}$ is countably $\H^{n-2}$-rectifiable, we find
\[
P\Big(F_+;G\times\R\Big)=\int_{G\cap S_{v_+}}\,[v_+]\,d\H^{n-2}=P(F_+;G\cap S_{v_+})\,,
\]
for every Borel set $G\subset \{v^\wedge=0<v^\vee\}$; moreover, an analogous formula holds true for $F_-$. Thus, \eqref{cool*} takes the form
\begin{equation}\label{cool**}
P(F_+;G\cap S_{v_+})+P(F_-;G\cap S_{v_-}) \le P(F;G\times\R)\,,
\end{equation}
for every Borel set $G\subset\{v^\wedge=0<v^\vee\}$. If $G\subset\{v^\wedge=0<v^\vee\}\setminus S_{v_-}$, then \eqref{cool**} reduces to $P(F_+;G\cap S_{v_+})\le P(F;G\times\R)$, which follows immediately by \eqref{cool9}. Similarly, if we choose $G\subset\{v^\wedge=0<v^\vee\}\setminus S_{v_+}$. We may thus conclude the proof of the lemma, by showing that
\begin{equation}
  \label{coolz}
  \H^{n-2}\Big(\{v^\wedge=0<v^\vee\}\cap S_{v_+}\cap S_{v_-}\Big)=0\,.
\end{equation}
To prove \eqref{coolz}, let us notice that for $\H^{n-2}$-a.e. $z\in \{v^\wedge=0<v^\vee\}\cap S_{v_+}\cap S_{v_-}$, we have
\begin{eqnarray}\label{lamb1}
  \{v>t\}_{z,t}\toloc H_0\,,&&\qquad\forall t\in(0,v^\vee(z))\,,
  \\\label{lamb2}
  \{v_+>t\}_{z,t}\toloc H_1\,,&&\qquad\forall t\in(v_+^\wedge(z),v_+^\vee(z))\,,
  \\\nonumber
  \{v_->t\}_{z,t}\toloc H_2\,,&&\qquad\forall t\in(v_-^\wedge(z),v_-^\vee(z))\,,
\end{eqnarray}
as $r\to 0^+$. Now, $v_+^\vee(z)\le v^\vee(z)$, therefore $(v_+^\wedge(z),v_+^\vee(z))\subset (0,v^\vee(z))$. We may thus pick $t>0$ such that \eqref{lamb1} and \eqref{lamb2} hold true, and correspondingly,
\[
\{v>t\}_{z,t}\toloc H_0\,,\qquad (G_+\cap\{v>t\})_{z,r}=\{v_+>t\}_{z,t}\toloc H_1\,,
\]
as $r\to 0^+$. Since $G_+\cap\{v>t\}\subset\{v>t\}$, it must be $H_1\subset H_0$, and thus $H_1=H_0$. This implies that
\[
\H^{n-1}(\D_{z,r}\cap((z+H_0)\setminus G_+))=o(r^{n-1})\,,\qquad\mbox{as $r\to 0^+$}\,.
\]
The same argument applies to $v_-$, and gives
\[
\H^{n-1}(\D_{z,r}\cap((z+H_0)\setminus G_-))=o(r^{n-1})\,,\qquad\mbox{as $r\to 0^+$}\,.
\]
Hence, $\theta^*(G_+\cap G_-,z)\ge \theta(z+H_0,z)=1/2$, a contradiction to $\H^{n-1}(G_+\cap G_-)=0$.
\end{proof}

\subsection{Characterizations of rigidity without vertical boundaries}\label{section proof of charact W11} We now prove Theorem \ref{thm characterization no vertical SUFF}, by combining Theorem \ref{thm sufficient bv} and the results from section \ref{section F indecomposable}.

\begin{proof}[Proof of Theorem \ref{thm characterization no vertical SUFF}] We start noticing that the equivalence between (ii) and (iii) was proved in Theorem \ref{thm indecomponibili}. We are thus left to prove the equivalence between (i) and (ii).

\medskip

\noindent {\it Step one:} We prove that (ii) implies (i). By Lemma \ref{lemma Dcv}, we have that $D^cv\llcorner\{v^\wedge=0\}=0$; since we are now assuming that $D^sv\llcorner \{v^\wedge>0\}=0$, we conclude that $D^cv=0$. We now show that $\{v^\wedge=0\}\cup S_v$ does not essentially disconnect $\{v>0\}$. Otherwise, there exists a non-trivial Borel partition $\{G_+,G_-\}$ modulo $\H^{n-1}$ of $\{v>0\}$ such that
  \begin{equation}
    \label{hazel}
    \{v^\wedge>0\}\cap\pae G_+\cap\pae G_-\subset \{v>0\}^{(1)}\cap\pae G_+\cap\pae G_-\subset_{\H^{n-2}} \{v^\wedge=0\}\cup S_v\,,
  \end{equation}
  where the first inclusion follows from \eqref{dens3}. Since $\{v^\wedge=0\}$ does not essentially disconnect $\{v>0\}$ and since $D^sv\llcorner\{v^\wedge>0\}=0$ implies $\H^{n-2}(S_v\cap\{v^\wedge>0\})=0$, we conclude
  \begin{eqnarray*}
\nonumber
  0&<& \H^{n-2}\Big(\Big(\{v>0\}^{(1)}\cap\pae G_+\cap\pae G_-\Big)\setminus \{v^\wedge=0\}\Big)\\
   &=& \H^{n-2}\Big(\{v^\wedge>0\}\cap\pae G_+\cap\pae G_-\Big)
      = \H^{n-2}\Big( \Big(\{v^\wedge>0\}\cap\pae G_+\cap\pae G_- \Big) \setminus S_v\Big)\,,
  \end{eqnarray*}
  a contradiction to \eqref{hazel}.   This proves that $\{v^\wedge=0\}\cup S_v$ does not essentially disconnect $\{v>0\}$. Since, as said, $D^cv=0$, we can thus apply Theorem \ref{thm sufficient bv} to deduce (i).

  \medskip

  \noindent {\it Step two:} We prove that (i) implies (ii). Indeed, if (ii) fails, then there exists a non-trivial Borel partition $\{G_+,G_-\}$ of $\{v>0\}$ modulo $\H^{n-1}$, such that $\{v>0\}^{(1)}\cap\pae G_+\cap\pae G_-\subset_{\H^{n-2}}\{v^\wedge=0\}$. By Lemma \ref{lemma yayyy}, we find that $F_+=F\cap(G_+\times\R)$ and $F_-=F\cap(G_-\times\R)$ are sets of finite perimeter with $P(F_+)+P(F_-)=P(F)$. Let us now set
  $E=(e_n+F_+)\cup F_-$.  By \cite[Lemma 12.22]{maggiBOOK}, we have that $E$ is a $v$-distributed set of finite perimeter, with
  \[
  P(F)\le P(E)\le P(e_n+F_+)+P(F_-)=P(F_+)+P(F_-)=P(F)\,,
  \]
  that is $E\in\M(v)$. However, $\H^n(E\Delta(t\,e_n+F))>0$ for every $t\in\R$ since $\{G_+,G_-\}$ was a non-trivial Borel partition of $\{v>0\}$.
\end{proof}

\subsection{Characterizations of rigidity on planar sets}\label{section planar sets} We finally prove Theorem \ref{thm characterization R2}, that fully addresses the rigidity problem for planar sets.

\begin{proof}
  [Proof of Theorem \ref{thm characterization R2}] {\it Step one:} Let us assume that (ii) holds true. We first notice that, in this case, $D^cv=0$, so that, thanks to Theorem \ref{thm sufficient bv}, we are left to prove that
  \begin{equation}\label{fedex1}
    \mbox{$\{v^\wedge=0\}\cup S_v$ does not essentially disconnect $\{v>0\}$}\,,
  \end{equation}
  in order to show the validity of (i). Since (ii) implies that $\{v^\wedge=0\}\cup S_v\subset\R\setminus(a,b)$ where $\{v>0\}$ is $\H^1$-equivalent to $(a,b)$, \eqref{fedex1} follows from the fact that $\R\setminus(a,b)$ does not essentially disconnect $(a,b)$.

  \medskip

  \noindent {\it Step two:} We now assume the validity of (i). Let $[a,b]$ be the least closed interval which contains $\{v>0\}$ modulo $\H^1$. (Note that $[a,b]$ could a priori be unbounded.) Let us assume without loss of generality that $\H^1(\{v>0\})>0$, so that $(a,b)$ is non-empty. We now show that $v^\wedge(c)>0$ for every $c\in(a,b)$. Indeed, let $F=F[v]$, $F_+=F\cap[[c,\infty)\times\R]$, and $F_-=F\cap[(-\infty,c)\times\R]$. Since $F_+=F[1_{[c,\infty)}v]$ and $F_-=F[1_{(-\infty,c)}v]$, we can apply \eqref{perimetro di F} to find that
  \begin{eqnarray}\label{+}
    P(F_+)&=&2\,\int_{\{v>0\}\cap(c,\infty)}\sqrt{1+\Big|\frac{v'}2\Big|^2}+\int_{S_v\cap(c,\infty)}[v]d\H^0+v(c^+)
    \\\nonumber
    &&+|D^cv|\Big(\{\widetilde{v}>0\}\cap(c,\infty)\Big)\,,
  \end{eqnarray}
  and
  \begin{eqnarray} \label{-}
    P(F_-)&=&2\,\int_{\{v>0\}\cap(-\infty,c)}\sqrt{1+\Big|\frac{v'}2\Big|^2}+\int_{S_v\cap(-\infty,c)}[v]d\H^0+v(c^-)
    \\\nonumber
    &&+|D^cv|\Big(\{\widetilde{v}>0\}\cap(-\infty,c)\Big)\,,
  \end{eqnarray}
  where we have set $v(c^+)=\aplim(v,(c,\infty),c)$, $v(c^-)=\aplim(v,(-\infty,c),c)$, and we have used the fact that $D^c(1_{(c,\infty)}v)$ is the restriction of $D^cv$ to $(c,\infty)$, that
  \[
  [1_{(c,\infty)}v](z)=\left\{\begin{array}
    {l l}
    [v](z)\,,&\mbox{if $z>c$}\,,
    \\
    v(c^+)\,,&\mbox{if $z=c$}\,,
    \\
    0\,,&\mbox{if $z<c$}\,,
  \end{array}\right .
  \]
  as well as the analogous facts for $1_{(-\infty,c)}v$. Notice that, if $v^\wedge(c)=0$, then either $v(c^+)=0$ or $v(c^-)=0$, and, correspondingly, $P(F_+)+P(F_-)=P(F)$ by \eqref{perimetro di F}, \eqref{+}, and \eqref{-}. As a consequence, if we set $E=F_+\cup(e_2+F_-)$, then by,
  arguing as in step two of the proof of Theorem \ref{thm characterization no vertical SUFF}, we find that
  \[
  P(F)\le P(E)\le P(F_+)+P(e_2+F_-)=P(F_+)+P(F_-)= P(F)\,,
  \]
  that is $E\in\M(v)$, in contradiction to (i). This proves that $v^\wedge(c)>0$ for every $c\in(a,b)$. In particular, since $\{v>0\}$ is $\H^1$-equivalent to $\{v^\wedge>0\}$, we find that $\{v>0\}$ is $\H^1$-equivalent to $(a,b)$. We now prove $(a,b)$ to be bounded. Let us now decompose $v$ as $v=v_1+v_2$ where $v_1\in W^{1,1}(\R)$ and $v_2\in BV(\R)$ with $D^av_2=0$ (see \cite[Corollary 3.33]{AFP}. If $v_2$ is non-constant (modulo $\H^1$) in $(a,b)$, then we find a contradiction with (i) by Proposition \ref{proposition dsv}. Thus, there exists $t\in\R$ such that $v_2=t$ on $(a,b)$, and, in particular, $v=v_1+t\in W^{1,1}(a,b)$. In particular, since $\{v>0\}=_{\H^1}(a,b)$ and $\H^1(\{v>0\})<\infty$, we find that $(a,b)$ is bounded.

  \medskip

  \noindent {\it Step three:} We prove that (ii) implies (iii). Indeed, since $\{v>0\}$ is $\H^1$-equivalent to $(a,b)$ and $v^\wedge>0$ on $(a,b)$, then, by Remark \ref{remark essential connected}, we have that $\{v^\wedge=0\}$ does not essentially disconnect $\{v>0\}$. In particular, by Theorem \ref{thm indecomponibili}, we have that $F[v]$ is indecomposable. Since $v\in W^{1,1}(a,b)$, by \cite[Proposition 1.2]{ChlebikCianchiFuscoAnnals05}, we find that
  \begin{equation}
    \label{pauraaaa}
      \H^1\Big(\Big\{x\in\pa^*F[v]:\q\nu_{F[v]}=0\,,\p x\in(a,b)\Big\}\Big)=0\,.
  \end{equation}
  Since  $\{v^\wedge>0\}=(a,b)$, we deduce \eqref{paura}.

  \medskip

  \noindent {\it Step four:} We prove that (iii) implies (ii). Since $F[v]$ is now indecomposable, by Theorem \ref{thm indecomponibili} we have that $\{v^\wedge=0\}$ does not essentially disconnect $\{v>0\}$. In particular, $\{v>0\}$ is an essentially connected subset of $\R$, and thus, by \cite[Proof of Theorem 1.6, step one]{ccdpmGAUSS}, $\{v>0\}$ is $\H^1$-equivalent to an interval. Since $\H^1(\{v>0\})<\infty$, we thus have that $\{v>0\}=_{\H^1}(a,b)$, with $(a,b)$ bounded. Since $\{v^\wedge=0\}$ does not essentially disconnect $\{v>0\}$, it must be $v^\wedge>0$ on $(a,b)$. Finally, by \eqref{paura} and the fact that $v^\wedge>0$ on $(a,b)$, we find \eqref{pauraaaa}. Again by \cite[Proposition 1.2]{ChlebikCianchiFuscoAnnals05}, we conclude that $v\in W^{1,1}(a,b)$.
\end{proof}

\appendix

\section{Equality cases in the localized Steiner's inequality}\label{section fusco} The rigidity results described in this paper for the equality cases in Steiner's inequality $P(E)\ge P(F[v])$ can be suitably formulated and proved for the localized Steiner's inequality $P(E;\Om\times\R)\ge P(F[v];\Om\times\R)$ under the assumption that $\Om$ is an open connected set. This generalization does not require the introduction of new ideas, but, of course, requires a clumsier notation. Another possible approach is that of obtaining the localized rigidity results through an approximation process. For the sake of clarity, we exemplify this by showing a proof of Theorem \ref{thm ccf2} based on Theorem \ref{thm sufficient bv}. The required approximation technique is described in the following lemma.

\begin{lemma}\label{lemma omega}
  If $\Om$ is a connected open set in $\R^{n-1}$, $v\in BV(\Om;[0,\infty))$ with $\H^{n-1}(\{v>0\})<\infty$, $E$ is a $v$-distributed set with $P(E;\Om\times\R)<\infty$ and segments as vertical sections, then there exists an increasing sequence $\{\Om_k\}_{k\in\N}$ of bounded open connected sets of finite perimeter such that $\Om=\bigcup_{k\in\N}\Om_k$, $\Om_k$ is compactly contained in $\Om$, $v_k=1_{\Om_k}\,v\in BV(\R^{n-1};[0,\infty))$ with $\H^{n-1}(\{v_k>0\})<\infty$, $E_k=E\cap(\Om_k\times\R)$ is a $v_k$-distributed set of finite perimeter, and
  \begin{eqnarray}
    \label{per Ek}
    P(E_k)&=&P(E;\Om_k\times\R)+P(F[v_k];\pa^*\Om_k\times\R)\,,
    \\
    \label{per Fk}
    P(F[v_k])&=&P(F[v];\Om_k\times\R)+P(F[v_k];\pa^*\Om_k\times\R)\,.
  \end{eqnarray}
  Finally, if $E\in\M_\Om(v)$, see \eqref{MGv}, then $E_k\in\M(v_k)$.
\end{lemma}

\begin{proof}
  By intersecting $\Om$ with increasingly larger balls, and by a diagonal argument, we may assume that $\Om$ is bounded. Let $u$ be the distance function from $\R^{n-1}\setminus\Om$. By \cite[Remark 18.2]{maggiBOOK}, $\{u>\e\}$ is an open bounded set of finite perimeter with $\pa^*\{u>\e\}=_{\H^{n-2}}\{u=\e\}$ for a.e. $\e>0$.  Moreover, if we set $f(x)=u(\p x)$, $x\in\R^n$, then $f:\R^n\to\R$ is a Lipschitz function with $|\nabla f|=1$ a.e. on $\Om\times\R$, and $\{f=\e\}=\{u=\e\}\times\R$ for every $\e>0$, so that, by the coarea formula for Lipschitz functions \cite[Theorem 18.1]{maggiBOOK},
  \[
  \int_0^\infty\,\H^{n-1}\Big(E^{(1)}\cap(\{u=\e\}\times\R)\Big)\,d\e=\int_{E^{(1)}\cap(\Om\times\R)}|\nabla f|\,d\H^n=\|v\|_{L^1(\Om)}<\infty\,.
  \]
  We may thus claim that for a.e. $\e>0$,
  \begin{equation}
    \label{staycool}
    \H^{n-1}\Big(E^{(1)}\cap(\pa^*\{u>\e\}\times\R)\Big)<\infty\,.
  \end{equation}
  We now fix a sequence $\{\e_k\}_{k\in\N}$ such that $\e_k\to 0^+$ as $k\to\infty$, $\{u>\e_k\}$ is an open set of finite perimeter and $\e=\e_k$ satisfies \eqref{staycool} for every $k\in\N$. Let now $\{A_{k,i}\}_{i\in I_k}$ be the family of connected components of $\{u>\e_k\}$. Since $\pa A_{k,i}\subset\{u=\e_k\}$ and $\{u=\e_k\}=_{\H^{n-2}}\pa^*\{u>\e_k\}$ is $\H^{n-2}$-finite, we conclude by Federer's criterion that $A_{k,i}$ is of finite perimeter for every $k\in\N$ and $i\in I_k$. Let us now fix $z\in\Om$, and let $k_0\in\N$ be such that $z\in \{u>\e_k\}$ for every $k\ge k_0$. In this way, for every $k\ge k_0$, there exists $i_k(z)\in I_k$ such that $z\in A_{k,i_k(z)}$. We shall set
  \[
  \Om_k=A_{k,i_k(z)}\,.
  \]
  By construction, each $\Om_k$ is a bounded open connected set of finite perimeter, and $\Om_k\subset\Om_{k+1}$ for every $k\ge k_0$. Let us now prove $\Om=\bigcup_{k\in\N}\Om_k$. Indeed, let $y\in\Om$, let $\g\in C^0([0,1];\Om)$ such that $\g(0)=z$, $\g(1)=y$, and consider $K=\g([0,1])$. Since $K$ is compact, there exists $k_1\in\N$ such that $K\subset\{u>\e_k\}$ for every $k\ge k_1$. Since $K$ is connected and $\{z\}\subset K\cap \Om_k$ for every $k\ge k_1$, we find that $K\subset\Om_k$, thus $y\in\Om_k$, for every $k\ge k_1$. We now prove that $E_k$ is a set of finite perimeter. Indeed, since $E_k=E\cap(\Om_k\times\R)$, $\pae E_k\subset\Big[\pae  E\cap(\ov{\Om}_k\times\R)\Big]\cup \Big[E^{(1)}\cap(\pae \Om_k\times\R)\Big]$.  Since $\Om_k$ is compactly contained in $\Om$, we find $\H^{n-1}(\pae  E\cap(\ov{\Om}_k\times\R))\le P(E;\Om\times\R)<\infty$; thus, by taking \eqref{staycool} into account, we find $\H^{n-1}(\pae E_k)<\infty$, and thus, that $E_k$ is a set of finite perimeter thanks to Federer's criterion. By Proposition \ref{proposition insieme u1u2}, $v_k\in BV(\R^{n-1})$ with $\H^{n-1}(\{v_k>0\})<\infty$, and $F[v_k]$ is a set of finite perimeter too. Since $E_k$ is a $v_k$-distributed set of finite perimeter and $\pae \Om_k$ is a countably $\H^{n-2}$-rectifiable set contained in $\{v_k^\wedge=0\}$, by Proposition \ref{corollario v uguale 0}\,
  \[
  P(E_k;\pae \Om_k\times\R)=P(F[v_k];\pae \Om_k\times\R)\,.
  \]
  Moreover, by $E_k=E\cap(\Om_k\times\R)$ and $F[v_k]=F[v]\cap(\Om_k\times\R)$,
  \begin{eqnarray*}
    P(E_k;\Om_k^{(1)}\times\R)=P(E;\Om_k^{(1)}\times\R)\,,\qquad P(F[v_k];\Om_k^{(1)}\times\R)=P(F[v];\Om_k^{(1)}\times\R)\,.
  \end{eqnarray*}
  Since $\Om_k^{(0)}\times\R\subset E_k^{(0)}\cap F[v_k]^{(0)}$, we have proved \eqref{per Ek} and \eqref{per Fk}. Finally, if $E\in\M_\Om(v)$, then by \eqref{steiner inequality} we have $P(E;\Om_k\times\R)=P(F[v];\Om_k\times\R)$, and thus, by \eqref{per Ek} and \eqref{per Fk}, that $P(E_k)=P(F[v_k])$.
\end{proof}

\begin{proof}[Proof of Theorem \ref{thm ccf2}] Let $v\in BV(\Om;[0,\infty))$ with $\H^{n-1}(\{v>0\})<\infty$, $D^sv\llcorner\{v^\wedge>0\}=0$ and $v^\wedge>0$ $\H^{n-2}$-a.e. on $\Om$ (so that $D^sv\llcorner\Om=0$). Let $E\in\M_\Om(v)$, and assume by contradiction that $\H^n(E\Delta(t\,e_n+F[v]))>0$ for every $t\in\R$. Let $\Om_k$ be defined as in Lemma \ref{lemma omega}, and let $v_k=1_{\Om_k}\,v$, $E_k=E\cap(\Om_k\times\R)$, so that $E_k\in\M(v_k)$ for every $k\in\N$. However, as it is easily seen, $\H^n(E_k\Delta(t\,e_n+F[v_k]))>0$ for every $t\in\R$ and for every $k$ large enough. Thus, rigidity fails for $v_k$ if $k$ is large enough. By Theorem \ref{thm sufficient bv},
\begin{equation}
  \label{hofame}
  \mbox{$\{v_k^\wedge=0\}\cup S_{v_k}\cup M_k$ essentially disconnects $\{v_k>0\}$,}
\end{equation}
where $M_k$ is a concentration set for $D^cv_k$. Since $v_k^\wedge=1_{\Om_k^{(1)}}v^\wedge$ in $\Om$, $v^\wedge>0$ $\H^{n-2}$-a.e. on $\Om$, and $\Om_k$ is compactly contained in $\Om$, we find that
\[
\{v_k^\wedge=0\}=(\R^{n-1}\setminus\Om_k^{(1)})\cup(\{v^\wedge=0\}\cap\Om_k^{(1)})=_{\H^{n-2}}\R^{n-1}\setminus\Om_k^{(1)}\,.
\]
By $D^sv\llcorner\Om=0$, by Lemma \ref{lemma fgE},  and again by taking into account that $\Om_k$ is compactly contained in $\Om$, we find
\[
S_{v_k}\cap\Om_k^{(1)}=S_v\cap\Om_k^{(1)}=_{\H^{n-2}}S_v\cap(\Om_k^{(1)}\setminus\Om)=\emptyset\,,
\]
Moreover, by Lemma \ref{lemma fgE}, $D^cv_k=D^cv\llcorner\Om_k^{(1)}=D^cv\llcorner(\Om_k^{(1)}\setminus\Om)=0$,
so that we may take $M_k=\emptyset$. Finally, $\{v_k>0\}$ is $\H^{n-1}$-equivalent to $\Om_k$, and thus, by Remark \ref{remark essential connected}, \eqref{hofame} can be equivalently rephrased as
\begin{equation}
  \label{hofame33}
  \mbox{$(\R^{n-1}\setminus\Om_k^{(1)})\cup (S_{v_k}\setminus\Om_k^{(1)})$ essentially disconnects $\Om_k$\,.}
\end{equation}
In turn, this is equivalent to saying that $\Om_k$ is not essentially connected. Since $\Om_k$ is of finite perimeter, by Remark \ref{remark iiii}, $\Om_k$ is not indecomposable. By \cite[Proposition 2]{ambrosiocaselles}, $\Om_k$ is not connected. We have thus reached a contradiction.
\end{proof}

\section{A perimeter formula for vertically convex sets}\label{section perimeter formula} We summarize here a perimeter formula for sets with segments as vertical sections that can be obtained as a consequence of Corollary \ref{lemma W} and Proposition \ref{corollario v uguale 0}, and that may be of independent interest.

\begin{theorem}
  If $E=\{x\in\R^n:u_1(x')<x_n<u_2(x')\}$ is a set of finite perimeter and volume defined by $u_1,u_2:\R^{n-1}\to\R$ with $u_1\le u_2$ on $\R^{n-1}$, then $u_1$ and $u_2$ are approximately differentiable $\H^{n-1}$-a.e. on $\{u_2>u_1\}$, and
  \begin{eqnarray*}\nonumber
  P(E)
  &=&\int_{\{v>0\}}\sqrt{1+|\nabla u_1|^2}+\sqrt{1+|\nabla u_2|^2}\,d\H^{n-1}
  \\
  &&+\int_{S_v\cup S_b}\,\min\Big\{v^\vee+v^\wedge,\max\{[v],2[b]\}\Big\}\,d\H^{n-2}
  \\
  \nonumber
  &&+|D^cu_1|^+(\{v^\wedge>0\})+|D^cu_2|^+(\{v^\wedge>0\})\,,
  \end{eqnarray*}
  where $v=u_2-u_1$, $b=(u_1+u_2)/2$ and, for every Borel set $G\subset\R^{n-1}$ we set
  \begin{eqnarray}\label{3D}
    |D^cu_i|^+(G)=\lim_{h\to\infty}|D^c(1_{\S_h}\,u_i)|(G)\,,\qquad i=1,2\,,
  \end{eqnarray}
  where $\S_h=\{\de_h<v<L_h\}$ for sequences $\de_h\to 0$ and $L_h\to\infty$ as $h\to\infty$ such that $\{v>\de_h\}$ and $\{v<L_h\}$ are sets of finite perimeter. (Notice that $1_{\S_h}\,u_i\in GBV(\R^{n-1})$ for $i=1,2$, so that  $|D^c(1_{\S_h}\,u_i)|$ are well-defined as Borel measures, and the right-hand side of \eqref{3D} makes sense by monotonicity.)
\end{theorem}

\begin{proof}
  By construction and by Theorem \ref{thm tauM b delta}, if we set $v_h=1_{\S_h}v$ and $b_h=1_{\S_h}\,b$, then $v_h\in (BV\cap L^\infty)(\R^{n-1})$ and $b_h\in GBV(\R^{n-1})$ for every $h\in\N$, so that $1_{\S_h}\,u_1=b_h-(v_h/2)\in GBV(\R^{n-1})$, $1_{\S_h}\,u_2=b_h+(v_h/2)\in GBV(\R^{n-1})$, and, by Corollary \ref{lemma W}, we find
  \begin{eqnarray*}\nonumber
  P(E_h;G\times\R)
  &=&\int_{G\cap\{v_h>0\}}\sqrt{1+\Big|\nabla b_h+\frac{\nabla v_h}2\Big|^2}+\sqrt{1+\Big|\nabla b_h-\frac{\nabla v_h}2\Big|^2}\,d\H^{n-1}
  \\\nonumber
  &&+\Big|D^c\Big(b_h+\frac{v_h}2\Big)\Big|\Big(G\cap\{v_h^\wedge>0\}\Big)+\Big|D^c\Big(b_h-\frac{v_h}2\Big)\Big|\Big(G\cap\{v_h^\wedge>0\}\Big)
  \\
  &&+\int_{G\cap(S_{v_h}\cup S_{b_h})}\,\min\Big\{v_h^\vee+v_h^\wedge,\max\{[v_h],2[b_h]\}\Big\}\,d\H^{n-2}\,,
  \end{eqnarray*}
  for every Borel set $G\subset\R^{n-1}$, provided we set $E_h=W[v_h,b_h]$. By taking into account that  $P(E;\S_h^{(1)}\times\R)=P(E_h;\S_h^{(1)}\times\R)$, the above formula gives
  \begin{eqnarray*}\nonumber
  P(E;\S_h^{(1)}\times\R)
  &=&\int_{\S_h}\sqrt{1+\Big|\nabla b+\frac{\nabla v}2\Big|^2}+\sqrt{1+\Big|\nabla b-\frac{\nabla v}2\Big|^2}\,d\H^{n-1}
  \\
  &&+\int_{\S_h^{(1)}\cap(S_v\cup S_{b})}\,\min\Big\{v^\vee+v^\wedge,\max\{[v],2[b]\}\Big\}\,d\H^{n-2}
  \\\nonumber
  &&+\Big|D^c\Big(b_h+\frac{v_h}2\Big)\Big|\Big(\{v^\wedge>0\}\Big)+\Big|D^c\Big(b_h-\frac{v_h}2\Big)\Big|\Big(\{v^\wedge>0\}\Big)\,,
  \end{eqnarray*}
  where we have also taken into account that, for every $h\in\N$,
  \[
  |D^c(b_h\pm(v_h/2))|(\S_h^{(1)})=|D^c(b_h\pm(v_h/2))|(\R^{n-1})=|D^c(b_h\pm(v_h/2))|(\{v^\wedge>0\})\,.
  \]
  By monotonicity, and since $\bigcup_{h\in\N}\S_h^{(1)}=\{v^\wedge>0\}\cap\{v^\vee=\infty\}=_{\H^{n-2}}\{v^\wedge>0\}$ (thanks to \cite[4.5.9(3)]{FedererBOOK} and since $v\in BV(\R^{n-1})$ by Proposition \ref{proposition insieme u1u2}), we find that
  \begin{eqnarray*}\nonumber
  P(E;\{v^\wedge>0\}\times\R)  &=&\int_{\{v>0\}}\sqrt{1+|\nabla u_1|^2}+\sqrt{1+|\nabla u_2|^2}\,d\H^{n-1}
  \\
  &&+\int_{\{v^\wedge>0\}\cap(S_v\cup S_{b})}\,\min\Big\{v^\vee+v^\wedge,\max\{[v],2[b]\}\Big\}\,d\H^{n-2}
  \\
  \nonumber
  &&+|D^cu_1|^+(\{v^\wedge>0\})+|D^cu_2|^+(\{v^\wedge>0\})\,.
  \end{eqnarray*}
  At the same time, by Proposition \ref{corollario v uguale 0}, we have $ P(E;\{v^\wedge=0\}\times\R)=\int_{S_v\cap\{v^\wedge=0\}}v^\vee\,d\H^{n-2}$.
  Adding up the last two identities we complete the proof of the formula for $P(E)$.
\end{proof}

\bibliography{references}
\bibliographystyle{is-alpha}

\end{document}

%% file: basic.pstex_t
\begin{picture}(0,0)%
\includegraphics{basic.eps}%
\end{picture}%
\setlength{\unitlength}{3947sp}%
\begingroup\makeatletter\ifx\SetFigFont\undefined%
\gdef\SetFigFont#1#2#3#4#5{%
  \reset@font\fontsize{#1}{#2pt}%
  \fontfamily{#3}\fontseries{#4}\fontshape{#5}%
  \selectfont}%
\fi\endgroup%
\begin{picture}(4625,2014)(223,-1267)
\put(2029,-467){\makebox(0,0)[lb]{\smash{{\SetFigFont{10}{12.0}{\rmdefault}{\mddefault}{\updefault}{\color[rgb]{0,0,0}$F[v]$}%
}}}}
\put(4591,-757){\makebox(0,0)[lb]{\smash{{\SetFigFont{9}{10.8}{\rmdefault}{\mddefault}{\updefault}{\color[rgb]{0,0,0}$1$}%
}}}}
\put(1943,-765){\makebox(0,0)[lb]{\smash{{\SetFigFont{9}{10.8}{\rmdefault}{\mddefault}{\updefault}{\color[rgb]{0,0,0}$1$}%
}}}}
\put(708,-765){\makebox(0,0)[lb]{\smash{{\SetFigFont{9}{10.8}{\rmdefault}{\mddefault}{\updefault}{\color[rgb]{0,0,0}$0$}%
}}}}
\put(2889,601){\makebox(0,0)[lb]{\smash{{\SetFigFont{10}{12.0}{\rmdefault}{\mddefault}{\updefault}{\color[rgb]{0,0,0}(b)}%
}}}}
\put(238,589){\makebox(0,0)[lb]{\smash{{\SetFigFont{10}{12.0}{\rmdefault}{\mddefault}{\updefault}{\color[rgb]{0,0,0}(a)}%
}}}}
\put(3356,-757){\makebox(0,0)[lb]{\smash{{\SetFigFont{9}{10.8}{\rmdefault}{\mddefault}{\updefault}{\color[rgb]{0,0,0}$0$}%
}}}}
\put(1994,366){\makebox(0,0)[lb]{\smash{{\SetFigFont{10}{12.0}{\rmdefault}{\mddefault}{\updefault}{\color[rgb]{0,0,0}$E$}%
}}}}
\put(4687,370){\makebox(0,0)[lb]{\smash{{\SetFigFont{10}{12.0}{\rmdefault}{\mddefault}{\updefault}{\color[rgb]{0,0,0}$E$}%
}}}}
\put(3863,-743){\makebox(0,0)[lb]{\smash{{\SetFigFont{9}{10.8}{\rmdefault}{\mddefault}{\updefault}{\color[rgb]{0,0,0}$1/2$}%
}}}}
\put(4723,-480){\makebox(0,0)[lb]{\smash{{\SetFigFont{10}{12.0}{\rmdefault}{\mddefault}{\updefault}{\color[rgb]{0,0,0}$F[v]$}%
}}}}
\end{picture}%

%% file: casetta.pstex_t
\begin{picture}(0,0)%
\includegraphics{casetta.eps}%
\end{picture}%
\setlength{\unitlength}{3947sp}%
\begingroup\makeatletter\ifx\SetFigFont\undefined%
\gdef\SetFigFont#1#2#3#4#5{%
  \reset@font\fontsize{#1}{#2pt}%
  \fontfamily{#3}\fontseries{#4}\fontshape{#5}%
  \selectfont}%
\fi\endgroup%
\begin{picture}(3966,1658)(1164,-1284)
\put(4565,-973){\makebox(0,0)[lb]{\smash{{\SetFigFont{9}{10.8}{\rmdefault}{\mddefault}{\updefault}{\color[rgb]{0,0,0}$\{[v]>0\}$}%
}}}}
\put(4755,-396){\makebox(0,0)[lb]{\smash{{\SetFigFont{9}{10.8}{\rmdefault}{\mddefault}{\updefault}{\color[rgb]{0,0,0}$(0,1)^2$}%
}}}}
\put(2475, 44){\makebox(0,0)[lb]{\smash{{\SetFigFont{10}{12.0}{\rmdefault}{\mddefault}{\updefault}{\color[rgb]{0,0,0}$F[v]$}%
}}}}
\end{picture}%

%% file: salsicciotto.pstex_t
\begin{picture}(0,0)%
\includegraphics{salsicciotto.eps}%
\end{picture}%
\setlength{\unitlength}{3947sp}%
\begingroup\makeatletter\ifx\SetFigFont\undefined%
\gdef\SetFigFont#1#2#3#4#5{%
  \reset@font\fontsize{#1}{#2pt}%
  \fontfamily{#3}\fontseries{#4}\fontshape{#5}%
  \selectfont}%
\fi\endgroup%
\begin{picture}(3796,1450)(925,-1154)
\put(4068,-1094){\makebox(0,0)[lb]{\smash{{\SetFigFont{9}{10.8}{\rmdefault}{\mddefault}{\updefault}{\color[rgb]{0,0,0}$\{v=0\}$}%
}}}}
\put(2706,150){\makebox(0,0)[lb]{\smash{{\SetFigFont{10}{12.0}{\rmdefault}{\mddefault}{\updefault}{\color[rgb]{0,0,0}$F[v]$}%
}}}}
\put(4157, 10){\makebox(0,0)[lb]{\smash{{\SetFigFont{9}{10.8}{\rmdefault}{\mddefault}{\updefault}{\color[rgb]{0,0,0}$(0,1)^2$}%
}}}}
\end{picture}%

%% file: disco.pstex_t
\begin{picture}(0,0)%
\includegraphics{disco.eps}%
\end{picture}%
\setlength{\unitlength}{3947sp}%
\begingroup\makeatletter\ifx\SetFigFont\undefined%
\gdef\SetFigFont#1#2#3#4#5{%
  \reset@font\fontsize{#1}{#2pt}%
  \fontfamily{#3}\fontseries{#4}\fontshape{#5}%
  \selectfont}%
\fi\endgroup%
\begin{picture}(3558,1365)(1639,-1678)
\put(2084,-808){\makebox(0,0)[lb]{\smash{{\SetFigFont{10}{12.0}{\rmdefault}{\mddefault}{\updefault}{\color[rgb]{0,0,0}$G_+$}%
}}}}
\put(5182,-1034){\makebox(0,0)[lb]{\smash{{\SetFigFont{10}{12.0}{\rmdefault}{\mddefault}{\updefault}{\color[rgb]{0,0,0}$K'$}%
}}}}
\put(3007,-1038){\makebox(0,0)[lb]{\smash{{\SetFigFont{10}{12.0}{\rmdefault}{\mddefault}{\updefault}{\color[rgb]{0,0,0}$K$}%
}}}}
\put(4867,-459){\makebox(0,0)[lb]{\smash{{\SetFigFont{10}{12.0}{\rmdefault}{\mddefault}{\updefault}{\color[rgb]{0,0,0}$G$}%
}}}}
\put(2702,-459){\makebox(0,0)[lb]{\smash{{\SetFigFont{10}{12.0}{\rmdefault}{\mddefault}{\updefault}{\color[rgb]{0,0,0}$G$}%
}}}}
\put(2013,-1462){\makebox(0,0)[lb]{\smash{{\SetFigFont{10}{12.0}{\rmdefault}{\mddefault}{\updefault}{\color[rgb]{0,0,0}$G_-$}%
}}}}
\end{picture}%

%% file: jumps.pstex_t
\begin{picture}(0,0)%
\includegraphics{jumps.eps}%
\end{picture}%
\setlength{\unitlength}{3947sp}%
\begingroup\makeatletter\ifx\SetFigFont\undefined%
\gdef\SetFigFont#1#2#3#4#5{%
  \reset@font\fontsize{#1}{#2pt}%
  \fontfamily{#3}\fontseries{#4}\fontshape{#5}%
  \selectfont}%
\fi\endgroup%
\begin{picture}(3737,1365)(387,-1458)
\put(4109,-232){\makebox(0,0)[lb]{\smash{{\SetFigFont{9}{10.8}{\rmdefault}{\mddefault}{\updefault}{\color[rgb]{0,0,0}$\frac{[v](z)}2$}%
}}}}
\put(2656,-1398){\makebox(0,0)[lb]{\smash{{\SetFigFont{9}{10.8}{\rmdefault}{\mddefault}{\updefault}{\color[rgb]{0,0,0}$v^\wedge(z)>0$}%
}}}}
\put(653,-1398){\makebox(0,0)[lb]{\smash{{\SetFigFont{9}{10.8}{\rmdefault}{\mddefault}{\updefault}{\color[rgb]{0,0,0}$v^\wedge(z)=0$}%
}}}}
\put(3228,-518){\makebox(0,0)[lb]{\smash{{\SetFigFont{9}{10.8}{\rmdefault}{\mddefault}{\updefault}{\color[rgb]{0,0,0}$z$}%
}}}}
\put(1099,-518){\makebox(0,0)[lb]{\smash{{\SetFigFont{9}{10.8}{\rmdefault}{\mddefault}{\updefault}{\color[rgb]{0,0,0}$z$}%
}}}}
\end{picture}%

%% file: mariasuper.pstex_t
\begin{picture}(0,0)%
\includegraphics{mariasuper.eps}%
\end{picture}%
\setlength{\unitlength}{3947sp}%
\begingroup\makeatletter\ifx\SetFigFont\undefined%
\gdef\SetFigFont#1#2#3#4#5{%
  \reset@font\fontsize{#1}{#2pt}%
  \fontfamily{#3}\fontseries{#4}\fontshape{#5}%
  \selectfont}%
\fi\endgroup%
\begin{picture}(6336,3016)(721,-2478)
\put(6437,-778){\makebox(0,0)[lb]{\smash{{\SetFigFont{9}{10.8}{\rmdefault}{\mddefault}{\updefault}{\color[rgb]{0,0,0}$\frac32$}%
}}}}
\put(3041,-1013){\makebox(0,0)[lb]{\smash{{\SetFigFont{9}{10.8}{\rmdefault}{\mddefault}{\updefault}{\color[rgb]{0,0,0}$\frac32$}%
}}}}
\put(1164,-1003){\makebox(0,0)[lb]{\smash{{\SetFigFont{9}{10.8}{\rmdefault}{\mddefault}{\updefault}{\color[rgb]{0,0,0}$\frac12$}%
}}}}
\put(5444,-554){\makebox(0,0)[lb]{\smash{{\SetFigFont{9}{10.8}{\rmdefault}{\mddefault}{\updefault}{\color[rgb]{0,0,0}$1$}%
}}}}
\put(2102,-501){\makebox(0,0)[lb]{\smash{{\SetFigFont{9}{10.8}{\rmdefault}{\mddefault}{\updefault}{\color[rgb]{0,0,0}$1$}%
}}}}
\put(4570,-778){\makebox(0,0)[lb]{\smash{{\SetFigFont{9}{10.8}{\rmdefault}{\mddefault}{\updefault}{\color[rgb]{0,0,0}$\frac12$}%
}}}}
\put(4113,-977){\makebox(0,0)[lb]{\smash{{\SetFigFont{7}{8.4}{\rmdefault}{\mddefault}{\updefault}{\color[rgb]{0,0,0}$\frac18$}%
}}}}
\put(4363,-995){\makebox(0,0)[lb]{\smash{{\SetFigFont{7}{8.4}{\rmdefault}{\mddefault}{\updefault}{\color[rgb]{0,0,0}$\frac38$}%
}}}}
\put(4822,-995){\makebox(0,0)[lb]{\smash{{\SetFigFont{7}{8.4}{\rmdefault}{\mddefault}{\updefault}{\color[rgb]{0,0,0}$\frac58$}%
}}}}
\put(5090,-987){\makebox(0,0)[lb]{\smash{{\SetFigFont{7}{8.4}{\rmdefault}{\mddefault}{\updefault}{\color[rgb]{0,0,0}$\frac78$}%
}}}}
\put(5946,-960){\makebox(0,0)[lb]{\smash{{\SetFigFont{7}{8.4}{\rmdefault}{\mddefault}{\updefault}{\color[rgb]{0,0,0}$\frac98$}%
}}}}
\put(6249,-995){\makebox(0,0)[lb]{\smash{{\SetFigFont{7}{8.4}{\rmdefault}{\mddefault}{\updefault}{\color[rgb]{0,0,0}$\frac{11}8$}%
}}}}
\put(6655,-960){\makebox(0,0)[lb]{\smash{{\SetFigFont{7}{8.4}{\rmdefault}{\mddefault}{\updefault}{\color[rgb]{0,0,0}$\frac{13}8$}%
}}}}
\put(6941,-987){\makebox(0,0)[lb]{\smash{{\SetFigFont{7}{8.4}{\rmdefault}{\mddefault}{\updefault}{\color[rgb]{0,0,0}$\frac{15}8$}%
}}}}
\put(4226,-899){\makebox(0,0)[lb]{\smash{{\SetFigFont{7}{8.4}{\rmdefault}{\mddefault}{\updefault}{\color[rgb]{0,0,0}$\frac14$}%
}}}}
\put(4944,-899){\makebox(0,0)[lb]{\smash{{\SetFigFont{7}{8.4}{\rmdefault}{\mddefault}{\updefault}{\color[rgb]{0,0,0}$\frac34$}%
}}}}
\put(6798,-891){\makebox(0,0)[lb]{\smash{{\SetFigFont{7}{8.4}{\rmdefault}{\mddefault}{\updefault}{\color[rgb]{0,0,0}$\frac74$}%
}}}}
\put(6090,-891){\makebox(0,0)[lb]{\smash{{\SetFigFont{7}{8.4}{\rmdefault}{\mddefault}{\updefault}{\color[rgb]{0,0,0}$\frac54$}%
}}}}
\end{picture}%

%% file: u1u2.pstex_t
\begin{picture}(0,0)%
\includegraphics{u1u2.eps}%
\end{picture}%
\setlength{\unitlength}{3947sp}%
\begingroup\makeatletter\ifx\SetFigFont\undefined%
\gdef\SetFigFont#1#2#3#4#5{%
  \reset@font\fontsize{#1}{#2pt}%
  \fontfamily{#3}\fontseries{#4}\fontshape{#5}%
  \selectfont}%
\fi\endgroup%
\begin{picture}(2427,1686)(916,-1307)
\put(1933,-427){\makebox(0,0)[lb]{\smash{{\SetFigFont{9}{10.8}{\rmdefault}{\mddefault}{\updefault}{\color[rgb]{0,0,0}$z$}%
}}}}
\put(965,-1198){\makebox(0,0)[lb]{\smash{{\SetFigFont{9}{10.8}{\rmdefault}{\mddefault}{\updefault}{\color[rgb]{0,0,0}$u_1$}%
}}}}
\put(969,195){\makebox(0,0)[lb]{\smash{{\SetFigFont{9}{10.8}{\rmdefault}{\mddefault}{\updefault}{\color[rgb]{0,0,0}$u_2$}%
}}}}
\put(1938,142){\makebox(0,0)[lb]{\smash{{\SetFigFont{8}{9.6}{\rmdefault}{\mddefault}{\updefault}{\color[rgb]{0,0,0}$u_2^\vee(z)$}%
}}}}
\put(1039,-274){\makebox(0,0)[lb]{\smash{{\SetFigFont{10}{12.0}{\rmdefault}{\mddefault}{\updefault}{\color[rgb]{0,0,0}$E$}%
}}}}
\put(1946,-817){\makebox(0,0)[lb]{\smash{{\SetFigFont{8}{9.6}{\rmdefault}{\mddefault}{\updefault}{\color[rgb]{0,0,0}$u_1^\vee(z)$}%
}}}}
\put(1933,-174){\makebox(0,0)[lb]{\smash{{\SetFigFont{8}{9.6}{\rmdefault}{\mddefault}{\updefault}{\color[rgb]{0,0,0}$u_2^\wedge(z)$}%
}}}}
\put(1946,-1119){\makebox(0,0)[lb]{\smash{{\SetFigFont{8}{9.6}{\rmdefault}{\mddefault}{\updefault}{\color[rgb]{0,0,0}$u_1^\wedge(z)$}%
}}}}
\end{picture}%